\documentclass{amsart}
\usepackage[all]{xy}
\usepackage{amsmath,amsfonts,amssymb,amsthm,epsfig,amscd,psfrag,framed,comment}
\usepackage{nicefrac,xspace,tikz}
\usetikzlibrary{arrows}

\usepackage{hyperref}

\newcommand{\puteps}[2][0.5]
{\includegraphics[scale=#1]{#2.eps}}

\newtheorem{Theorem}{Theorem}[section]
 
\newtheorem{Proposition}[Theorem]{Proposition} 
\newtheorem{Lemma}[Theorem]{Lemma}

\newtheorem{Corollary}[Theorem]{Corollary}
\newtheorem{Conjecture}[Theorem]{Conjecture}

\theoremstyle{definition}
\newtheorem{Remark}[Theorem]{Remark}

\newcommand{\smcbub}[1]{
\xybox{%
 (-5,0)*{};
  (5,0)*{};
  (-2,0)*{}="t1";
  (2,0)*{}="t2";
  "t2";"t1" **\crv{(2,3) & (-2,3)}; ?(0)*\dir{<} ?(.95)*\dir{<} ?(.3)*\dir{}+(0,0)*{\bullet}+(1,2)*{\scs {#1}};
  "t2";"t1" **\crv{(2,-3) & (-2,-3)};
}}

\newcommand{\smccbub}[1]{
\xybox{%
 (-5,0)*{};
  (5,0)*{};
  (-2,0)*{}="t1";
  (2,0)*{}="t2";
  "t2";"t1" **\crv{(2,3) & (-2,3)}; ?(.05)*\dir{>} ?(1)*\dir{>}; ?(.3)*\dir{}+(0,0)*{\bullet}+(1,2)*{\scs
  {#1}};
  "t2";"t1" **\crv{(2,-3) & (-2,-3)};
}}

\newcommand{\cbub}[1]{
\xybox{%
 (-5,0)*{};
  (5,0)*{};
  (-2,0)*{}="t1";
  (2,0)*{}="t2";
  "t2";"t1" **\crv{(2,3) & (-2,3)}; ?(0)*\dir{<} ?(.95)*\dir{<};
  "t2";"t1" **\crv{(2,-3) & (-2,-3)};
}}

\newcommand{\ccbub}[1]{
\xybox{%
(-5,0)*{};
  (5,0)*{};
  (-2,0)*{}="t1";
  (2,0)*{}="t2";
  "t2";"t1" **\crv{(2,3) & (-2,3)}; ?(.05)*\dir{>} ?(1)*\dir{>}; 
  "t2";"t1" **\crv{(2,-3) & (-2,-3)};
}}

\newcommand{\Udotcapl}{\;\;
    \vcenter{\xy (2,-2)*{}; (-2,-2)*{} **\crv{(2,2) & (-2,2)}?(1)*\dir{>};
            (2,-3)*{};(-2,3)*{};(1.5,0)*{\bullet}+(-3,2); \endxy} \;\; }

\newcommand{\Udotcupr}{
\xy (-2,2)*{}; (2,2)*{} **\crv{(-2,-2) & (2,-2)}?(1)*\dir{>};
            (2,-3)*{};(-2,3)*{};(1.5,0)*{\bullet} \endxy}

\newcommand{\Ucupr}{
   \xy (-2,2)*{}; (2,2)*{} **\crv{(-2,-2) & (2,-2)}?(1)*\dir{>};
            (2,-3)*{};(-2,3)*{}; \endxy}

\newcommand{\Ucapr}{\;\;
    \vcenter{\xy (-2,-1)*{}; (2,-1)*{} **\crv{(-2,3) & (2,3)}?(1)*\dir{>};
            (2,-3)*{};(-2,3)*{}; \endxy} \;\; }

\newcommand{\Ucapl}{\;\;
    \vcenter{\xy (2,-2)*{}; (-2,-2)*{} **\crv{(2,2) & (-2,2)}?(1)*\dir{>};
            (2,-3)*{};(-2,3)*{}; \endxy} \;\; }

\newcommand{\sUup}{
    \xy {\ar (0,-2)*{};(0,2)*{} };(1.5,0)*{};(-1.5,0)*{};\endxy}

\newcommand{\sUdown}{
    \xy {\ar (0,2)*{};(0,-2)*{} };(1.5,0)*{};(-1.5,0)*{};\endxy}

\newcommand{\sUupdot}{
   \xy {\ar (0,-2)*{};(0,2)*{} };(0,0)*{\scs \bullet};(1.5,0)*{};(-1.5,0)*{};\endxy}

\newcommand{\sUdowndot}{
   \xy {\ar (0,2)*{};(0,-2)*{} };(0,0)*{\scs \bullet};(1.5,0)*{};(-1.5,0)*{};\endxy}

\newcommand{\scs}{\scriptstyle}
\newcommand{\qbin}[2]{
\left[
 \begin{array}{c}
 #1 \\
 #2 \\
 \end{array}
 \right]
}
\newcommand{\qbins}[2]{
\left[
 \begin{array}{c}
 \scs #1 \\
 \scs #2 \\
 \end{array}
 \right]
}
\def\th{\theta}
\def\sC{{\mathcal{C}}}
\def\bD{\mathbb{D}}

\def\Sym{{\mbox{Sym}}}
\def\dmod{{\mathrm{-mod}}}

\def\bk{{\underline{k}}}
\def\i{{\underline{i}}}
\def\j{{\underline{j}}}
\def\ux{{\underline{x}}}
\def\0{{\underline{0}}}
\def\u2{{\underline{2}}}
\def\um{{\underline{m}}}

\def\rA{{\mathrm{A}}}
\def\M{\mathfrak{M}}
\def\hK{\hat{K}}
\def\K{\mathcal{K}}

\def\A{{\sf{A}}}
\def\B{{\sf{B}}}
\def\bA{{\mathbb{A}}}
\def\fC{{\sf{C}}}
\def\wt{{\sf{wt}}}
\def\Kom{{\sf{Kom}}}
\def\1{{\mathbf{1}}}
\def\P{{\sf{P}}}
\def\bP{{\mathbb{P}}}

\def\k{{\mathbb C}}
\def\id{{\mathrm{id}}}
\def\l{{\lambda}}

\def\R{{\sf{R}}}
\def\sl{\mathfrak{sl}}

\def\g{\mathfrak{g}}

\newcommand{\Gr}{\mathrm{Gr}}
\newcommand{\ul}{{\underline{\lambda}}}
\newcommand{\umu}{{\underline{\mu}}}
\newcommand{\Z}{\mathbb{Z}}

\newcommand{\Cone}{\operatorname{Cone}}

\newcommand{\D}{\mathcal{D}}
\def\O{{\mathcal O}}
\def\e{{\underline{e}}}
\def\sA{{\mathcal{A}}}
\def\sE{{\mathcal{E}}}

\def\sF{{\mathcal{F}}}
\def\sL{{\mathcal{L}}}

\def\sH{{\mathcal{H}}}

\def\sP{{\mathcal{P}}}
\def\sQ{{\mathcal{Q}}}
\def\bG{\mathbb{G}}

\def\t{{T}}
\def\I{{\sf I}}

\def\v{{\underline{v}}}
\def\w{{\underline{w}}}
\def\N{\mathbb N}
\def\tY{\tilde{Y}}
\newcommand{\E}{\mathsf{E}}
\newcommand{\F}{\mathsf{F}}

\newcommand{\G}{\mathbb{G}}

\newcommand{\U}{\mathsf{U}}
\newcommand{\T}{\mathsf{T}}
\newcommand{\la}{\langle}
\newcommand{\ra}{\rangle}

\newcommand{\Hom}{\operatorname{Hom}}
\newcommand{\End}{\operatorname{End}}

\begin{document}
\setcounter{tocdepth}{1}

\title[Clasp technology to knot homology via the affine Grassmannian]{Clasp technology to knot homology via the affine Grassmannian}

\author{Sabin Cautis}
\email{cautis@math.ubc.ca}
\address{Department of Mathematics\\ University of British Columbia \\ Vancouver, Canada}

\begin{abstract}
We categorify all the Reshetikhin-Turaev tangle invariants of type A. Our main tool is a categorification of the generalized Jones-Wenzl projectors (a.k.a. clasps) as infinite twists. Applying this to certain convolution product varieties on the affine Grassmannian we extend our earlier work with Kamnitzer \cite{CK1,CK2} from standard to arbitrary representations. 
\end{abstract}

\maketitle
\tableofcontents
\section{Introduction}

We give a method, based on the higher representation theory of $U_q(\sl_\infty)$, for categorifying all the Reshetikhin-Turaev tangle invariants associated to $\sl_m$. In particular, we show how to define homological tangle invariants using any categorification of the $U_q(\sl_\infty)$-module $\Lambda_q^{m \infty}(\k^m \otimes \k^{2 \infty})$. 

One example of such a categorification uses the affine Grassmannian of ${\rm PGL}_m$. Alternatively, one could use Nakajima quiver varieties to categorify the sub-module $\Lambda_q^{\infty}(\k^{2\infty})^{\otimes m}$. These two approaches are closely related and give the same homological knot invariants.

The main tool we use is skew Howe duality, building on work from \cite{CKL1}. The key technical construction is the categorification of the Jones-Wenzl projectors and their generalizations called clasps.

\subsection{Reshetikhin-Turaev invariants and clasps}\label{sec:RTinv}
  
Let $\g'$ be a complex semisimple Lie algebra and denote by $U_q(\g')$ the corresponding quantum group. Consider a tangle $T$ whose strands are labeled by dominant weights of $\g'$ so that the strands at the bottom and top are labeled by $\ul = (\l_1, \dots, \l_n)$ and $\umu = (\mu_1, \dots, \mu_{n'})$ respectively. 

Following Reshetikhin-Turaev \cite{RT} one can associate to $T$ a map of $U_q(\g')$-modules
$$\psi(T): V_{\ul} = V_{\l_1} \otimes \dots \otimes V_{\l_n} \longrightarrow V_{\mu_1} \otimes \dots \otimes V_{\mu_{n'}} = V_{\umu}$$
where $V_{\l}$ denotes the irreducible $U_q(\g')$-module with highest weight $\l$. This map is an invariant of oriented tangles. In particular, if $T=K$ is a link then $\psi(K)$ is an endomorphism of the trivial module and hence given as multiplication by some $\psi(K)(1) \in \k(q)$. If $\g'=\sl_2$ and all strands are labeled by the standard representation then $\psi(K)(1)$ is just the Jones polynomial of $K$. 

If $\g' = \sl_m$ then one way to define $\psi(T)$ uses the following data. 
\begin{enumerate}
\item Maps $\psi(T)$ where the strands of $T$ are labeled by fundamental weights $\Lambda_1, \dots, \Lambda_{m-1}$. 
\item Idempotents $P \1_\i$ given as the composition 
$$V_{\Lambda_{i_1}} \otimes \dots \otimes V_{\Lambda_{i_N}} \xrightarrow{\pi} V_{\sum_k \Lambda_{i_k}} \xrightarrow{\iota} V_{\Lambda_{i_1}} \otimes \dots \otimes V_{\Lambda_{i_N}}$$
where $\pi$ and $\iota$ are the natural projection and inclusion maps. 
\end{enumerate}
For example, to compute the invariant of the unknot labeled by $V_{2 \Lambda_1}$ one calculates the composition 
$$\k \rightarrow (V_{\Lambda_1}^\vee \otimes V_{\Lambda_1}^\vee) \otimes (V_{\Lambda_1} \otimes V_{\Lambda_1}) \xrightarrow{(II)(P)} (V_{\Lambda_1}^\vee \otimes V_{\Lambda_1}^\vee) \otimes (V_{\Lambda_1} \otimes V_{\Lambda_1}) \rightarrow \k$$
where the leftmost (resp. rightmost) map corresponds to a double cup (resp. double cap). 

\begin{Remark}
When $m=2$ the maps $P$ are the standard {\em Jones-Wenzl projectors}. These can be described recursively in terms of caps and cups. When $m=3$ the maps $P$ were studied by Kuperberg \cite{Kup} and also given a recursive definition in terms of webs. Kuperberg called these idempotents {\em clasps}. We adopt this terminology and call all idempotents $P$ clasps. 
\end{Remark}

One way to define the maps in (i) is using skew Howe duality (see section \ref{sec:Howe}). More precisely, fix $m$ and consider the vector space $\Lambda^{mN}(\k^m \otimes \k^{2N})$ equipped with commuting actions of $U(\sl_m)$ and $U(\sl_N)$. As a $U(\sl_N)$-module it breaks up into weight spaces of the form $V_{\Lambda_{i_1}} \otimes \dots \otimes V_{\Lambda_{i_N}}$ where $E_k, F_k \in U(\sl_\infty)$ are maps 
\begin{equation}\label{eq:EFmaps}
V_{\Lambda_{i_1}} \otimes \dots \otimes V_{\Lambda_{i_k}} \otimes V_{\Lambda_{i_{k+1}}} \otimes \dots \otimes V_{\Lambda_{i_N}} \overset{E_k}{\underset{F_k}{\rightleftarrows}} V_{\Lambda_{i_1}} \otimes \dots \otimes V_{\Lambda_{i_k-1}} \otimes V_{\Lambda_{i_{k+1}+1}} \otimes \dots \otimes V_{\Lambda_{i_N}}.
\end{equation}
The generators $s_k$ of the Weyl group of $\sl_N$ can be lifted from $U(\sl_N)$ to $U_q(\sl_N)$ to obtain maps
\begin{equation}\label{eq:intro1}
\t_k: V_{\Lambda_{i_1}} \otimes \dots \otimes V_{\Lambda_{i_k}} \otimes V_{\Lambda_{i_{k+1}}} \otimes \dots \otimes V_{\Lambda_{i_N}} \longrightarrow V_{\Lambda_{i_1}} \otimes \dots \otimes V_{\Lambda_{i_{k+1}}} \otimes V_{\Lambda_{i_{k}}} \otimes \dots \otimes V_{\Lambda_{i_N}}
\end{equation}
which generate an action of the braid group on the weight spaces of $\Lambda_q^{N}(\k^m \otimes \k^{2N})$. In \cite{CKL1} we showed that this action agrees with the braid group action defined by Reshetikhin-Turaev using the R-matrix associated with $\sl_m$. Using the $E$'s and $F$'s (see section \ref{sec:tanglesfund}) one can also define caps and cups and thus recover all the tangle maps $\psi(T)$ from (i). 

To define the clasps $P \1_\i$ in (ii) we will use the braid group action from (i). More specifically, denote by $\t_\omega \1_\i$ the full twist of $n$ strands (see equation (\ref{eq:Tomega})). Then one can show that the limit $\lim_{\ell \rightarrow \infty} \t_\omega^{2 \ell} \1_\i$ exists and converges to give the clasp $P \1_\i$. 

Notice that, in the constructions above, the more strands in our tangle the larger the $N$ we need to use. In order to work with all tangles in a uniform manner we will let $N \rightarrow \infty$ and pass to the $U_q(\sl_\infty)$-module $\Lambda_q^{m \infty}(\k^m \otimes \k^{2 \infty})$.

\begin{Remark}\label{rem:only}
In the end, to recover the Reshetikhin-Turaev invariants of $\g'=\sl_m$, we only use the data encoded in the $U_q(\sl_\infty)$-module $\Lambda_q^{m \infty}(\k^m \otimes \k^{2 \infty})$.
\end{Remark}

\subsection{Categorification}\label{sec:categorification}

The first examples of homological knot invariants are due to Khovanov \cite{K1,K2}. He considers the case where $\g'=\sl_2$ and all the strands are labeled by the standard representation. In subsequent work \cite{KR}, Khovanov and Rozansky consider $\g' = \sl_m$ where all strands are again labeled by the standard representation (or its dual). Their construction uses categories of matrix factorizations. 

In \cite{CK1,CK2} we also considered $\g'=\sl_m$ where all strands are labeled by the standard representation (or its dual). The categories used, inspired by the geometric Satake correspondence, are derived categories of coherent sheaves on certain varieties $Y(i_1,i_2,\dots,i_N)$ (see section \ref{sec:varieties}) where each $i_k$ is either $1$ or $m-1$. The current paper uses tools from higher representation theory to generalize \cite{CK1,CK2} to tangles labeled by arbitrary representations of $\g'=\sl_m$. 

More precisely, we will categorify the skew Howe duality construction from section \ref{sec:RTinv}. The categorical analogue of a $U_q(\sl_N)$ action that we use is an $(\sl_N, \th)$ action (section \ref{sec:catgactions}). It was introduced in \cite{C2} where it was shown that such an action induces a categorical action of $\sl_N$ in the sense of \cite{KL3,Rou1}. The advantage of an $(\sl_N,\th)$ action is that it is simpler and easier to check in practice. 

The first step is to lift the $U_q(\sl_N)$-module $\Lambda_q^{mN}(\k^m \otimes \k^{2N})$ to a 2-category $\K$ equipped with an $(\sl_N,\th)$ action. Roughly, this means that the nonzero objects of $\K$ are in bijection with nonzero weight spaces of $\Lambda_q^{mN}(\k^m \otimes \k^{2N})$. In other words, the objects of $\K$ are indexed by sequences $\i = (i_1, \dots, i_N)$ which correspond to the weight space $V_{\Lambda_{i_1}} \otimes \dots \otimes V_{\Lambda_{i_N}}$. 

The fact that $\K$ is equipped with a $(\sl_N,\th)$ means that inside $\K$ we have 1-morphisms 
\begin{align*}
\xymatrix{
(i_1, \dots, i_k, i_{k+1}, \dots, i_N) \ar@/^/[rr]|-{\E_k} && \ar@/^/[ll]|-{\F_k} (i_1, \dots, i_k-1, i_{k+1}+1, \dots, i_N) }
\end{align*}
lifting the maps in (\ref{eq:EFmaps}). Following \cite{CKL3,CK3} one can use these 1-morphisms to define complexes $\T_k \in \Kom(\K)$ which lift the maps in (\ref{eq:intro1}) and give rise to a braid group action (here $\Kom(\K)$ denotes the homotopy category of $\K$). These complexes allow us to lift the maps $\psi(T)$ associated to a tangle whose strands are labeled by fundamental representations. 

Most of the work in this paper is to explain how to also lift the clasps $P$. This is done by showing that the limit $\lim_{\ell \rightarrow \infty} \T_\omega^{2 \ell}$ converges to an idempotent $\P^- \in \Kom^-_*(\K)$ (Theorem \ref{thm:main1}) and that this idempotent categorifies the clasp $P$ (Theorem \ref{thm:main2}). Putting all this together gives us a $\Z^2$-graded link invariant $\sH^{i,j}_-(K)$ (Theorem \ref{thm:main3}).

\subsection{The affine Grassmannian}\label{sec:grass}

The 2-category $\K = \K_{\Gr,m}$ that we use in the constructions above is obtained from the affine Grassmannian of ${\rm PGL}_m$. More precisely, the objects in $\K_{\Gr,m}$ are derived categories of coherent sheaves $D(Y(\i))$ on certain convolution product varieties $Y(\i)$ (see section \ref{sec:varieties}). The 1-morphisms are then kernels and the 2-morphisms are morphisms between kernels. 

In Theorem \ref{thm:main4} we show that $\K_{\Gr,m}$ carries a $(\sl_\infty,\th)$ action. It follows that, using affine Grassmannians of type A, one can categorify all the Reshetikhin-Turaev invariants of type A. The resulting homology of a knot labeled by fundamental representations is finite dimensional whereas for non-fundamental representations the homology is infinite dimensional. 

\subsection{Rigidity}\label{sec:rigid}

The input for our categorification of $\g'=\sl_m$ link invariants is any 2-category $\K$ equipped with an $(\sl_\infty, \th)$ action which lifts the $U_q(\sl_\infty)$-module $\Lambda_q^{m \infty}(\k^m \otimes \k^{2 \infty})$ ({\it c.f.} Remark \ref{rem:only}). 

At first it seems that the resulting homology depends on the 2-category $\K$. However, in \cite{C2} we showed that any $(\sl_N,\th)$ action can be equipped with an action of the quiver Hecke algebras (KLR algebras). Together with \cite{CLa} this means that any $(\sl_N,\th)$ action can be lifted to a 2-representation in the sense of \cite{KL3}. 

The extra structure in such a 2-representation allows us to take any complex of 1-morphisms and simplify it. In particular, it gives us a uniform way to compute the homology of a link which is independent of the choice of $\K$. 

This observation implies a type of rigidity for homological link invariants. It means that two homological link invariants which are obtained (using the skew Howe duality approach) from a 2-category $\K$ must be isomorphic. Various link homologies such as those obtained using derived categories of coherent sheaves \cite{CK1,CK2}, category $\O$ \cite{MS,Su}, matrix factorizations \cite{Wu,Y} and foams \cite{MSV,LQR,QR} are known (or at least expected to) fit in this picture. The argument above shows they are isomorphic. 

\subsection{Lee-like deformations}

Lee defined in \cite{Lee} a modification of Khovanov homology (now called Lee homology) and showed that there exists a spectral sequence whose $E_2$-term is Khovanov homology and which converges to Lee homology. Subsequently, Rasmussen \cite{Ras} used this to define a concordance invariant and give a combinatorial proof of the Milnor conjecture. 

Using the 2-category $\K_{\Gr,2}$ it is possible to recover the analogue of Lee homology as follows. The objects in this category are varieties which carry a natural action of ${\rm PGL}_2$. Moreover, all morphisms and 2-morphisms we define are ${\rm PGL}_2$-equivariant so we can define a ${\rm PGL}_2$-equivariant version of this category. 

The resulting homology associates to a link a complex of modules over $\End_{{\rm PGL}_2}(\k) \cong \k[x]$ where $\deg(x)=4$. If one specializes $x=0$ then this recovers Khovanov homology. If one specializes to $x \ne 0$ then this gives the analogue of Lee homology. Moreover, the natural filtration on the complex induced by $\k[x] \supset x \k[x] \supset x^2 \k[x] \supset \dots$ means that there is a spectral sequence starting at the former and converging to the latter.

This approach generalizes to arbitrary $m$ by considering the ${\rm PGL}_m$-equivariant version of $\K_{\Gr,m}$. The resulting link homology is a complex of free modules over $\End_{{\rm PGL}_m}(\k) \cong \k[x_2, \dots, x_m]$ where $\deg(x_i)=2i$. This defines a family of homologies. Specializing to $x_2=\dots=x_m=0$ one obtains the $\sl_m$ homology discussed in this paper while specializing to generic values of $x_2,\dots,x_m$ gives the $\sl_m$ analogue of Lee's homology. For the same reasons as above, there is again a spectral sequence starting with the former and converging to the latter.

\subsection{Related work on categorical clasps}\label{sec:catproj}

Rozansky \cite{Roz} categorified the Jones-Wenzl projectors within Bar-Natan's graphical formulation of Khovanov homology. In this setup the 1-morphisms are tangles and 2-morphisms are cobordisms modulo certain relations. He constructs the Jones-Wenzl projectors as infinite braids (as far as we know this is the first place that such an infinite braid construction appears). 

Cooper and Krushkal \cite{CoK} independently categorified the Jones-Wenzl projectors within Bar-Natan's setup using a recursive definition of the projectors. In the case of two strands, they describe the projector explicitly \cite[Sect. 4.1]{CoK} as a complex which is very similar to that in (\ref{eq:tempcpx}). Indeed, if we depict $\F_2$ as a cap and $\E_2$ as a cup then (\ref{eq:tempcpx}) can be identified with the (dual) of their complex. 

Another categorification of Jones-Wenzl projectors appears in \cite{FSS}. In this case the projectors are given by explicit bimodules. More recently, Rose \cite{Ros} works within the Morrison-Nieh formulation of $\g' = \sl_3$ knot homology to categorify $\sl_3$ clasps. Finally, Webster \cite{W1,W2} describes a categorification of $\g' = \sl_m$ Reshetikhin-Turaev invariants though his approach is quite different from ours and does not treat the clasps separately. It would be interesting to relate his construction to ours.

\subsection{Further remarks}

Although, for convenience, we work over $\k$ it is possible to work over $\Z$ (we illustrate with some examples in section \ref{sec:example}). This produces homological knot invariants defined over $\Z$. The existence of invariants over $\Z$ is not clear even in simpler cases like Khovanov-Rozansky homology. 

To work over $\Z$, we first lift the $(\sl_\infty, \th)$ action to a 2-representation in the sense of \cite{KL3} (see section \ref{sec:rigid}). Then we have to show that all the $\sl_\infty$ relations we proved among 1-morphisms hold over $\Z$ (for example, those from Lemma \ref{lem:rels1}). Such relations fall into two types: the $\sl_2$ relations (involving a single node) and the $\sl_3$ relations (involving two adjacent nodes). The former were checked in \cite{KLMS} and the latter in \cite{St}. 

\subsection{Acknowledgements}

I would like to thank Pramod Achar, Joel Kamnitzer, Mikhail Khovanov, Aaron Lauda, Anthony Licata, Jacob Rasmussen, Raphael Rouquier, Lev Rozansky, Noah Snyder and Joshua Sussan for helpful discussions. Lauda and Rasmussen also corrected and helped with the calculations in section \ref{sec:example}. Research was supported by NSF grant DMS-1101439 and the Alfred P. Sloan foundation. 

\section{Statements of results}

\subsection{Lie algebras and braid groups}

For convenience the base field will always be $\k$. In the remainder of the paper we take $\g = \sl_n$ or $\g = \sl_\infty$. This means that the vertex set $I$ of the Dynkin diagram of $\g$ is indexed by $\{1, \dots, n-1\}$ if $\g = \sl_n$ or by $\Z$ if $\g = \sl_\infty$. We fix the following data:
\begin{enumerate}
\item a free $\Z$-module $X$  (the weight lattice),
\item for $i \in I$ an element $\alpha_i \in X$ (simple roots),
\item for $i \in I$ an element $\Lambda_i \in X$ (fundamental weights),
\item a symmetric non-degenerate bilinear form $\la \cdot,\cdot \ra$ on $X$.
\end{enumerate}
These data should satisfy:
\begin{enumerate}
\item the set $\{\alpha_i\}_{i \in I}$ is linearly independent,
\item for all $i,j \in I$ we have $\la \alpha_i, \alpha_j \ra = \begin{cases} 2 &\text{ if } i=j \\ -1 &\text{ if } |i-j|=1 \\ 0 &\text{ if } |i-j| > 1 \end{cases}$, 
\item $\la \Lambda_i, \alpha_j \ra = \delta_{i,j}$ for all $i, j \in I$. 
\end{enumerate}
We will often abbreviate $\la \l, \alpha_i \ra = \l_i$ and $\la \alpha_i, \alpha_j \ra = \la i,j \ra$. The root lattice will be denoted $Y$ and $Y_\k := Y \otimes_\Z \k$.

If $V$ is a representation then $V(\mu)$ denotes the weight space corresponding to $\mu \in X$. If $\l \in X$ is a dominant weight then $V_\l$ denotes the irreducible representation with highest weight $\l$. More generally, for a sequence $\ul = (\l_1, \dots, \l_k)$, where each $\l_i \in X$ is dominant, $V_\ul := V_{\l_1} \otimes \dots \otimes V_{\l_k}$. 

The Weyl group of $\g$ has generators $s_i$ for $i \in I$ and relations $s_i^2 = 1$ and 
\begin{align*}
s_i s_j s_i = s_j s_i s_j & \text{ if } |i-j|=1 \\
s_i s_j = s_j s_i & \text{ if } |i-j| > 1.
\end{align*}
This group acts on $X$ via $s_i \cdot \l := \l - \l_i \alpha_i$. 

The braid group of $\g$ has generators $\sigma_i$ for $i \in I$ and relations
\begin{align*}
\sigma_i \sigma_j \sigma_i = \sigma_j \sigma_i \sigma_j & \text{ if } |i-j|=1  \\
\sigma_i \sigma_j = \sigma_j \sigma_i & \text{ if } |i-j|>1.
\end{align*}

\subsection{$(\g,\th)$ actions}\label{sec:catgactions}

A $(\g,\th)$ action consists of a target graded, additive, $\k$-linear idempotent complete 2-category $\K$ where the objects (0-morphisms) are indexed by $\l \in X$ and equipped with
\begin{enumerate}
\item 1-morphisms: $\E_i \1_\l = \1_{\l+\alpha_i} \E_i$ and $\F_i \1_{\l+\alpha_i} = \1_\l \F_i$ where $\1_\l$ is the identity 1-morphism of $\l$.
\item 2-morphisms: for each $\l \in X$, a linear map $Y_\k \rightarrow \End^2(\1_\l)$.
\end{enumerate}
On this data we impose the following conditions.
\begin{enumerate}
\item \label{co:hom1} $\Hom(\1_\l, \1_\l \la l \ra)$ is zero if $l < 0$ and one-dimensional if $l=0$ and $\1_\l \ne 0$. Moreover, the space of maps between any two 1-morphisms is finite dimensional.
\item \label{co:adj} $\E_i$ and $\F_i$ are left and right adjoints of each other up to specified shifts. More precisely
\begin{enumerate}
\item $(\E_i \1_\l)_R \cong \1_\l \F_i \la \l_i + 1 \ra$
\item $(\E_i \1_\l)_L \cong \1_\l \F_i \la - \l_i -1 \ra$.
\end{enumerate}

\item \label{co:EF} We have
\begin{align*}
\E_i \F_i \1_\l &\cong \F_i \E_i \1_{\l} \bigoplus_{[\l_i]} \1_\l \ \ \text{ if } \l_i \ge 0 \\
\F_i \E_i \1_{\l} &\cong \E_i \F_i \1_{\l} \bigoplus_{[-\l_i]} \1_\l \ \ \text{ if } \l_i \le 0
\end{align*}

\item \label{co:EiFj} If $i \ne j \in I$ then $\F_j \E_i \1_\l \cong \E_i \F_j \1_\l$.

\item \label{co:theta} If $\l_i \ge 0$ then map $(I \th I) \in \End^2(\E_i \1_\l \F_i)$ induces an isomorphism between $\l_i+1$ (resp. zero) of the $\l_i+2$ summands $\1_{\l+\alpha_i}$ when $\la \th, \alpha_i \ra \ne 0$ (resp. $\la \th, \alpha_i \ra = 0$). If $\l_i \le 0$ then the analogous result holds for $(I \th I) \in \End^2(\F_i \1_\l \E_i)$.

\item \label{co:vanish1} If $\alpha = \alpha_i$ or $\alpha = \alpha_i + \alpha_j$ for some $i,j \in I$ with $\la i,j \ra = -1$ then $\1_{\l+r \alpha} = 0$ for $r \gg 0$ or $r \ll 0$.

\item \label{co:new} Suppose $i \ne j \in I$ and $\l \in X$. If $\1_{\l+\alpha_i}$ and $\1_{\l+\alpha_j}$ are nonzero then $\1_{\l}$ and $\1_{\l+\alpha_i+\alpha_j}$ are also nonzero.
\end{enumerate}

The main result of \cite{C2} states that a $(\g,\th)$ action carries an action of the quiver Hecke algebras and, in particular, the affine nilHecke relations. These affine nilHecke relations give us the divided powers $\E_i^{(r)}$ or $\F_i^{(r)}$. These 1-morphisms satisfy 
\begin{enumerate}
\item $(\E^{(r)}_i \1_\l)_R \cong \1_\l \F^{(r)}_i \la r(\l_i+r) \ra$
\item $(\E^{(r)}_i \1_\l)_L \cong \1_\l \F^{(r)}_i \la -r(\l_i+r) \ra$.
\end{enumerate}

We will denote by $\Kom^-(\K)$ the (bounded above) homotopy category of $\K$. Here the objects are the same as those of $\K$, the 1-morphisms are bounded above complexes of 1-morphisms in $\K$ (up to homotopy) and the 2-morphisms are maps of complexes. Notice that these are now $\Z \oplus \Z$-graded additive categories since we have the old grading $\la \cdot \ra$ and the new cohomological grading which we denote $[\cdot]$. 

In section \ref{sec:qadic} we define a certain smaller subcategory $\Kom^-_*(\K) \subset \Kom^-(\K)$ which comes equipped with a map $p: K(\Kom^-_*(\K)) \rightarrow \hK(\K)$. Here $K(\cdot)$ denotes the Grothendieck group of a category and $\hK(\cdot)$ its completion (see section \ref{sec:hptycats}). 

Inside $\Kom^-_*(\K)$ one can define the Rickard complexes 
\begin{align}
\label{eq:cpx'} \T_i \1_\l &:= \left[ \dots \rightarrow \E_i^{(-\l_i+s)} \F_i^{(s)} \la -s \ra \rightarrow \E_i^{(-\l_i+s-1)} \F_i^{(s-1)} \la -(s-1) \ra \rightarrow \dots \rightarrow \E_i^{(-\l_i)} \right] \1_\l \\
\label{eq:cpx} \T_i \1_\l &:= \left[ \dots \rightarrow \F_i^{(\l_i+s)} \E_i^{(s)} \la -s \ra \rightarrow \F_i^{(\l_i+s-1)} \E_i^{(s-1)} \la -(s-1) \ra \rightarrow \dots \rightarrow \F_i^{(\l_i)} \right] \1_\l 
\end{align}
depending on whether $\l_i \le 0$ or $\l_i \ge 0$ respectively. Notice that these complexes are finite since $\E_i^{(s)} \1_\l = 0$ and $\F_i^{(s)} \1_\l = 0$ for $s \gg 0$ (because $\1_{\l \pm s \alpha_i} = 0$ for $s \gg 0$). In \cite{CK3} we showed that these $\T_i$ satisfy the braid relations. 

\begin{Remark} These expressions for $\T_i \1_\l$ categorify Lusztig's definition from \cite[Sect. 5.2.1]{Lu}. There he defines $\t_i \1_\l$ as 
\begin{equation}\label{eq:cpxdecat2}
\sum_{a,b,c \ge 0, a-b+c = \l_i} (-1)^b q^{ac-b} E_i^{(a)} F_i^{(b)} E_i^{(c)}.
\end{equation}
The summation in (\ref{eq:cpxdecat2}) over terms with $c=0$ agrees with the decategorification of (\ref{eq:cpx'}) and (\ref{eq:cpx}) (up to a factor of $-q$). On the other hand, if $c \ne 0$ then the sum in (\ref{eq:cpxdecat2}) is actually zero. So (\ref{eq:cpxdecat2}) agrees with (\ref{eq:cpx'}) and (\ref{eq:cpx}). 
\end{Remark}

\subsection{Statements of results} $ $ \\
\noindent{\bf Result \#1.} Suppose $\g = \sl_n$. Associated to the longest element $\omega$ of the Weyl group of $\g$ we have the complex $\T_\omega \1_\l \in \Kom^-_*(\K)$ defined by
\begin{align}\label{eq:Tomega}
\T_\omega \1_\l
&\cong (\T_{n-1})(\T_{n-2}\T_{n-1}) \dots (\T_2 \dots \T_{n-1})(\T_1 \dots \T_{n-1}) \1_\l \\
&\cong (\T_{n-1}\dots \T_1)(\T_{n-1} \dots \T_2) \dots (\T_{n-1} \T_{n-2})(\T_{n-1}) \1_\l.
\end{align}
For example if $n=3$ then $I = \{1,2\}$ and $\T_\omega \1_\l = \T_1 \T_2 \T_1 \1_\l$. Note that $\T_\omega^2 \1_\l = \1_\l \T_\omega^2$. Our first result is that the limit $\lim_{\ell \rightarrow \infty} \T_\omega^{2 \ell}$ exists. 

\begin{Theorem}\label{thm:main1}
Given a 2-category $\K$ equipped with an $(\sl_n, \th)$ action there exists a 2-morphism $\1_\l \rightarrow \T_\omega^2 \1_\l$ so that the limit $\lim_{\ell \rightarrow \infty} \T_\omega^{2 \ell} \1_\l$ converges to a complex $\P^- \1_\l = \T_\omega^\infty \1_\l \in \Kom^-_*(\K)$. Moreover, this element is idempotent, meaning that $\P^- \cdot \P^- \cong \P^- $.
\end{Theorem}

\noindent {\bf Result \#2.} Suppose $\K$ is a 2-category equipped with an $(\sl_n,\th)$ action. Moreover, let us assume that the objects $\l$ of $\K$ are additive categorifies so that on Grothendieck groups we get the $U_q(\sl_n)$-module $\Lambda_q^{N}(\k^m \otimes \k^n)$ for some $m,N$. The nonzero weight spaces in this case are parametrized by sequences $\i = (i_1, \dots, i_n)$ such that $0 \le i_1, \dots, i_n \le m$ and $\sum_k i_k = N$. More precisely, the weight space $V(\i)$ is isomorphic to 
$$\Lambda_q^{\i}(\k^m) := \Lambda_q^{i_1}(\k^m) \otimes \dots \otimes \Lambda_q^{i_n}(\k^m).$$ 
The composition $\Lambda_q^{\i}(\k^m) \xrightarrow{\pi} V_{\sum_k \Lambda_{i_k}} \xrightarrow{\iota} \Lambda_q^{\i}(\k^m)$ (where $\pi$ and $\iota$ are the projection and inclusion maps) defines an idempotent $P \1_\i := (\iota \circ \pi) \1_\i$. This idempotent is a {\em generalized Jones-Wenzl projector} or, following the terminology in \cite{Kup}, a {\em clasp}. 

\begin{Theorem}\label{thm:main2} 
The functor $\P^- \1_{\i} \in \Kom^-_*(\K)$ categorifies the clasp $P \1_\i$. 
\end{Theorem}

\noindent{\bf Result \#3.} Fix $m$. The $U_q(\sl_\infty)$-module $\Lambda_q^{m \infty}(\k^m \otimes \k^{2\infty})$ is defined as a limit as $N \rightarrow \infty$ of the $U_q(\sl_{2N})$-modules $\Lambda_q^{mN}(\k^m \otimes \k^{2N})$. The nonzero weight spaces of $\Lambda_q^{m \infty}(\k^m \otimes \k^{2\infty})$ are parametrized by sequences $\i = (\dots, i_k,i_{k+1}, \dots)$ where $0 \le i_k \le m$ and $i_k=0$ for $k \ll 0$, $i_k = m$ if $k \gg 0$ (see section \ref{sec:weightspaces}). The weight space indexed by $\i$ is isomorphic to 
$$\Lambda_q^{\i}(\k^m) := \dots \otimes \Lambda_q^{i_k}(\k^m) \otimes \Lambda_q^{i_{k+1}}(\k^m) \otimes \dots$$ 

\begin{Theorem}\label{thm:main3}
Suppose $\K$ is a 2-category equipped with an $(\sl_\infty,\th)$ action and such that the nonzero weight spaces of $\K$ are the same as those of $\Lambda_q^{m \infty}(\k^m \otimes \k^{2\infty})$. From this data one can construct a tangle invariant. In particular, this gives a homological link invariant which categorifies the Reshetikhin-Turaev invariants of links labeled by arbitrary representations of $\sl_m$. 
\end{Theorem}

This construction works as follows. To a positive crossing exchanging strands $k$ and $k+1$ one associates (up to a shift) the functor $\T_k \1_\i$, to a cap and cup the functors
\begin{align*}
\E_k^{(i_k)} :& (\dots, i_k, m-i_k, \dots) \rightarrow (\dots, 0,m, \dots) \\
\F_k^{(i_k)} :& (\dots, 0, m, \dots) \rightarrow (\dots, i_k,m-i_k, \dots)
\end{align*}
and to a clasp $\P^- \1_\i$ (see section \ref{sec:tanglesfund}). For example, the fact that the unknot labeled with $V_{\Lambda_i}$ evaluates to a vector space of (graded) dimension $\qbin{m}{i}$ corresponds to the fact that $\E_k^{(i)} \F_k^{(i)} \in \End((\dots,0,m,\dots))$ evaluates to $\qbin{m}{i}$ copies of the identity functor. 

\begin{Remark} 
If the objects of $\K$ are categories and the $(\sl_n,\th)$ action categorifies the $\sl_\infty$-module $\Lambda_q^{m \infty}(\k^m \otimes \k^{2\infty})$ then, at the level of Grothendieck groups, the tangle invariant from Theorem \ref{thm:main3} also recovers the $\sl_m$ Reshetikhin-Turaev tangle invariants. 
\end{Remark}

\noindent{\bf Result \#4.} In section \ref{sec:affGrass} we define a 2-category $\K_{\Gr,m}$ whose objects are derived categories of coherent sheaves on certain twisted product varieties $Y(\i) = Y(i_1) \tilde{\times} \dots \tilde{\times} Y(i_{2N})$. The 1-morphisms in $\K_{\Gr,m}$ are kernels between these varieties and the 2-morphisms are maps of kernels. 

Here $Y(i_j)$ is the Grassmannian of $i_j$-planes in $\k^m$ and should be thought of as an orbit in the affine Grassmannian of ${\rm PGL}_m$. The twisted product is then the convolution product for the affine Grassmannian. 

\begin{Theorem}\label{thm:main4}
One can define an $(\sl_\infty,\th)$ action on $\K_{\Gr,m}$ which categorifies the $U_q(\sl_\infty)$-module $\Lambda_q^{m \infty}(\k^m \otimes \k^{2 \infty})$.
\end{Theorem}

There is a similar 2-category $\K_{{\rm Q},m}$ where the objects consist of derived categories of coherent sheaves on certain Nakajima quiver varieties (where the quiver is of type $\sl_\infty$). This category can also be equipped with an $(\sl_\infty,\th)$ action which, at the level of K-theory, recovers the $U_q(\sl_\infty)$-module $\Lambda_q^{\infty}(\k^{2\infty})^{\otimes m}$. In section \ref{sec:quivers} we relate these two 2-categories using a geometric 2-functor $\K_{\Gr,m} \rightarrow \K_{{\rm Q},m}$. This 2-functor categorifies the natural projection $\Lambda_q^{m \infty}(\k^m \otimes \k^{2 \infty}) \rightarrow \Lambda_q^{\infty}(\k^{2\infty})^{\otimes m}$ of $U_q(\sl_\infty)$-modules.

\noindent{\bf Final remarks.} The clasp $\P^- \in \Kom^-(\K_{\Gr,m})$ consists of complexes of 1-morphisms which are bounded above but not below. In section \ref{sec:remarks} we discuss the analogous projectors $\P^+ \in \Kom^+(\K_{\Gr,m})$ and explain that they are related to $\P^-$ via a duality functor $\bD$. We also relate the resulting link homologies (Proposition \ref{prop:Dmap}). 

The 2-category $\K_{\Gr,m}$ has a triangulated structure. Instead of defining $\T_i$ inside $\Kom(\K_{\Gr,m})$ one could use this structure to take the convolution (iterated cone) of $\T_i$ to obtain an object in $\K_{\Gr,m}$. 

We conjecture that one can also define $\P^-$ inside $\K^-_{\Gr,m}$ (Conjecture \ref{conj:convolution}). Moreover, when $m=2$ we conjecture a geometric description of this clasp (Conjecture \ref{conj:clasp}). If these conjectures hold it would give a purely geometric construction, in terms of the affine Grassmannian for ${\rm PGL}_2$, of all the Reshetikhin-Turaev tangle invariants associated with $\g'=\sl_2$. 

\section{Some preliminaries}

In this section we collect some conventions, definitions and technical results which will be used later. Except for some examples in sections \ref{sec:computation1} and \ref{sec:computation2}, our ground field will always be $\k$. We denote by $[n]$ the quantum integer $q^{n-1} + q^{n-3} + \dots + q^{-n+3} + q^{-n+1}$. More generally, 
$$\qbin{n}{k} := \frac{[n]\dots[1]}{([n-k] \dots [1])([k]\dots[1])}.$$ 
If $f = f_aq^a \in {\mathbb N}[q,q^{-1}]$ and $A$ is an object in a graded category (see below) then we write $\oplus_f A$ for the direct sum $\oplus_{a \in \Z} A^{f_a} \la a \ra$. For example, $\oplus_{[n]} A = \oplus_{k=0}^{n-1} A \la n-1-2k \ra$. 

\subsection{Categories}\label{sec:cats}

By a graded category $\sC$ we mean a category equipped with an auto-equivalence $\la 1 \ra$. We denote by $\la l \ra$ the auto-equivalence obtained by applying $\la 1 \ra$ $l$ times. 

A graded additive $\k$-linear 2-category $\K$ is a category enriched over graded additive $\k$-linear categories. This means that the Hom categories $\Hom_{\K}(A,B)$ between objects $A$ and $B$ are graded additive $\k$-linear categories and the composition map $\Hom_{\K}(A,B) \times \Hom_{\K}(B,C) \to \Hom_{\K}(A,C)$ is a graded additive $\k$-linear functor.

An additive category $\sC$ is idempotent complete if whenever $e \in \End(A)$ and $e^2=e$ then $A \cong A_1 \oplus A_2$ where $e$ acts by the identity on $A_1$ and by zero on $A_2$. Similarly, an additive 2-category $\K$ is idempotent complete when the Hom categories $\Hom_{\K}(A,B)$ are idempotent complete for any pair of objects $A, B$ of $\K$, so that all idempotent 2-morphisms split. 

Given a 2-category $\K$ we denote by $\Kom(\K)$ its homotopy 2-category. The objects of $\Kom(\K)$ are the same as those objects of $\K$. The 1-morphisms of $\Kom(\K)$ are bounded complexes of 1-morphisms in $\K$, and 2-morphisms are maps of complexes. Two complexes of 1-morphisms are then deemed isomorphic if they are homotopy equivalent. 

\subsection{Cancellation laws}\label{sec:cancellaw}
Suppose that (each graded piece of) the space of homs between two objects in a graded category $\sC$ is finite dimensional. Then every object in ${\mathcal C}$ has a unique, up to isomorphism, direct sum decomposition into indecomposables (see section 2.2 of \cite{Rin}). This means that for any $A,B,C \in {\mathcal C}$ we have the following cancellation laws:
$$A \oplus B \cong A \oplus C \Rightarrow B \cong C \ \ \ \ \ \  A \otimes_\k V \cong B \otimes_\k V \Rightarrow A \cong B $$
where $V$ is a graded $\k$ vector space (see section 4.1 of \cite{CK3}). In this paper we also assume that every graded category satisfies the following condition:
$$\text{for any nonzero object } A \in {\mathcal C} \text{ we have } A \cong A \la k \ra \Rightarrow k=0.$$

One of the tools we will use on many occasions in this paper is the following generalization of a lemma which Bar-Natan calls ``Gaussian elimination''. 

\begin{Lemma}\label{lem:cancel}
Let $ X, Y, Z, W, U, V$ be six objects in an additive category and consider a complex
\begin{equation}\label{eq:6.5}
\dots \rightarrow U \xrightarrow{u} X \oplus Y \xrightarrow{f} Z \oplus W \xrightarrow{v} V \rightarrow \dots
\end{equation}
where $f = \left( \begin{matrix} A & B \\ C & D \end{matrix} \right)$ and $u,v$ are arbitrary morphisms. If $D: Y \rightarrow W$ is an isomorphism, then (\ref{eq:6.5}) is homotopic to a complex
\begin{eqnarray}\label{eq:new}
\dots \rightarrow U \xrightarrow{u} X \xrightarrow{A-BD^{-1}C} Z \xrightarrow{v|_Z} V \rightarrow \dots
\end{eqnarray}
\end{Lemma}
\begin{proof}
See Lemma 6.2 of \cite{CLi}. 
\end{proof}

Let $\sC$ be a graded category and consider an object $A \in \sC$ with $\End(A) \cong \k$. Now, suppose $ X, Y $ are two objects of $\sC$ and let $ f :X \rightarrow Y $ be a morphism. Then $f$ gives rise to a bilinear pairing $ \Hom(A, X) \times \Hom(Y, A) \rightarrow \Hom(A,A) = \k $. We define the {\it $A$-rank of $f$} to be the rank of this bilinear pairing and the {\it total $A$-rank of $f$} to be the sum over all $A \la i \ra$-ranks where $i \in \Z$. See section 4.1 of \cite{CK3} for another discussion of rank. 

\subsection{Convergence of complexes}\label{sec:convergence}

We now explain what it means for a sequence of complexes to converge. We try to use the same terminology as in \cite{Roz}. 

Fix an additive category $\sC$ and consider a complex ${\sf{K}}^\bullet$ in the homotopy category $\Kom(\sC)$ of $\sC$. We say that ${\sf{K}}^\bullet$ is supported in homological degrees $\le k$ if it is homotopic to a complex $\tilde{{\sf{K}}}^\bullet$ where $\tilde{{\sf{K}}}^i = 0$ if $i > k$. The infimum over all such $k$ is called the {\em homological order} of ${\sf{K}}^\bullet$ and is denoted $|{\sf{K}}^\bullet|_h$. 

A direct system 
$${\sf K}^\bullet_\star := {\sf K}_0^\bullet \xrightarrow{f_0} {\sf K}_1^\bullet \xrightarrow{f_1} {\sf K}_2^\bullet \xrightarrow{f_2} \dots$$ 
of complexes in $\Kom^-(\sC)$ is Cauchy if $\lim_{i \rightarrow \infty} |\Cone(f_i)|_h = -\infty$. Moreover, we say that $\lim_{\rightarrow} {\sf K}^\bullet_\star = {\sf K}^\bullet$ if there exist maps $\tilde{f}_i: {\sf K}^\bullet_i \rightarrow {\sf K}^\bullet$ such that $\tilde{f}_i$ is homotopic to $\tilde{f}_{i+1} \circ f_i$ and $\lim_{i \rightarrow \infty} |\Cone(\tilde{f}_i)|_h = -\infty$. 

\begin{Theorem}\cite[Thm. 2.5, 2.6]{Roz} 
A direct system ${\sf K}^\bullet_\star$ has a limit if and only if it is Cauchy. Moreover, this limit is unique up to homotopy equivalence.
\end{Theorem}

\subsection{Grothendieck groups}

If $\sC$ is an idempotent complete, additive category then we can consider the free group generated by isomorphism classes of objects in $\sC$ modulo the relation that $[B] = [A]+[C]$ if $B \cong A \oplus C$. If $\sC$ is abelian or triangulated we also quotient by this relation if $A \rightarrow B \rightarrow C$ is an exact sequence (or distinguished triangle). By $K({\mathcal C})$ we denote the tensor product of this quotient with the base field $\k$, and refer to it as the (split) Grothendieck group of $\sC$.  

If $\sC$ is also graded then the autoequivalence $\la -1 \ra$ corresponds to multiplication by an indeterminate $q$. More precisely, the object $A \la -1 \ra \in {\mathcal C}$ has class $q[A] \in K({\mathcal C})$. In this case $K({\mathcal C})$ is a $\k[q,q^{-1}]$-module rather than just a $\k$-module. 

One can likewise decategorify a 2-category to obtain a 1-category by applying the process above to 1-morphisms and 2-morphisms while keeping the objects the same. Thus, when we have a $(\g,\th)$ action on $\K$ then $K(\K)$ is a 1-category whose objects are the same and whose 1-morphisms are $\k[q,q^{-1}]$-modules that include $E_i^{(r)} \1_\l$ and $\1_\l F_i^{(r)}$ satisfying the usual relations. Finally, we forget about the 2-morphisms altogether. 

\subsection{The homotopy category $\Kom^-_*(\K)$}\label{sec:qadic}\label{sec:hptycats}

Suppose $\K$ is a graded 2-category and consider the natural map $\K \rightarrow \Kom^-(\K)$ given by including in cohomological degree zero. We would like this map to induce an isomorphism on Grothendieck group. Unfortunately, this map is not even injective. To fix this problem we consider a subcategory $\Kom^-_*(\K) \subset \Kom^-(\K)$ so that there exist maps 
$$\hK(\K) \xrightarrow{i} K(\Kom_*^-(\K)) \xrightarrow{p} \hK(\K)$$
whose composition is the identity ($\hK(\K)$ is defined below). 

To define $\Kom^-_*(\K)$ we assume that the space of 1-morphisms in $K(\K)$ is finite dimensional. Choose a basis $b_1, \dots, b_s$ of morphisms in $K(\K)$. For a morphism $[\A]$ in $K(\K)$ write $[\A] = \sum_{i,j} a_{ij} q^i b_j$ and denote $\la \A \ra_q = {\rm min}\{i: a_{ij} \ne 0 \}$. 

We now define $\Kom^-_*(\K) \subset \Kom^-(\K)$ as the subcategory made up of complexes which are homotopic to a complex $\dots \rightarrow \A^{-2} \rightarrow \A^{-1} \rightarrow \A^0$ where $\la \A^{-u} \ra_q \rightarrow \infty$ as $u \rightarrow \infty$. It is not hard to see that this condition does {\it not} depend on the choice of basis $b_1, \dots, b_s$. It is also clear that this subcategory is closed under taking cones. 

{\bf Example.} Let $\A \in \K$ be a 1-morphism. Then the complex 
$$\A^\bullet := \left[ \dots \xrightarrow{0} \A \la -2 \ra \xrightarrow{0} \A \la -1 \ra \xrightarrow{0} \A \right]$$
belongs to $\Kom^-_*(\K)$. Its class in $K(\K)$ is equal to $\left( \sum_{j \ge 0} (-1)^j q^j \right) [\A]$. On the other hand, the complex $\left[ \dots \xrightarrow{0} \A \xrightarrow{0} \A \xrightarrow{0} \A \right]$ does not belong to $\Kom^-_*(\K)$. 

We will denote by $\hK(\K)$ the completion of $K(\K)$ in the $q$-adic norm. This means that we are allowing elements of the form $f \cdot A$ where $f \in \k[[q]][q^{-1}]$ (instead of just $f \in \k[q,q^{-1}]$). Now one can define the map $p: K(\Kom^-_*(\K)) \rightarrow \hK(K)$ by mapping 1-morphisms as follows
\begin{equation}\label{eq:14}
[A^\bullet] \mapsto \sum_{u=0}^{\infty} (-1)^u [\A^{-u}].
\end{equation}
This sum converges because $\la \A^{-u} \ra_q \rightarrow \infty$ as $u \rightarrow \infty$. To see that this map is well defined suppose $\A^{\bullet} \xrightarrow{f} \B^{\bullet}$ is a homotopy equivalence. We need to show that they are mapped to the same thing. Now $\Cone(f)$, which is homotopic to zero, is mapped to $p(\A^\bullet) - p(\B^\bullet)$. So it suffice to show that any nul-homotopic complex $\fC^\bullet$ is mapped to zero. 

Now if $\fC^\bullet = [\dots \rightarrow \fC^{-2} \rightarrow \fC^{-1} \xrightarrow{f} \fC^0]$ is nul-homotopic there is a map $h: \fC^0 \rightarrow \fC^{-1}$ so that $f \circ h = \id_{\fC^0}$. Thus $\fC^{-1} = \fC^0 \oplus \hat{\fC}^{-1}$ for some $\hat{\fC}^{-1}$ and $\fC^{\bullet}$ is homotopic to $[\dots \rightarrow \fC^{-2} \rightarrow \hat{\fC}^{-1} \rightarrow 0]$. Now we repeat. Since $|[\fC^{-u}]|_q \rightarrow 0$ this means that for any $e$ one can find $d \gg 0$ so that $\sum_{u=0}^d (-1)^u [\fC^{-u}] = \sum_{i,j} a_{ij} q^i b_j$ with $a_{ij}=0$ if $i < e$. It follows that $\sum_{u=0}^d (-1)^u [\fC^{-u}]$ converges to zero as $d \rightarrow \infty$. 


\begin{Remark} 
Instead of bounded above complexes one could consider bounded below complexes. This leads to the analogous definitions of $\Kom^+_*(\K)$. With the exception of a discussion in section \ref{sec:remarks} we will work with $\Kom_*^-(\K)$. However, we could have worked with $\Kom_*^+(\K)$ since all our arguments apply in the same way.
\end{Remark}

\section{Rickard complexes and relations}

In this section $\g = \sl_2$ so we abbreviate $\E_1$ and $\F_1$ as $\E$ and $\F$. We will abuse notation slightly and write $\l \in \Z$ instead of $\l_1$. Also, $Y_\k$ is one-dimensional so we take $\th \in Y_\k$ to be any nonzero vector. Our aim is to define and study how the complexes (\ref{eq:cpx'}) and (\ref{eq:cpx}) commute with functors $\E^{(p)}$ and $\F^{(p)}$.

\subsection{Rickard complexes}

First recall the following result which describes how $\E$s and $\F$s commute. 

\begin{Lemma}\label{lem:rels1}
We have
$$\E^{(a)} \F^{(b)} \1_\l \cong \bigoplus_{j \ge 0} \bigoplus_{\qbins{\l+a-b}{j}} \F^{(b-j)} \E^{(a-j)} \1_\l \text{  if  } \l+a-b \ge 0$$
$$\F^{(b)} \E^{(a)} \1_\l \cong \bigoplus_{j \ge 0} \bigoplus_{\qbins{-\l-a+b}{j}} \E^{(a-j)} \F^{(b-j)} \1_\l \text{  if  } \l+a-b \le 0.$$
\end{Lemma}
\begin{proof}
See, for instance, \cite[Thm. 5.9]{KLMS}. 
\end{proof}

Suppose from hereon that $\l+r \ge 0$. We define the complex $\1_\l \tau_{\l+r} \1_{-\l-2r} \in \Kom(\K)$ by 
\begin{align}
\label{eq:cpx1} 
\dots \rightarrow \1_\l \E^{(\l+r+s)} \F^{(s)} \la -s(r+1) \ra \xrightarrow{d^s} \1_\l \E^{(\l+r+s-1)} \F^{(s-1)} \la -(s-1)(r+1) \ra \rightarrow \dots \rightarrow \1_\l \E^{(\l+r)} & \\
\label{eq:cpx2}
\dots \rightarrow \1_\l \F^{(s)} \E^{(\l+r+s)} \la s(r-1) \ra \xrightarrow{d^s} \1_\l \F^{(s-1)} \E^{(\l+r+s-1)} \la (s-1)(r-1) \ra \rightarrow \dots \rightarrow \1_\l \E^{(\l+r)} &
\end{align}
depending on whether $r \ge 0$ or $r \le 0$. The differential $d^s$ in (\ref{eq:cpx1}) is given as the composition
\begin{align*}
\1_\l \E^{(\l+r+s)} \F^{(s)} 
&\rightarrow \1_\l \E^{(\l+r+s-1)} \E \1_{-\l-2r-2s} \F \F^{(s-1)} \la -\l-r-2s+2 \ra \\
&\cong \1_\l \E^{(\l+r+s-1)} \E \E_R \la \l+2r+2s-1 \ra \F^{(s-1)} \la -\l-r-2s+2 \ra \\
&\rightarrow \1_\l \E^{(\l+r+s-1)} \F^{(s-1)} \la r+1 \ra 
\end{align*}
where the first map is inclusion into the lowest degree summand and the last map is adjunction. The differential in (\ref{eq:cpx2}) is similar. 

Likewise, we define the complex $\1_{-\l} \tau'_{\l+r} \1_{\l+2r} \in \Kom(\K)$ by
\begin{align}
\label{eq:cpx3}
\dots \rightarrow \1_{-\l} \F^{(\l+r+s)} \E^{(s)} \la -s(r+1) \ra \xrightarrow {d^s} \1_{-\l} \F^{(\l+r+s-1)} \E^{(s-1)} \la -(s-1)(r+1) \ra \rightarrow \dots \rightarrow \1_{-\l} \F^{(\l+r)} & \\
\label{eq:cpx4}
\dots \rightarrow \1_{-\l} \E^{(s)} \F^{(\l+r+s)} \la s(r-1) \ra \xrightarrow {d^s} \1_{-\l} \E^{(s-1)} \F^{(\l+r+s-1)} \la (s-1)(r-1) \ra \rightarrow \dots \rightarrow \1_{-\l} \F^{(\l+r)} &
\end{align}
depending on whether $r \ge 0$ or $r \le 0$. Notice that when $r=0$ then (\ref{eq:cpx1}) and (\ref{eq:cpx2}) are both equal to the complex (\ref{eq:cpx'}) while (\ref{eq:cpx3}) and (\ref{eq:cpx4}) are both equal to (\ref{eq:cpx}). In \cite{CR} the complexes (\ref{eq:cpx'}) and (\ref{eq:cpx}) are called Rickard complexes so we will also refer to (\ref{eq:cpx1}) - (\ref{eq:cpx4}) as Rickard complexes. Note, however, that if $r \ne 0$ then the complexes (\ref{eq:cpx1})--(\ref{eq:cpx4}) are {\em not} invertible in $\Kom(\K)$. 

\begin{Remark}\label{rem:cpx} By Lemma \ref{lem:homs} the space of maps between two consecutive terms in the complex (\ref{eq:cpx1}) is one-dimensional. This means that any complex which is indecomposable and which has the same terms as the complex in (\ref{eq:cpx1}) must actually be homotopic to (\ref{eq:cpx1}). Subsequently, we do not need to worry about what precise differential to choose (just pick any nonzero multiple of the unique map). The same is also true of complexes (\ref{eq:cpx2}), (\ref{eq:cpx3}) and (\ref{eq:cpx4}). 
\end{Remark}

\begin{Lemma}\label{lem:homs} 
If $\dots \rightarrow \A \rightarrow \A' \rightarrow \dots$ are two consecutive terms appearing in any of the four complexes (\ref{eq:cpx1})--(\ref{eq:cpx4}) then $\Hom(\A,\A \la \ell \ra) \cong \Hom(\A,\A' \la \ell \ra) \cong \begin{cases} 0 \text{ if } \ell < 0 \\ \k \text{ if } \ell=0. \end{cases}$
\end{Lemma}
\begin{proof}
The proof is basically the same for all four complexes. We illustrate with (\ref{eq:cpx3}) where 
$$\A = \1_{-\l} \F^{(\l+r+s)} \E^{(s)} \la -s(r+1) \ra \text{ and } \A' = \1_{-\l} \F^{(\l+r+s-1)} \E^{(s-1)} \la -(s-1)(r+1) \ra.$$
We have:
\begin{align*}
\Hom(\A,\A') 
&\cong \Hom(\1_{-\l} \F^{(\l+r+s)} \E^{(s)}, \1_{-\l} \F^{(\l+r+s-1)} \E^{(s-1)} \la r+1 \ra) \\
&\cong \Hom((\1_{-\l} \F^{(\mu)})_L \F^{(\mu+1)} \E^{(s)}, \E^{(s-1)} \1_{\l+2r} \la r+1 \ra) \\
&\cong \Hom(\E^{(\mu)} \F^{(\mu+1)} \1_{\l+2r+2s} \E^{(s)} \la \mu(r+s-1) \ra, \E^{(s-1)} \la r+1 \ra) \\
&\cong \bigoplus_{j=0}^\mu \bigoplus_{\qbins{\l+2r+2s-1}{j}} \Hom(\1_{\l+2r+2s-2} \F^{(\mu-j+1)} \E^{(\mu-j)} \E^{(s)}, \E^{(s-1)} \la r+1 - \mu(r+s-1) \ra) \\
&\cong \bigoplus_{j=0}^\mu \bigoplus_{\qbins{\mu+r+s}{j}} \Hom(\E^{(\mu-j)} \E^{(s)}, \E^{(\mu-j+1)} \E^{(s-1)} \la t \ra) \\
&\cong \bigoplus_{j=0}^\mu \bigoplus_{\qbins{\mu+r+s}{j}} \Hom ( \bigoplus_{\qbins{\mu-j+s}{s}} \E^{(\mu-j+s)}, \bigoplus_{\qbins{\mu-j+s}{s-1}} \E^{(\mu-j+s)} \la t \ra )
\end{align*}
where, for convenience, we abbreviated $\mu := \l+r+s-1$ and 
$$t = r+1 - \mu(r+s-1) - (\mu-j+1)(2\mu+r+s-1-j).$$ 
Thus the last line is a direct sum of terms of the form $\Hom(\E^{(\mu-j+s)}, \E^{(\mu-j+s)} \la k \ra)$ for some $k$. Since $\qbins{N}{m}$ is supported in degrees $-m(N-m), \dots,  m(N-m)$ we find that 
$$k \le j(\mu+r+s-j) + s(\mu-j) + (s-1)(\mu-j+1) + t = -2(\mu-j)^2 -2(\mu-j)(r+1).$$
Since $\mu \ge j$ and $\E^{(\mu-j+s)}$ has no negative degree endomorphisms we see that $\Hom(\A,\A' \la \ell \ra) = 0$ if $\ell < 0$. Moreover, if $\ell = 0$ then there is precisely one term which is nonzero (when $j=\mu$) and we get $\Hom(\A,\A') \cong \End(\E^{(s)}) \cong \k$. A similar argument also works with $\Hom(\A,\A \la \ell \ra)$.

\begin{Remark} Above we used that $\dim \End^k(\E^{(r)}) = \begin{cases} 0 \text{ if } k < 0 \\ 1 \text{ if } k=0. \end{cases}$ This is proved by induction in \cite[Lemma 4.9]{CKL2} assuming only condition (\ref{co:hom1}) from our definition in section \ref{sec:catgactions}. 
\end{Remark}
\end{proof}

\subsection{The affine nilHecke algebra}\label{sec:nilHecke}

In \cite{C2} we show that a $(\g,\th)$ action carries an action of the quiver Hecke algebras. In the simplest case when $\g=\sl_2$ these algebras are just the affine nilHecke algebras. 
Recall that the action of the affine nilHecke algebra consists of two 2-morphisms
$$X: \E \1_\l \rightarrow \E \1_\l \la 2 \ra \ \ \text{ and } \ \ T: \E \E \1_\l \rightarrow \E \E \1_\l \la -2 \ra $$
which satisfy the following relations
\begin{enumerate}
\item $T^2=0$ where $T \in \End(\E \E)$,
\item $(IT)(TI)(IT)=(TI)(IT)(TI)$ where $TI,IT \in \End(\E \E \E)$,
\item $(XI)T-T(IX)=II=-(IX)T+T(XI)$ where $XI,IX,T \in \End(\E \E)$.
\end{enumerate}
We will use this action (in a minor way) in several of the proofs in this section.

\subsection{The main Proposition}

\begin{Proposition}\label{prop:0} 
For $p \ge 0$ we have 
\begin{itemize}
\item $\tau_{\l} \1_{-\l} \E^{(p)} \cong \tau_{\l+p} \1_{-\l-2p} \la -p(\l+p+1) \ra [p]$ if $\l \ge 0$
\item $\tau_{\l} \1_{-\l} \F^{(p)} \cong \tau_{\l-p} \1_{-\l+2p}$ if $\l-p \ge 0$
\item $\E^{(p)} \tau_\l \1_{-\l} \cong \tau_{\l+p} \1_{-\l} \la -p(\l+p+1) \ra [p]$ if $\l \ge 0$
\item $\F^{(p)} \tau_\l \1_{-\l} \cong \tau_{\l-p} \1_{-\l}$ if $\l-p \ge 0$.
\end{itemize}
Similarly, we have 
\begin{itemize}
\item $\tau'_{\l} \1_\l \E^{(p)} \cong \tau'_{\l-p} \1_{\l-2p}$ if $\l-p \ge 0$
\item $\tau'_{\l} \1_\l \F^{(p)} \cong \tau'_{\l+p} \1_{\l-2p} \la -p(\l+p+1) \ra [p]$ if $\l \ge 0$
\item $\E^{(p)} \tau'_\l \1_\l \cong \tau'_{\l-p} \1_\l$ if $\l-p \ge 0$
\item $\F^{(p)} \tau'_\l \1_\l \cong \tau'_{\l+p} \1_\l \la -p(\l+p+1) \ra [p]$ if $\l \ge 0$.
\end{itemize}
\end{Proposition}
\begin{proof}
Let us prove that 
\begin{equation}\label{eq:prop01}
\tau'_{\l} \1_\l \F^{(p)} \cong \tau'_{\l+p} \1_{\l-2p} \la -p(\l+p+1) \ra [p].
\end{equation}
We do this by induction on $p$. We assume the result for $p \le r$ and consider $\1_{-\l} \tau'_{\l+r} \F $. The general term here is 
$$\1_{-\l} \F^{(\l+r+s)} \E^{(s)} \F \cong \left[ \bigoplus_{[\l+r+s+1]} \F^{(\l+r+s+1)} \E^{(s)} \bigoplus_{[\l+2r+s+1]} \F^{(\l+r+s)} \E^{(s-1)} \right] \la -s(r+1) \ra$$
which gives us a complex 
\begin{equation*}
\dots \rightarrow 
\left[ \begin{matrix} \oplus_{[\l+r+s+1]} \F^{(\l+r+s+1)} \E^{(s)} \\ \oplus_{[\l+2r+s+1]} \F^{(\l+r+s)} \E^{(s-1)} \end{matrix} \right] \la -s(r+1) \ra \xrightarrow{\alpha} 
\left[ \begin{matrix} \oplus_{[\l+r+s]} \F^{(\l+r+s-1)} \E^{(s-1)} \\ \oplus_{[\l+2r+s]} \F^{(\l+r+s-1)} \E^{(s-2)} \end{matrix} \right] \la -(s-1)(r+1) \ra \rightarrow \dots 
\end{equation*}
By Lemma \ref{lem:alphainduces} the map $\alpha$ induces a surjective map 
$$\oplus_{[\l+2r+s+1]} \F^{(\l+r+s)} \E^{(s-1)} \la -s(r+1) \ra \rightarrow \oplus_{[\l+r+s]} \F^{(\l+r+s-1)} \E^{(s-1)} \la -(s-1)(r+1) \ra.$$ 
Using the cancellation Lemma \ref{lem:cancel} on all such terms in every degree we end up with a complex
\begin{equation}\label{eq:prop02}
\left[ \dots \longrightarrow \begin{matrix} \bigoplus_{[r+1]} \F^{(\l+r+s-1)} \E^{(s-1)} \\ \la -(s-1)(r+2) \ra \end{matrix} \longrightarrow 
\begin{matrix} \bigoplus_{[r+1]} \F^{(\l+r+s-2)} \E^{(s-2)} \\ \la -(s-2)(r+2) \ra \end{matrix} \longrightarrow \dots \right] \la -(\l+2r+2) \ra [1].
\end{equation}
We would like to show that this is isomorphic to 
\begin{equation}\label{eq:prop03}
\bigoplus_{[r+1]} \left[ \dots \rightarrow \begin{matrix} \F^{(\l+r+s-1)} \E^{(s-1)} \\ \la -(s-1)(r+2) \ra \end{matrix} \longrightarrow
\begin{matrix} \F^{(\l+r+s-2)} \E^{(s-2)} \\ \la -(s-2)(r+2) \ra \end{matrix} \longrightarrow \dots \right] \la -(\l+2r+2) \ra [1]
\end{equation}
for some sequence of differentials. To do this we use the following trick. Suppose one has a complex of the form 
$$\B^\bullet := \dots \rightarrow \oplus_{[r+1]} \A^{-n} \rightarrow \oplus_{[r+1]} \A^{-n+1} \rightarrow \dots$$
where $\Hom(\A^{-n},\A^{-n+1} \la \ell \ra) = 0$ if $\ell < 0$ and $\Hom(\A^{-n},\A^{-n+1})$ is one-dimensional. Then there are two natural maps
$$\iota: {\sf C}_1^\bullet \la r+1 \ra \rightarrow \B^\bullet \text{ and } \pi: \B^\bullet \rightarrow {\sf C}_2^\bullet \la -(r+1) \ra$$
where ${\sf C}_1^\bullet$ and ${\sf C}_2^\bullet$ are both complexes of the form $\dots \rightarrow \A^{-n} \rightarrow \A^{-n+1} \rightarrow \dots$. These maps are just including into and projecting from $\B^\bullet$. Now suppose further that there exists a map $\gamma: \B^\bullet \rightarrow \B^\bullet \la 2 \ra$ which, in every homological degree $-n$, induces an isomorphism between $r$ summands $\A^{-n}$. Then the composition 
$${\sf C}_1^\bullet \xrightarrow{\iota} \B^\bullet \la -(r+1) \ra \xrightarrow{\gamma^{r+1}} \B^\bullet \la r+1 \ra \xrightarrow{\pi} {\sf C}_2^\bullet$$
is a homotopy equivalence. This in turn implies that $\B^\bullet \cong \oplus_{[r+1]} {\sf C}_1^\bullet$.

Now let us apply this to the case where $\B^\bullet$ is the complex in (\ref{eq:prop02}). The $\Hom$ conditions on the $\A^{-n}$s follow from Lemma \ref{lem:homs}. The map $\gamma$ is given by 
$IIX: \1_{-\l} \tau'_{\l+r} \F \rightarrow \1_{-\l} \tau'_{\l+r} \F \la 2 \ra.$
It has the isomorphism property described above because, by Lemma \ref{lem:IXonEF},
$$IIIX: \1_{-\l} \F^{(\l+r+s)} \E^{(s)}  \F \rightarrow \1_{-\l} \F^{(\l+r+s)} \E^{(s)} \F \la 2 \ra$$ 
induces an isomorphism between all but one summand of the form $\F^{(\l+r+s)} \E^{(s-1)}$ on either side. 

Thus, we conclude that $\1_{-\l} \tau'_{\l+r} \F$ is isomorphic to a complex as in (\ref{eq:prop03}) for some choice of differentials. On the other hand, by induction we know that $\1_{-\l} \tau'_{\l} \F^{(r)} \cong \1_{-\l} \tau'_{\l+r} \la -r(\l+r+1) \ra [r]$, which means 
$$\1_{-\l} \tau'_{\l+r} \F \cong \1_{-\l} \tau'_{\l} \F^{(r)} \F \la r(\l+r+1) \ra [-r] \cong \oplus_{[r+1]} \1_{-\l} \tau'_{\l} \F^{(r+1)} \la r(\l+r+1) \ra [-r].$$
Hence $\1_{-\l} \tau'_{\l} \F^{(r+1)}$ must be isomorphic to one of the summands in (\ref{eq:prop03}), namely, to a complex 
$$\left[ \dots \longrightarrow \begin{matrix} \F^{(\l+r+s-1)} \E^{(s-1)} \\ \la -(s-1)(r+2) \ra \end{matrix} \longrightarrow \begin{matrix} \F^{(\l+r+s-2)} \E^{(s-2)} \\ \la -(s-2)(r+2) \ra \end{matrix} \longrightarrow \dots \right] \la -(r+1)(\l+r+2) \ra [r+1].$$
Notice that the terms in this complex are the same as those in $\1_{-\l} \tau'_{\l+r+1} \la -(r+1)(\l+r+2) \ra [r+1]$. Since $\1_{-\l} \tau'_\l$ is invertible $\1_{-\l} \tau'_\l \F^{(r+1)}$ must be indecomposable. Hence, by Remark \ref{rem:cpx}, 
$$\1_{-\l} \tau'_{\l} \F^{(r+1)} \cong \1_{-\l} \tau'_{\l+r+1} \la -(r+1)(\l+r+2) \ra [r+1]$$
and our induction is complete. 
\end{proof}

Since $\1_\l \tau_\l = \1_\l \T$ the following is an immediate Corollary of Proposition \ref{prop:0}. 

\begin{Corollary}\label{cor:0} 
For $p \ge 0$ we have
\begin{itemize}
\item $\T \1_\l \F^{(p)} \cong \1_{-\l} \E^{(p)} \T \la -p(\l+p+1) \ra [p]$ if $\l \ge 0$
\item $\T \1_\l \E^{(p)} \cong \1_{-\l} \F^{(p)} \T \la -p(-\l+p+1) \ra [p]$ if $ \l \le 0$
\item $\T \1_\l \E^{(\l)} \cong \1_{-\l} \F^{(\l)} \T$ if $\l \ge 0$
\item $\T \1_\l \F^{(-\l)} \cong \1_{-\l} \E^{(-\l)} \T$ if $\l \le 0$.
\end{itemize}
\end{Corollary}

\subsection{Some useful maps}\label{sec:helpmaps} 

Recall that 
\begin{equation}\label{eq:temp}
\F \F^{(r)} \cong \oplus_{[r+1]} \F^{(r+1)} \cong \F^{(r)} \F
\end{equation}
so (abusing notation a little) we can define maps 
$$\iota: \F^{(r+1)} \rightarrow \F \F^{(r)} \la -r \ra \text{ and } \iota: \F^{(r+1)} \rightarrow \F^{(r)} \F \la -r \ra$$
by including into the lowest summand. Including into the lowest (as opposed to say highest) summand is natural because there is a unique (up to a nonzero multiple) such map {\it i.e.} the map does not depend on the choice of isomorphisms in (\ref{eq:temp}). This is a consequence of the fact $\dim \End^\ell(\F^{(r+1)})$ is zero if $\ell < 0$ and one if $\ell=0$. Likewise, we have maps 
$$\pi: \F \F^{(r)} \rightarrow \F^{(r+1)} \la -r \ra \text{ and } \pi: \F^{(r)} \F \rightarrow \F^{(r+1)} \la -r \ra$$
given by projecting out of the top degree summand in $\F \F^{(r)}$. Again, these maps are unique. 

Next we have the adjunction maps 
\begin{align*}
\epsilon: \F \E \1_\mu \rightarrow \1_\mu \la \mu+1 \ra \ \ &\text{ and } \ \ \eta: \1_\mu \rightarrow \E \F \1_\mu \la -\mu+1 \ra \\ 
\epsilon: \E \F \1_\mu \rightarrow \1_\mu \la -\mu+1 \ra \ \ &\text{ and } \ \ \eta: \1_\mu \rightarrow \F \E \1_\mu \la \mu+1 \ra.
\end{align*}

More generally we can define the map $\epsilon'$ and $\eta'$ as the compositions
\begin{align*}
& \epsilon': \F \1_{\mu} \E^{(s)} \xrightarrow{I \iota} \F \1_{\mu} \E \E^{(s-1)} \la -(s-1) \ra \xrightarrow{\epsilon} \E^{(s-1)} \la \mu-s \ra \\
& \eta': \E^{(s-1)} \1_{\mu+2} \xrightarrow{I \eta} \E^{(s-1)} \1_{\mu} \E \F \la -\mu-1 \ra \xrightarrow{\pi I} \E^{(s)} \F \la -\mu-s \ra.
\end{align*}
Finally, by adjunction we have $\dim \Hom(\F \E \1_\mu, \E \F \1_\mu) = 1$ \cite[Lem. A6]{C2} so we fix $\chi: \F \E \rightarrow \E \F$ to be this map (uniquely defined up to rescaling). 

\begin{Lemma}\label{lem:chi}
The composition $\F \E \1_\mu \xrightarrow{\chi} \E \F \1_\mu \xrightarrow{IX} \E \F \la 2 \ra \1_\mu$ is equal to a nontrivial linear combination of the compositions $\F \E \1_\mu \xrightarrow{\epsilon} \1_\mu \la \mu+1 \ra \xrightarrow{\eta} \E \F \la 2 \ra \1_\mu$ and $\F \E \1_\mu \xrightarrow{XI} \F \E \la 2 \ra \1_\mu \xrightarrow{\chi} \E \F \la 2 \ra \1_\mu$.
\end{Lemma}
\begin{proof}
Since $\Hom(\F \E \1_\mu, \E \F \1_\mu) \cong \k$ the map $\chi$ is, up to rescaling, equal to the composition 
$$\F \E \1_\mu \xrightarrow{\eta II} \E \F \F \E \la -\mu+1 \ra \1_\mu \xrightarrow{ITI} \E \F \F \E \la -\mu-1 \ra \1_\mu \xrightarrow{II \epsilon} \E \F \1_\mu.$$
The result then follows using the affine nilHecke relation $(XI)T = II + T(IX)$. 
\end{proof}

\begin{Lemma}\label{lem:IXonEF}
If $\mu + s - 1 \ge 0$ then the total $\E^{(s-1)}$-rank of $\E^{(s)} \F \1_\mu \xrightarrow{IX} \E^{(s)} \F \1_\mu$ is $\mu+s-2$. 
\end{Lemma}
\begin{proof}
This can be proved by induction on $s$. For instance, the base $s=1$ is equivalent to condition (\ref{co:theta}). However, to save work one can use \cite[Lemma 4.16]{KLMS} with $a=s-1,b=0,n=\mu$ and $x+y=\mu+s-2$ to obtain that, up to rescaling, the composition 
$$\E^{(s-1)} \1_\mu \xrightarrow{(\pi I)(I\eta)} \E^{(s)} \F \1_\mu \la -\mu-s+2 \ra \xrightarrow{IX^{\mu+s-2}I} \E^{(s)} \F \1_\mu \la \mu+s-2 \ra \xrightarrow{(I \epsilon I)(\iota II)} \E^{(s-1)} \1_\mu$$
is equal to the identity. Thus the map $IXI \in \End^2(\E^{(s)} \F \1_\mu)$ must induce an isomorphism between $(\mu+s-2)$ summands $\E^{(s-1)} \1_\mu$ on either side. The result follows. 

\end{proof}

\subsection{The main Lemma}

\begin{Lemma}\label{lem:alphainduces} 
If $\l,r,s \ge 0$ then the map 
$$\1_{-\l} \F^{(\l+r+s)} \E^{(s)} \F \xrightarrow{d^sI} \1_{-\l} \F^{(\l+r+s-1)} \E^{(s-1)} \F \la r+1 \ra$$
is surjective on summands of the form $\1_{-\l} \F^{(\l+r+s)} \E^{(s-1)}$. 
\end{Lemma}
\begin{proof}
Let $\A := \1_{-\l} \F^{(\l+r+s)} \E^{(s-1)}$. Note that by Lemma \ref{lem:homs} we have $\End(\A) \cong \k$. By Lemma \ref{lem:IXonEF}, the summands $\A$ inside $\1_{-\l} \F^{(\l+r+s)} \E^{(s)} \F$ are all picked up by the composition 
\begin{equation}\label{eq:comp1}
\A \la \l + 2r + s - 2k \ra \xrightarrow{II \eta'} \1_{-\l} \F^{(\l+r+s)} \E^{(s)} \F \la -2k \ra \xrightarrow{IIIX^k} \1_{-\l} \F^{(\l+r+s)} \E^{(s)} \F 
\end{equation}
where $k=0, \dots, \l+2r+s$. In order to prove surjectivity of $d^sI$ it suffices to show that the composition of the map in (\ref{eq:comp1}) with $d^s$ has $\A$-rank one for $k=0, \dots, \l+r+s$. Now consider the following commutative diagram 
\begin{equation}\label{eq:lemalpha0}
\xymatrix{
\F^{(\l+r+s)} \E^{(s)} \F \ar[r]^{\iota II} & \F^{(\l+r+s-1)} \F \E^{(s)} \F \ar[r]^{\epsilon'} & \F^{(\l+r+s-1)} \E^{(s-1)} \F \\
\F^{(\l+r+s)} \E^{(s)} \F \ar[u]^{IIX^k} \ar[r]^{\iota II} & \F^{(\l+r+s-1)} \F \E^{(s)} \F \ar[r]^{I \epsilon' I} \ar[u]^{IIIX^k} & \F^{(\l+r+s-1)} \E^{(s-1)} \F \ar[u]^{IIX^k} \\
\F^{(\l+r+s)} \E^{(s-1)} \ar[r]^{\iota I} \ar[u]^{I \eta'} & \F^{(\l+r+s-1)} \F \E^{(s-1)} \ar[u]^{II \eta'} \ar[r]^{=} & \F^{(\l+r+s-1)} \F \E^{(s-1)} \ar[u]^{I \chi'}
}
\end{equation} 
where $k \in \{0, \dots, \l+r+s\}$ and $\chi' = (\epsilon' I)(I \eta')$. For simplicity we have ommited the $\la \cdot \ra$ shifts. All squares clearly commute. Since $\iota$ is an inclusion it suffices to show that the composition 
\begin{equation}
\label{eq:new3}
\1_{-\l} \F^{(\l+r+s-1)} \F \E^{(s-1)} \xrightarrow{I \chi'} \1_{-\l} \F^{(\l+r+s-1)} \E^{(s-1)} \F \xrightarrow{IIX} \1_{-\l} \F^{(\l+r+s-1)} \E^{(s-1)} \F \la 2 \ra
\end{equation}
has $\A$-rank $(\l+r+s-1)$. Now, it is not hard to check that the composition
$$\bigoplus_{[(s-1)!]} \F \E^{(s-1)} \cong \F \E^{s-1} \xrightarrow{\chi I^{s-2}} \E \F \E^{s-2} \rightarrow \dots \rightarrow \E^{s-2} \F \E \xrightarrow{I^{s-2} \chi} \E^{s-1} \F \cong \bigoplus_{[(s-1)!]} \E^{(s-1)} \F$$
is given by $\chi'$ acting diagonally. For $1 \le a \le s-1$ we can use Corollary \ref{lem:chi} to conclude that
\begin{equation}
\label{eq:temp1}
\F^{(\l+r+s-1)} \left[ \E^{a-1} \F \E^{s-a} \xrightarrow{I^{a-1} \chi I^{s-a-1}} \E^{a} \F \E^{s-a-1} \xrightarrow{I^a X I^{s-a-1}} \E^{a} \F \E^{s-a-1} \la 2 \ra \right]
\end{equation}
is a nontrivial linear combination of the following two compositions
\begin{align}
\label{eq:temp2} \F^{(\l+r+s-1)} & \left[ \E^{a-1} \F \E \E^{s-a-1} \xrightarrow{I^{a-1} \epsilon I^{s-a-1}} \E^{a-1} \E^{s-a-1} \xrightarrow{I^{a-1} \eta I^{s-a-1}} \E^{a-1} \E \F \E^{s-a-1} \right] \\
\label{eq:temp3} \F^{(\l+r+s-1)} & \left[\E^{a-1} \F \E \E^{s-a-1} \xrightarrow{I^{a-1} XI I^{s-a-1}} \E^{a-1} \F \E \E^{s-a-1} \xrightarrow{I^{a-1} \chi I^{s-a-1}} \E^{a-1} \E \F \E^{s-a-1} \right]
\end{align}
where we omit the grading shifts for convenience. The composition (\ref{eq:temp2}) factors through $\F^{(\l+r+s-1)} \E^{s-2}$ which contains no summands $\A$. So its $\A$-rank is zero which means that (\ref{eq:temp1}) and (\ref{eq:temp3}) have the same $\A$-rank. Applying this fact repeatedly we find that the two compositions 
\begin{align}
\label{eq:new1} & \F^{(\l+r+s-1)} \left[\F \E^{s-1} \xrightarrow{\chi I^{s-2}} \E \F \E^{s-2} \rightarrow \dots \rightarrow \E^{s-2} \F \E \xrightarrow{I^{s-2} \chi} \E^{s-1} \F \xrightarrow{I^{s-1}X} \E^{s-1} \F \la 2 \ra \right] \\
\label{eq:new2} & \F^{(\l+r+s-1)} \left[\F \E^{s-1} \xrightarrow{XI^{s-1}} \F \E^{s-1} \la 2 \ra \xrightarrow{\chi I^{s-2}} \E \F \E^{s-2} \la 2 \ra \rightarrow \dots \rightarrow \E^{s-2} \F \E \la 2 \ra \xrightarrow{I^{s-2} \chi} \E^{s-1} \F \la 2 \ra\right] 
\end{align}
have the same $\A$-rank. Using the affine nilHecke relations (see also \cite[Prop. 4.2]{CKL2}) the map $\F^{(\l+r+s-1)} \F \xrightarrow{IX} \F^{(\l+r+s-1)} \F \la 2 \ra$ has $\F^{(\l+r+s)}$-rank $(\l+r+s-1)$. It follows that 
$$\F^{(\l+r+s-1)} \F \E^{s-1} \xrightarrow{IXI^{s-1}} \F^{(\l+r+s-1)} \F \E^{s-1} \la 2 \ra$$
has $\A$-rank $(\l+r+s-1)(s-1)!$. The $\chi$ in the rest of the maps in (\ref{eq:new2}) are inclusions so the $\A$-rank of (\ref{eq:new2}) is also $(\l+r+s-1)(s-1)!$. It follows that (\ref{eq:new3}) has $\A$-rank $(\l+r+s-1)$ and we are done. 
\end{proof}

\section{The functor $\P^- = \T_\omega^\infty$}\label{sec:thm1}

In this section we prove that $\lim_{\ell \rightarrow \infty} \T_\omega^{2 \ell}$ is well defined and belongs to $\Kom^-_*(\K)$ (Theorem \ref{thm:main1}). 

\subsection{Braid group actions}\label{sec:braidgrpaction}

In \cite{CK3} we showed that a geometric categorical $\sl_n$ action induces an action of the braid group on its weight spaces via the complexes $\T_i \1_\l$. The definition of a geometric action was tailored to work with categories of coherent sheaves. However, the same proof works to show the following. 

\begin{Proposition}\label{prop:braidgrpaction}
Suppose $\g=\sl_n$ and consider a $(\g,\th)$ action on $\K$. Then, inside $\Kom(\K)$, the complexes $\T_i$ satisfy the braid relations $\T_i \T_j \T_i \cong \T_j \T_i \T_j$ if $|i-j|=1$ and $\T_i \T_j \cong \T_j \T_i $ if $|i-j| > 1$. 
\end{Proposition}

Now, suppose $\g = \sl_n$ and $|i-j|=1$. Then we have maps $T_{ij} \in \Hom (\E_i \E_j \la -1 \ra, \E_j \E_i \la 1 \ra)$ and likewise $T_{ij} \in \Hom(\F_i \F_j, \F_j \F_i \la 1 \ra)$ coming from the quiver Hecke algebra action. In this case these maps are easy to identify since both $\Hom$-spaces are one-dimensional \cite[Lem. A4]{C2}. We define 
$$\E_{ij} := [\E_i \E_j \la -1 \ra \xrightarrow{T_{ij}} \E_j \E_i] \ \ \text{ and } \ \ \F_{ij} := [\F_i \F_j \xrightarrow{T_{ij}} \F_j \F_i \la 1 \ra]$$ 
where $\E_i \E_j$ and $\F_i \F_j$ both lie in cohomological degree zero.

\begin{Lemma}\label{lem:2}
If $\g = \sl_n$ and $|i-j|=1$ then 
\begin{align*}
\E_{ij} \T_i \1_\l &\cong
\begin{cases}
\T_i \E_j \ \ &\text{ if } \ \ \l_i > 0 \\
\T_i \E_j [1]\la -1 \ra \ \ &\text{ if } \ \ \l_i \le 0
\end{cases} \\
\F_{ij} \T_i \1_\l &\cong
\begin{cases}
\T_i \F_j \ \ &\text{ if } \ \ \l_i \ge 0 \\
\T_i \F_j [-1]\la 1 \ra \ \ &\text{ if } \ \ \l_i < 0
\end{cases}\\
\1_\l \T_j \E_{ij} &\cong
\begin{cases}
\E_i \T_j \ \ &\text{ if } \ \ \l_j < 0 \\
\E_i \T_j [1] \la -1 \ra \ \ &\text{ if } \ \ \l_j \ge 0
\end{cases}\\
\1_\l \T_j \F_{ij} &\cong
\begin{cases}
\F_i \T_j \ \ &\text{ if } \ \ \l_j \le 0 \\
\F_i \T_j [-1] \la 1 \ra \ \ &\text{ if } \ \ \l_j > 0
\end{cases}
\end{align*}
\end{Lemma}
\begin{Remark} Note that if $|i-j|>1$ then it is clear that $\T_i$ commutes with $\E_j$ and $\F_j$.\end{Remark} 
\begin{proof}
The first assertion, namely that $\E_{ij} \T_i \1_\l \cong \T_i \E_j$ if $\l_i > 0$, was proved in Corollary 5.5 of \cite{CK3}. The rest of the claims follow via the same argument. 
\end{proof}

\subsection{The map $\1 \rightarrow \T_\omega^2$}\label{sec:1toT2map}

By the definition of $\T_i \1_\l$ we have natural maps 
\begin{align*}
\F_i^{(\l_i)} \1_\l \longrightarrow \T_i \1_\l & \text{ if } \l_i \ge 0 \\
\E_i^{(-\l_i)} \1_\l \longrightarrow \T_i \1_\l & \text{ if } \l_i \le 0.
\end{align*}
Moreover, if $\l_i \ge 0$, we have
\begin{equation}\label{eq:1}
\E_i^{(\l_i)} \1_{s_i \cdot \l} \F_i^{(\l_i)} \1_\l \cong \bigoplus_{j \ge 0} \bigoplus_{\qbins{\l_i}{j}} \F_i^{(\l_i - j)} \E_i^{(\l_i-j)} \1_\l 
\end{equation}
which means that there is a natural map $\1_\l \rightarrow \E_i^{(\l_i)} \F_i^{(\l_i)} \1_\l$ corresponding to the inclusion of the unique summand $\1_\l$ (which occurs when $j = \l_i$ in the summation above). The composition gives us a map $\1_\l \rightarrow \T_i^2 \1_\l$. If $\l_i \le 0$ then the same argument (with the roles of $\E$s and $\F$s switched) also gives such a map. 

Now using (\ref{eq:Tomega}) we have 
\begin{equation}\label{eq:0}
\T_\omega^2 \cong 
[(\T_{n-1})(\T_{n-2}\T_{n-1}) \dots (\T_1 \dots \T_{n-1})][(\T_{n-1} \dots \T_1) \dots (\T_{n-1} \T_{n-2})(\T_{n-1})].
\end{equation}
So repeatedly using the maps $\1_\l \rightarrow \T_i^2 \1_\l$ defined above we obtain a map $\1_\l \rightarrow \T_\omega^2 \1_\l$. 

\subsection{Convergence of $\lim_{\ell \rightarrow \infty} \T_\omega^{2 \ell}$}

\begin{Proposition}\label{prop:1} If $\g = \sl_n$ then we have 
\begin{align*}
\T_\omega^2 \1_\l \F_i \cong \F_i \T_\omega^2 \1_{\l+\alpha_i} \la -2(\l_i+2)\ra [2] \hspace{.5cm} & \text{ if } \l_i \ge 0 \\
\T_\omega^2 \1_\l \E_i \cong \E_i \T_\omega^2 \1_{\l-\alpha_i} \la -2(-\l_i+2)\ra [2] \hspace{.5cm} & \text{ if } \l_i \le 0.
\end{align*}
\end{Proposition}
\begin{proof}
We will prove the first assertion by induction on $n$ (the case $\l_i \le 0$ follows similarly). To emphasize the dependence on $n$ we write $\omega_n$ instead of $\omega$. The base case follows from Corollary \ref{cor:0}. To apply induction we use 
$$\T_{\omega_{n+1}} = (\T_1 \dots \T_{n}) \T_{\omega_n} = \T_{\omega_n} (\T_{n} \dots \T_1).$$
to obtain 
\begin{align*}
\T_{\omega_{n+1}}^2 \1_\l \F_i 
& \cong (\T_1 \dots \T_{n}) \T_{\omega_n}^2 (\T_{n} \dots \T_1) \1_\l \F_i \\
& \cong (\T_1 \dots \T_{n}) \T_{\omega_n}^2 (\T_{n} \dots \T_i \T_{i-1} \F_i \1_\mu \T_{i-2} \dots \T_1) \\
& \cong (\T_1 \dots \T_{n}) \T_{\omega_n}^2 (\T_{n} \dots \1_{\mu'} \T_i \F_{i-1,i} \T_{i-1} \dots \T_1) [s_1]\la -s_1 \ra \\
& \cong (\T_1 \dots \T_{n}) \T_{\omega_n}^2 (\T_{n} \dots \F_{i-1} \T_i \dots \T_1) [s_1+s_2] \la -s_1-s_2 \ra \\
& \cong (\T_1 \dots \T_{n}) \T_{\omega_n}^2 \1_{\l'} \F_{i-1} (\T_{n} \dots \T_1) [s_1+s_2] \la -s_1-s_2 \ra 
\end{align*}
where, by Lemma \ref{lem:2}, we have
\begin{itemize}
\item $s_1=-1$ if $\mu_{i-1} \le 0$ and $s_1=0$ otherwise,
\item $s_2=1$ if $\mu'_i > 0$ and $s_2=0$ otherwise.
\end{itemize}
Now, one can check that $\mu_{i-1} = -1 + \sum_{k=1}^{i-1} \l_k$ and $\mu'_i = - \sum_{k=1}^i \l_k$. Moreover, $\l'_{i-1} = \l_i \ge 0$ so by induction we have $\T_{\omega_n}^2 \1_{\l'} \F_{i-1} = \F_{i-1} \T_{\omega_n}^2 \la -2(\l_i+2) \ra [2]$. To finish off the calculation we note that
\begin{align*}
\1_\l (\T_1 \dots \T_{n}) \F_{i-1} 
&\cong \1_\l \T_1 \dots \T_{i-1} \T_i \1_\nu \F_{i-1} \T_{i+1} \dots \T_{n} \\
&\cong \1_\l \T_1 \dots \1_{\nu'} \T_{i-1} \F_{i,i-1} \T_i \dots \T_{n} [s_3]\la -s_3 \ra \\
&\cong \1_\l \T_1 \dots \T_{i-2} \F_i \T_{i-1} \dots \T_{n} [s_3+s_4]\la -s_3-s_4 \ra \\
&\cong \1_\l \F_i (\T_1 \dots \T_{n}) [s_3+s_4]\la -s_3-s_4 \ra
\end{align*}
where
\begin{itemize}
\item $s_3=1$ if $\nu_i \ge 0$ and $s_3=0$ otherwise, 
\item $s_4=-1$ if $\nu'_{i-1} > 0$ and $s_4=0$ otherwise.
\end{itemize}
A similar calculation as before shows that $\nu_i = -1-\sum_{k=1}^i \l_k$ while $\nu'_{i-1} = \sum_{k=1}^{i-1} \l_k$. The key point is that everything works out so that $s_1+s_4=-1$ and $s_2+s_3=1$ regardless of $\l$. Thus $s_1+s_2+s_3+s_4=0$ and hence $\T_{\omega_{n+1}}^2 \1_\l \F_i = \F_i \T_{\omega_{n+1}}^2 \la -2(\l_i+2) \ra [2]$ which completes our induction. 

Note that the case $i=1$ in the argument above is special. However, by symmetry, this is the same as the case $i=n$ where the argument works. 
\end{proof}

\begin{Corollary}\label{cor:1} 
Suppose $\g = \sl_n$ and fix $p \ge 0$. Then we have 
\begin{align*}
\T_\omega^2 \1_\l \F_i^{(p)} \cong \F_i^{(p)} \T_\omega^2 \1_{\l+p\alpha_i} \la -2p(\l_i+p+1)\ra [2p] \hspace{.5cm} & \text{ if } \l_i \ge 0 \\
\T_\omega^2 \1_\l \E_i^{(p)} \cong \E_i^{(p)} \T_\omega^2 \1_{\l-p\alpha_i} \la -2p(-\l_i+p+1)\ra [2p] \hspace{.5cm} & \text{ if } \l_i \le 0. 
\end{align*}
\end{Corollary}
\begin{proof}
This follows by applying Proposition \ref{prop:1} repeatedly. 
\end{proof}
Let us denote by $\U_i \1_\mu = \1_{s_i \cdot \mu} \U_i$ the map $\E_i^{(-\mu_i)} \1_\mu$ if $\mu_i \le 0$ and $\F_i^{(\mu_i)} \1_\mu$ if $\mu_i \ge 0$. 

\begin{Lemma}\label{lem:2.5}
If $\g = \sl_n$ and $|i-j|=1$ then $\T_i \T_j \U_i \1_\l \cong \U_j \T_i \T_j \1_\l$.  
\end{Lemma}
\begin{proof}
Suppose $\l_i \ge 0$ (the case $\l_i \le 0$ is similar). We will show that $\T_i \T_j \F_i^{\l_i} \1_\l \cong \F_j^{\l_i} \T_i \T_j \1_\l$ from which the result follows. Using Lemma \ref{lem:2} repeatedly we find that 
\begin{equation}\label{eq:s1}
\T_j \F_i^{\l_i} \1_\l \cong \F_{ji}^{\l_i} \T_j \1_\l [s] \la -s \ra
\end{equation}
where $s = \# \{w \in \{\l_j, \l_j+1, \dots, \l_j+r-1\}: w < 0\}$. Applying Lemma \ref{lem:2} again we also find that 
\begin{equation}\label{eq:s2}
\T_i \F_{ij}^{\l_i} \1_{s_j \cdot \l} \cong \F_j^{\l_i} \T_i \1_{s_j \cdot \l} [-s_2] \la s_2 \ra
\end{equation}
where $s_2 = \# \{w \in \{-\l_j, -\l_j-1, \dots, -\l_j-r+1 \}: w > 0\}$. Clearly $s_1=s_2$ and so, combining (\ref{eq:s1}) and (\ref{eq:s2}), the result follows. 
\end{proof}

\begin{Corollary}\label{cor:2} 
If $\g = \sl_n$ then we have $\T_\omega \U_i \1_\l \cong \U_{n-i} \T_\omega \1_\l$ and subsequently 
\begin{equation}\label{eq:3}
\T_\omega^2 \U_i \1_\l \cong \U_i \T_\omega^2 \1_\l.
\end{equation}
\end{Corollary}
\begin{proof}
We use induction, as in the proof of Proposition \ref{prop:1}. Suppose $\l_i \ge 0$ so that $\U_i \1_\l \cong \F_i^{(\l_i)} \1_\l$ (the case $\l_i \le 0$ is similar). Then 
\begin{align*}
\T_{\omega_{n+1}} \U_i \1_\l  
&\cong (\T_1 \dots \T_n) \T_{\omega_n} \U_i \1_\l \\
&\cong (\T_1 \dots \T_n) \U_{n-i} \T_{\omega_n} \1_\l \\
&\cong (\T_1 \dots \T_{n-i} \T_{n-i+1} \U_{n-i} \T_{n-i+2} \dots \T_n) \T_{\omega_n} \1_\l \\
&\cong (\T_1 \dots \U_{n-i+1} \T_{n-i} \T_{n-i+1} \dots \T_n) \T_{\omega_n} \1_\l \\
&\cong \U_{n+1-i} (\T_1 \dots \T_n) \T_{\omega_n} \1_\l
\end{align*}
where the second line follows by induction and the fourth from Lemma \ref{lem:2.5}. 
\end{proof}

Now let us denote by $\R \1_\l := \Cone(\1_\l \rightarrow \T_\omega^2 \1_\l)$. 

\begin{Proposition}\label{prop:2} 
If $\ell \ge 0$ then $\T_\omega^{2 \ell} \R \1_\l$ and $\R \T_\omega^{2 \ell} \1_\l$ are supported in homological degrees $\le -2 \ell$.
\end{Proposition}
\begin{proof}
We deal with the complex $\T_\omega^{2\ell} \R \1_\l$ (the proof for the other complex is the same). The idea is to use the expression for $\1_\l \T^2_\omega$ from (\ref{eq:0}) to study $\T_\omega^{2 \ell} \1_\l \T^2_\omega$. Consider the left most term $\1_\l \T_{n-1}$ and suppose $\l_{n-1} \le 0$. Then $\1_\l \T_{n-1}$ is given by a complex 
$$\dots \rightarrow \1_\l \E_{n-1}^{(s)} \F_{n-1}^{(-\l_{n-1}+s)} \la -s \ra \rightarrow \1_\l \E_{n-1}^{(s-1)} \F_{n-1}^{(-\l_{n-1}+s-1)} \la -s+1 \ra \rightarrow \dots \rightarrow \1_\l \F_{n-1}^{(-\l_{n-1})}.$$
Then by Corollary \ref{cor:1} we have 
\begin{equation}\label{eq:6}
\T_\omega^{2 \ell} \1_\l \E_{n-1}^{(s)} \F_{n-1}^{(-\l_{n-1}+s)} \cong \E_{n-1}^{(s)} \T_\omega^{2 \ell} \1_{\l - s\alpha_{n-1}} \F_{n-1}^{(-\l_{n-1}+s)} \la -2s \ell(-\l_{n-1}+s+1) \ra [2s \ell]
\end{equation}
which is a complex supported in homological degrees $\le -2s \ell$. Thus only the term $\T_\omega^{2 \ell} \1_\l \F_{n-1}^{(-\l_{n-1})}$ can contribute to the cohomology in degrees $\ge -2\ell$. In this case, by Corollary \ref{cor:2}, we have 
\begin{equation}\label{eq:8}
\T_\omega^{2 \ell} \1_\l \F_{n-1}^{(\l_{n-1})} \cong \F_{n-1}^{(-\l_{n-1})} \T_\omega^{2 \ell} \1_\l.
\end{equation}
Now we repeat this argument with the other $\T$s in the expression for $\1_\l \T_\omega^2$ and conclude that the only terms in $\1_\l \T^2_\omega$ which can contribute something in cohomology degrees $\ge -2 \ell$ in $\T_\omega^{2 \ell} \1_\l \T^2_\omega$ is 
$$[(\U_{n-1})(\U_{n-2}\U_{n-1}) \dots (\U_1 \dots \U_{n-1})][(\U_{n-1} \dots \U_1) \dots (\U_{n-1} \U_{n-2})(\U_{n-1})] \1_\l \T_\omega^{2 \ell}.$$ 
Using Corollary \ref{cor:2} we can rewrite this as 
\begin{equation}\label{eq:2}
[(\U_{n-1})(\U_{n-2}\U_{n-1}) \dots (\U_1 \dots \U_{n-2})] \T_\omega^{2\ell} \U_{n-1}^2 [(\U_{n-2} \dots \U_1) \dots (\U_{n-1} \U_{n-2})(\U_{n-1})] \1_\l.
\end{equation}
Consider the middle factor $\U_{n-1}^2 \1_\mu$ in (\ref{eq:2}) where $\mu = s_{n-1} s_{\omega} \cdot \l$. Let us suppose $\mu_{n-1} \ge 0$ (the case $\mu_{n-1} \le 0$ is the same). Then 
\begin{eqnarray*}
\U_{n-1} \U_{n-1} \1_\mu 
&\cong& \E_{n-1}^{(\mu_{n-1})} \F_{n-1}^{(\mu_{n-1})} \1_\mu \\
&\cong& \bigoplus_{j \ge 0} \bigoplus_{\qbins{\mu_{n-1}}{j}} \1_\mu \F_{n-1}^{(\mu_{n-1}-j)} \E_{n-1}^{(\mu_{n-1}-j)} \1_\mu.
\end{eqnarray*}
Now, by the same argument as above (using Corollary \ref{cor:1}), 
$$\T_\omega^{2 \ell} \1_\mu \F_{n-1}^{(\mu_{n-1}-j)} \E_{n-1}^{(\mu_{n-1}-j)} \cong 
\F_{n-1}^{(\mu_{n-1}-j)} \T_\omega^{2 \ell} \E_{n-1}^{(\mu_{n-1}-j)} \1_\mu \la -2p\ell (\mu_{n-1}+p+1)\ra [2p\ell]$$
where $p = \mu_{n-1}-j$. This is supported in cohomological degrees $\le -2\ell$ unless $j = \mu_{n-1}$ which leaves us with one copy of the identity. Thus the only terms in (\ref{eq:2}) which could contribute to cohomological degrees $> -2\ell$ come from 
$$[(\U_{n-1})(\U_{n-2}\U_{n-1}) \dots (\U_1 \dots \U_{n-2})] \T_\omega^{2 \ell} \1_\mu [(\U_{n-2} \dots \U_1) \dots (\U_{n-1} \U_{n-2})(\U_{n-1})] \1_\l.$$
Repeating this argument with $\U_{n-2}, \U_{n-3}$ and so on, we are left with just one term, namely $\T_\omega^{2\ell} \1_\l$. But this term in $\T_\omega^{2\ell} \R \1_\l = \T_\omega^{2\ell} \Cone(\1_\l \rightarrow \T_\omega^2 \1_\l)$ is zero since it is exactly the one that cames from the term $\1_\l$ inside the complex $\T_\omega^2 \1_\l$. Thus $\T_\omega^{2\ell} \R \1_\l$ is supported in homological degrees $\le -2\ell$. 
\end{proof}

Proposition \ref{prop:2} implies that $\lim_{\ell \rightarrow \infty} \T_\omega^{2\ell} \R \1_\l = 0$. This means that $\P^- := \lim_{\ell \rightarrow \infty} \T_\omega^{2\ell}$ exists (the negative in $\P^-$ indicates that the complex is bounded above). Notice that the map $\1_\l \rightarrow \T_\omega^2 \1_\l$ induces a map $\1_\l \rightarrow \P^- \1_\l$. 

\begin{Proposition}\label{prop:idempotent}
If $\g = \sl_n$ then $\P^- \in \Kom^-(\K)$ is an idempotent, meaning that $\P^- \P^- \cong \P^-$. 
\end{Proposition}
\begin{proof}
Consider the map $\phi: \P^- \1_\l \rightarrow \P^- \P^- \1_\l$ induced by $\1_\l \rightarrow \P^- \1_\l$. The cone of this map is (by definition) $\P^- \R \1_\l \cong 0$. Thus $\phi$ is an isomorphism. 
\end{proof}

\subsection{Convergence of $\P^-$ in K-theory}\label{sec:Kconverge}
We now show that $\P^- \1_\l$ converges in K-theory to give a well defined element $p(\P^- \1_\l) \in \hK(\K)$. 

Choose a basis $b_1, \dots, b_s$ of morphisms in $K(\K)$ (this space is finite dimensional by assumption). As before, if $A$ is a morphism in $K(\K)$ then we write $A = \sum_{i,j} a_{ij} q^i b_j$ and denote $\la A \ra_q = {\rm min}\{i: a_{ij} \ne 0 \}$. For a 1-morphism $\A \in \Kom(\K)$ we will write $\la \A \ra_q \ge e$ if $\A$ is homotopic to a complex $\A^\bullet$ where $\la \A^s \ra_q \ge e$ for all $s$. Likewise, we write $\la \A \ra_q^{\le N} \ge e$ if $\la \A^s \ra_q \ge e$ for all $s \le N$. 

\begin{Lemma}\label{lem:new1}
There exists some $L$, independent of $\ell$, so that $\la \T_\omega^{2 \ell} \R \ra_q \ge 4 \ell + L$ for all $\ell \ge 1$. 
\end{Lemma}
\begin{proof}
To study $\T_\omega^{2\ell} \R \1_\l$ we proceed as in the proof of Proposition \ref{prop:2}. Namely we first consider $\T_\omega^{2\ell} \1_\l \T_{n-1}$ and assume $\l_{n-1} \le 0$ (the case $\l_{n-1} \ge 0$ is similar). This complex is made up of terms of the form $\T_\omega^{2\ell} \1_\l \E_{n-1}^{(s)} \F_{n-1}^{(- \l_{n-1} + s)}$. Moving the $\E_{n-1}^{(s)}$ term over to the other side of $\T_\omega^{2\ell}$ we get equation $\T_\omega^{2 \ell} \T_\omega^2 \1_\l$ is homotopic to
\begin{equation}\label{eq:7}
\fC^s := \E_{n-1}^{(s)} \T_\omega^{2\ell} \F_{n-1}^{(-\l_{n-1}+s)} \T_{n-1}^{-1} \T_\omega^2 \la -2s\ell(-\l_{n-1}+s+1) \ra [2s\ell].
\end{equation}
Now suppose that we know by induction that $\la \T_\omega^{2\ell} \ra_q \ge m_\ell$ for some $m_\ell$. Let $M$ be the minimum over $\la \E_{n-1}^{(s)} \B_i \F_{n-1}^{(-\l_{n-1}+s)} \T_{n-1}^{-1} \T_\omega^2 \ra_q$ where $\B_i$ is a representative of $b_i$ in $\K$. Then
$$\la \fC^s \ra_q \ge 2s \ell (s+1) + m_\ell + M.$$
The right side is at least $4\ell+m_\ell+M$ unless $s=0$ in which case we get 
$$\F_{n-1}^{(-\l_{n-1})} \T_\omega^{2\ell} \T_{n-1}^{-1} \T_\omega^2 \cong \U_{n-1} \T_\omega^{2\ell} \T_{n-2} \T_{n-2}^{-1} \T_{n-1}^{-1} \T_\omega^2$$ 
and we repeat the argument above. In this way we eventually end up with 
$$[(\U_{n-1})(\U_{n-2}\U_{n-1}) \dots (\U_1 \dots \U_{n-2})]\T_\omega^{2\ell} \U_{n-1}^2 [(\U_{n-2} \dots \U_1) \dots (\U_{n-1} \U_{n-2})(\U_{n-1})].$$
We now simplify $\U_{n-1}^2$ and repeat the argument above just like in the proof of Proposition \ref{prop:2}. Thus we find that $\la \T_\omega^{2 \ell} \R \ra_q \ge 4\ell+m_\ell+M$. This means that, since $\R \1_\l \cong \Cone (\1_\l \rightarrow \T_\omega^2 \1_\l)$, $\la \T_\omega^{2(\ell+1)} \ra_q \ge {\rm min}(4\ell+m_\ell+M,m_\ell)$. So, if $\ell \ge -\frac{M}{4}$, we have $\la \T_\omega^{2(\ell+1)} \ra_q \ge m$ for some $m$. This implies $\la \T_\omega^{2\ell} \R \ra_q \ge 4 \ell + m + M$ and the result follows.  
\end{proof}

Now, given $\ell$ choose $N_\ell$ so that $\T_\omega^{2 \ell}$ is supported in degrees $> N_\ell$.  Using the exact triangle $\T_\omega^{2(r-1)} \rightarrow \T_\omega^{2r} \rightarrow \T_\omega^{2(r-1)} \R$ we find that if $\la \T_\omega^{2(r-1)} \ra_q^{\le N_\ell} \ge e_{r-1}$ for some $e_{r-1}$ then 
$$\la \T_\omega^{2r} \ra^{\le N_\ell}_q \ge {\rm min}(e_{r-1}, \la \T_\omega^{2r} \R \ra^{\le N_\ell}_q) \ge {\rm min}(e_{r-1}, 4r + L).$$
 Applying these inequalities repeatedly we find that for any $r \gg 0$ we have 
$$\la \T_\omega^{2r} \ra^{\le N_\ell}_q \ge {\rm min}(e_{\ell}, 4 \ell+L)$$
But since $\T_\omega^{2 \ell}$ is supported in degrees $> N_\ell$ we can take $e_\ell = \infty$ and so $\la \T_\omega^{2r} \ra_q^{\le N_\ell} \ge 4 \ell + L$ for any $r \gg 0$. Hence $\la \P^- \ra_q^{\le N_\ell} \ge 4 \ell+L$ for any $\ell$. Thus $\P^- \in \Kom^-_*(\K)$ which completes the proof of Theorem \ref{thm:main1}. 

\section{Categorified clasps}\label{sec:clasps}

In this section we prove Theorem \ref{thm:main2}. 

Denote by $\k^m$ the standard representation of $U_q(\sl_m)$ with basis $v_1, \dots, v_m$. The vector space 
$$\Lambda_q(\k^m) = \k[q,q^{-1}]\la v_1, \dots, v_m \ra/(v_i^2, v_iv_j + qv_jv_i \text{ for } i<j)$$
is the standard wedge product representation of $U_q(\sl_m)$. 

Now consider the $U_q(\sl_m)$-module $\Lambda_q^{\i}(\k^m) := \Lambda_q^{i_1}(\k^m) \otimes \dots \otimes \Lambda_q^{i_n}(\k^m)$ where ${\i}$ is a sequence of integers $0 \le i_1, \dots, i_n \le m$. In terms of highest weight representations $\Lambda_q^{i_k}(\k^m) = V_{\Lambda_{i_k}}$. As a $U_q(\sl_m)$-module $\Lambda_q^{\i}(\k^m)$ has a unique summand isomorphic to the highest weight representation $V_{\i} := V_{\sum_k \Lambda_{i_k}}$. We denote by 
$$\iota: V_{\i} \rightarrow \Lambda_q^{\i}(\k^m) \text{ and } 
\pi: \Lambda_q^{\i}(\k^m) \rightarrow  V_{\i}$$ 
the natural inclusion and projection. Their composition is denoted 
$$P \1_{\i} := \iota \circ \pi \in \End_{U_q(\sl_m)}(\Lambda_q^{\i}(\k^m)).$$ 
Notice $\pi \circ \iota \in \End_{U_q(\sl_m)}(V_{\i})$ is a multiple of the identity since $V_{\i}$ is irreducible. So $(P \1_{\i})^2$ is a multiple of $P \1_{\i}$ and we uniquely rescale $P \1_{\i}$ so that $(P \1_{\i})^2 = P \1_{\i}$. Following \cite{Kup} we refer to the idempotent $P \1_{\i}$ as a {\em clasp}. 

\subsection{Skew Howe duality}\label{sec:Howe}

Our aim now is to understand clasps using skew Howe duality, {\it i.e.} by studying $\Lambda_q(\k^m \otimes \k^n)$. Formally, this algebra is the quadratic dual of the quantum algebra $M_q(m \times n)$ which is the quantum analogue of the algebra of $m \times n$ matrices (see \cite{Man}, in particular section 8.9). There exist two isomorphisms 
$$\Lambda_q(\k^m)^{\otimes n} \longleftarrow \Lambda_q(\k^m \otimes \k^n) \longrightarrow \Lambda_q(\k^n)^{\otimes m}.$$
For our purposes, we only identify the composition isomorphism 
\begin{equation}\label{eq:4}
\Lambda_q(\k^m)^{\otimes n} \longrightarrow \Lambda_q(\k^n)^{\otimes m}
\end{equation}
(as $\k[q,q^{-1}]$-modules) as follows. Having fixed the basis $v_1, \dots, v_m$ for $\k^m$ the left side of (\ref{eq:4}) has basis $\{ v_{\i_1} \otimes \dots \otimes v_{\i_n} \}$ where each $\i_k$ is a sequence $1 \le s_1 < \dots < s_l \le m$ and $v_{\i_k} = v_{s_1} \wedge \dots \wedge v_{s_l}$. Likewise, the right side of (\ref{eq:4}) has basis $\{ w_{{\j}_1} \otimes \dots \otimes w_{{\j}_m} \}$ where $w_1, \dots, w_n$ is a basis of $\k^n$. Then the map in (\ref{eq:4}) is given by 
$$v_{\i_1} \otimes \dots \otimes v_{\i_n} \mapsto (-1)^{\#\{(a,b,k_1,k_2): a \in \i_{k_1}, b \in \i_{k_2}, a < b, k_1 < k_2 \}} w_{{\j}_1} \otimes \dots \otimes w_{{\j}_m}$$
where $\ell \in \j_{\ell'}$ if and only if $\ell' \in \i_\ell$.

\begin{Lemma}\label{lem:commute} The actions of $U_q(\sl_m)$ and $U_q(\sl_n)$ on $\Lambda_q(\k^m)^{\otimes n} \cong \Lambda_q(\k^n)^{\otimes m}$ commute.
\end{Lemma}
\begin{proof}
It is enough to consider the root $U_q(\sl_2)$ subalgebras of $U_q(\sl_m)$ and $U_q(\sl_n)$. So we can assume $m=n=2$ and that case can be checked by an explicit calculation. 
\end{proof} 

For $N \in \N$ denote by $\Lambda^N_q(\k^m)^{\otimes n}$ the $N$-graded piece of $\Lambda_q(\k^m)^{\otimes n}$ where $\deg(v_i)=1$. The action of $U_q(\sl_n)$ preserves this piece. The decomposition of $\Lambda^N_q(\k^m)^{\otimes n}$ into weight spaces is given by $\oplus_{\i} \Lambda_q^{\i}(\k^m)$ where the direct sum is over all $0 \le i_1, \dots, i_n \le m$ with $\sum_j i_k = N$. 

Under this action we have 
$$E_k: \Lambda_q^{\i}(\k^m) \rightarrow \Lambda_q^{\i + \alpha_k}(\k^m) \text{ and } F_k: \Lambda_q^{\i + \alpha_k}(\k^m) \rightarrow \Lambda_q^{\i}(\k^m)$$
where $\alpha_k = (0, \dots, 0, -1, 1, 0, \dots, 0)$ with the nonzero entries in the $k$ and $k+1$ spot. 
The dominant weights correspond to those ${\i}$ where $0 \le i_1 \le \dots \le i_n \le m$ (in this case we call ${\i}$ dominant). The notation $\1_{\i}$ indicates the projection onto this weight space. 

\begin{Proposition}\label{prop:Punique} If $\i$ is a dominant weight then $P \1_{\i} F_k = 0$ for $k=1, \dots, m-1$. Moreover, $P \1_{\i}$ is the unique such (nonzero) projection in $\End_{U_q(\sl_m)}(\Lambda_q^{\i}(\k^m))$. 
\end{Proposition}
\begin{proof}
Since $\Lambda_q^{\i+\alpha_k}(\k^m)$ does not contain $V_{\i}$ it follows that $P \1_{\i} F_k = 0$ for any $k=1, \dots, m-1$.

On the other hand, $\Lambda_q^{{\i}}(\k^m)$ breaks up as a direct sum $V_{\i} \oplus V'$ for some $U_q(\sl_m)$-module $V'$. Moreover, it is spanned by vectors $F_k(v)$ and highest weight vectors. By Lemma \ref{lem:3} the highest weight vectors span precisely $V_{\i}$. This means that if $P' \1_{\i}$ is a projector such that $P' \1_{\i} F_k = 0$ for all $k$ then it must either be zero or it must project onto $V_{\i}$ (in which case $P'=P$). 
\end{proof}

\begin{Lemma}\label{lem:3} The $U_q(\sl_m)$-submodule $V_{\i} \subset \Lambda_q^{\i}(\k^m)$ coincides with the vector space of highest weight vectors for the action of $U_q(\sl_n)$. 
\end{Lemma}
\begin{proof}
In order to prove this it suffices to consider the case $q=1$. Now, skew Howe duality says that, as an $(\sl_m,\sl_n)$-bimodule, we have 
$$\Lambda^N(\k^m \otimes \k^n) \cong \oplus_{|{\i}|=N} V_{{\i}^\vee} \boxtimes V_{\i}$$
where the sum is over all partitions (or equivalently Young diagrams) ${\i}$ of size $N$ which fit in an $n \times m$ box and ${\i}^\vee$ denotes the dual Young diagram (obtained by flipping about a diagonal). In particular, this means that the highest weight vectors of the isotypic component $V_{{\i}^\vee}$ for the action of $\sl_n$ is spanned by $V_{\i}$. 
\end{proof}

\subsection{The action of $\t_\omega$}

Consider again the action of $U_q(\sl_n)$ on $\Lambda_q^N(\k^m \otimes \k^n)$ described above. The decategorification of $\T_k \1_\i$ gives 
\begin{equation}\label{eq:cpxdecat}
\t_k \1_{\i} := \sum_{s \ge 0} (-q)^s E_k^{(- \la {\i}, \alpha_k \ra + s)} F_k^{(s)} \1_{\i} 
\hspace{1cm} \text{ or } \hspace{1cm}
\t_k \1_{\i} := \sum_{s \ge 0} (-q)^s F_k^{(\la {\i}, \alpha_k \ra + s)} E_k^{(s)} \1_{\i} 
\end{equation}
depending on whether $\la {\i}, \alpha_k \ra \le 0$ or $\la {\i}, \alpha_k \ra \ge 0$ (recall that the shift $\la 1 \ra$ decategorifies to multiplication by $q^{-1}$). Here the pairing $\la \cdot, \cdot \ra$ is given by the usual dot product on ${\i} = (i_1, \dots, i_n)$ and $\alpha_k = (0, \dots, -1,1, \dots, 0)$. Note that $\t_k \1_\i = \1_{s_k \cdot \i} \t_k$ where $s_k$ acts on ${\i}$ by switching $i_k$ and $i_{k+1}$. We define $\t_\omega$, just like $\T_\omega$ as 
$$\t_\omega \1_{\i} := (\t_{n-1} \dots \t_1)(\t_{n-1} \dots \t_2) \dots (\t_{n-1} \t_{n-2})(\t_{n-1}) \1_{\i}.$$

\begin{Proposition}\label{prop:3} Suppose ${\i}$ is dominant and let $v$ be a highest weight vector such that $F_{k_1}^{(p_1)} \dots F_{k_j}^{(p_j)}(v)$ has weight $\i$. Then 
\begin{equation}\label{eq:5}
\t_\omega^2 F_{k_1}^{(p_1)} \dots F_{k_j}^{(p_j)}(v) = q^{\la {\i'} + {\i}, {\i'} - {\i} \ra + 2 \sum_l p_l} F_{k_1}^{(p_1)} \dots F_{k_j}^{(p_j)}(v)
\end{equation}
where ${\i'}$ is the weight of $v$. 
\end{Proposition}
\begin{proof}
The decategorification of Corollary \ref{cor:1} states that 
$$\t_\omega^2 \1_{\i} F_k^{(p)} \cong q^{2p(\la {\i}, \alpha_k \ra + p + 1)} F_k^{(p)} \t_\omega^2 \1_{\i}.$$ 
Applying this repeatedly and keeping track of the powers of $q$ we find that the exponent of $q$ is 
$$2p_1(\la {\i}, \alpha_{k_1} \ra + p_1 + 1) + 2p_2(\la {\i} + p_1 \alpha_{k_1}, \alpha_{k_2} \ra + p_2 + 1) + \dots +  2p_j(\la {\i} + p_1 \alpha_{k_1} + \dots + p_{j-1} \alpha_{k_{j-1}}, \alpha_{k_j} \ra + p_j+1)$$
which simplifies to 
$$2 \sum_l p_l \la {\i}, \alpha_{k_l} \ra + 2 \sum_{1 \le a < b \le j} \la p_a \alpha_{k_a}, p_b \alpha_{k_b} \ra + 2 \sum_l (p_l^2+p_l).$$
Now using that ${\i'} - {\i} = \sum_l p_l \alpha_{k_l}$ it is easy to check that this simplifies to give
$$2 \la {\i}, {\i'} - {\i} \ra + \la {\i'} - {\i}, {\i'} - {\i} \ra + 2 \sum_l p_l $$
which is the same as the exponent of $q$ in (\ref{eq:5}). 

Finally, to complete the proof one shows that $\t_\omega^2 v = v$ since $v$ is a highest weight vector. To see this first note that if $E_k(w)=0$ (resp. $F_k(w)=0$) then $\t_k(w) = U_k(w)$ and $F_k U_k(w) = 0$ (resp. $E_k U_k(w) = 0$). Here $U_k$ is the decategorification of $\U_k$, namely $U_k(w)$ equals 
$$E_k^{(-\la \wt(w), \alpha_k \ra)}(w) \text{ if } \la \wt(w), \alpha_k \ra \le 0 \hspace{.5cm} \text{ and } \hspace{.5cm} F_k^{(\la \wt(w), \alpha_k \ra)}(w)  \text{ if } \la \wt(w), \alpha_k \ra \ge 0$$
where $\wt(w)$ denotes the weight of $w$. Thus we get:
$$\t_\omega^2(v) = [(U_{n-1})(U_{n-2} U_{n-1}) \dots (U_1 \dots U_{n-1})][(U_{n-1} \dots U_1) \dots (U_{n-1} U_{n-2})(U_{n-1})](v).$$
Now consider the middle two terms $U_{n-1} U_{n-1} (w)$ where $w = (U_{n-2} \dots U_1) \dots (U_{n-1} U_{n-2})(U_{n-1})(v)$. By the argument above we know either $E_{n-1}(w) = 0$ or $F_{n-1}(w) = 0$. Let us suppose $E_{n-1}(w)=0$ (the other case is the same). Then 
\begin{eqnarray*}
U_{n-1} U_{n-1} (w)  
&=& E_{n-1}^{(\la \wt(w), \alpha_{n-1} \ra)} F_{n-1}^{(\la \wt(w), \alpha_{n-1} \ra)} (w) \\
&=& \sum_{j \ge 0} {\qbins{\la \wt(w), \alpha_{n-1} \ra}{j}} F_{n-1}^{(\la \wt(w), \alpha_{n-1} \ra-j)} E_{n-1}^{(\la \wt(w), \alpha_{n-1} \ra-j)} (w) 
\end{eqnarray*}
Since $E_{n-1}(w)=0$ all these terms vanish except for the one term when $j = \la \wt(w), \alpha_{n-1} \ra$. Thus we get $U_{n-1} U_{n-1} (w) = w$. Repeating this way we obtain $\t_\omega^2(v) = v$. 
\end{proof}

\begin{Corollary}\label{cor:Punique}
Suppose ${\i}$ is a dominant weight. The clasp $P \1_{\i}$ is the unique (nonzero) idempotent in $\End(\Lambda_q^{\i}(\k^m))$ which satisfies $\t_\omega^2 P = P = P \t_\omega^2$.
\end{Corollary}
\begin{proof}
The vector space $\Lambda_q^{\i}(\k^m)$ is spanned by vectors $v \in V_{\i} \subset \Lambda_q^{\i}(\k^m)$ (which are highest weight vectors) and vectors of the form $F_{k_1}^{(p_1)} \dots F_{k_j}^{(p_j)}(v)$ where $v$ is a highest weight vector. 

In the first case we have $P(v)=v$ and $\t_\omega^2(v)=v$ by Proposition \ref{prop:3}. In the second case we have $P F_{k_1}^{(p_1)} \dots F_{k_j}^{(p_j)}(v) = 0$ by Proposition \ref{prop:Punique} and likewise $P \t_\omega^2 F_{k_1}^{(p_1)} \dots F_{k_j}^{(p_j)}(v) = 0$ by Proposition \ref{prop:3}. This shows that $\t_\omega^2 P = P = P \t_\omega^2$. 

On the other hand, suppose $P'$ is another projection which satisfies $\t_\omega^2 P' = P' = P' \t_\omega^2$. Then 
\begin{eqnarray*}
P' F_{k_1}^{(p_1)} \dots F_{k_j}^{(p_j)}(v) 
&=& P' \t_\omega^2 F_{k_1}^{(p_1)} \dots F_{k_j}^{(p_j)}(v) \\
&=& P' q^{\la {\i'} + {\i}, {\i'} - {\i} \ra + 2 \sum_l p_l} F_{k_1}^{(p_1)} \dots F_{k_j}^{(p_j)}(v)
\end{eqnarray*}
by Proposition \ref{prop:3}. Now $\la {\i'}, {\i'} - {\i} \ra \ge 0$ and $\la {\i}, {\i'} - {\i} \ra \ge 0$ since $\i, \i'$ are dominant weights and $\i'-\i = \sum_l p_l \alpha_{i_l}$ where $p_l > 0$. Thus the exponent of $q$ is positive and we conclude $P' F_{k_1}^{(p_1)} \dots F_{k_j}^{(p_j)}(v)=0$. Then $P' \1_{\i} F_k = 0$ for any $k$ and hence $P' \1_{\i} = P \1_{\i}$ by Proposition \ref{prop:Punique}. 
\end{proof} 

Recall that $\P^- \1_{\i} \in \Kom^-(\K)$ is an idempotent (Proposition \ref{prop:idempotent}) satisfying $\T_\omega^2 \P^- \1_{\i} \cong \P^- \1_{\i} \cong \T_\omega^2 \P^- \1_{\i}$ (essentially by definition). It follows that if $\i$ is a dominant weight then $\P^- \1_{\i}$ categorifies $P \1_\i$ in the sense that $p(\P^- \1_{\i}) = P \1_{\i} \in \hK(\K)$. 

If $\i$ is not dominant let $\sigma \in S_n$ be a permutation such that $\sigma \cdot \i$ is dominant and consider its lift $\t_\sigma$ to the braid group. Now $\t_\sigma^{-1} P \1_\i \t_\sigma$ is idempotent and $\t_\sigma, \t_\sigma^{-1}$ commute with $\t_\omega^2$ which means 
$$\t_\omega^2 (\t_\sigma^{-1} P \1_\i \t_\sigma) \1_{\sigma \cdot \i}
\cong \t_\sigma^{-1} P \1_\i \t_\sigma \1_{\sigma \cdot \i}
\cong (\t_\sigma^{-1} P \1_\i \t_\sigma) \t_\omega^2 \1_{\sigma \cdot \i}.$$
So by Corollary \ref{cor:Punique} we conclude that $\t_\sigma^{-1} P \1_\i \t_\sigma = P \1_{\sigma \cdot \i}$. 

A similar argument show that $\T_\sigma^{-1} \P^- \1_\i \T_\sigma$ is idempotent and 
$$\T_\omega^2 (\T_\sigma^{-1} \P^- \1_\i \T_\sigma) \1_{\sigma \cdot \i}
\cong \T_\sigma^{-1} \P^- \1_\i \T_\sigma \1_{\sigma \cdot \i}
\cong (\T_\sigma^{-1} \P^- \1_\i \T_\sigma) \T_\omega^2 \1_{\sigma \cdot \i}.$$
This means that $p(\T_\sigma^{-1} \P^- \1_\i \T_\sigma) = P \1_{\sigma \cdot \i}$ which, since $\t_\sigma^{-1} P \1_\i \t_\sigma = P \1_{\sigma \cdot \i}$, implies $p(\P^- \1_\i) = P \1_{\i}$. This concludes the proof of Theorem \ref{thm:main2}.

\section{The representation $\Lambda_q^{m \infty}(\k^m \otimes \k^{2\infty})$ and tangle invariants}\label{sec:replambda}

We will now prove Theorem \ref{thm:main3}. Most of the work is setting everything up correctly at the decategorified level (sections \ref{sec:tanglesfund} and \ref{sec:tanglesgen}). It is then clear how to pass to categories (section \ref{sec:catknots}). In section \ref{sec:linkinv} we explain how to obtain $\Z^2$-graded homological link invariants from our setup. 

\subsection{The weight spaces}\label{sec:weightspaces}

If we fix $m,N$ then 
\begin{eqnarray*}
\Lambda_q^{mN}(\k^m \otimes \k^{2N})
&\cong& \Lambda_q^{mN}(\k^m \oplus \dots \oplus \k^m) \\
&\cong& \bigoplus_{i_1 + \dots i_{2N} = mN} \Lambda_q^{i_1}(\k^m) \otimes \dots \otimes \Lambda_q^{i_{2N}}(\k^m)
\end{eqnarray*}
where there are $2N$ summands $\k^m$ in the first line. As noted in section \ref{sec:Howe}, each $\Lambda_q^{i_1}(\k^m) \otimes \dots \otimes \Lambda_q^{i_{2N}}(\k^m)$ is a weight space for the action of $U_q(\sl_{2N})$ where 
$$\left[ \dots \otimes \Lambda_q^{i_k}(\k^m) \otimes \Lambda_q^{i_{k+1}}(\k^m) \otimes \dots \right]
\overset{E_k}{\underset{F_k}{\rightleftarrows}} 
\left[ \dots \otimes \Lambda_q^{i_k-1}(\k^m) \otimes \Lambda_q^{i_{k+1}+1}(\k^m) \otimes \dots \right]$$

Thus, the nonzero weight spaces of $\Lambda_q^{mN}(\k^m \otimes \k^{2N})$ as a $U_q(\sl_{2N})$-module are in natural bijection with $2N$-tuples $(i_1, \dots, i_{2N})$ where $0 \le i_1, \dots, i_{2N} \le m$ and $\sum_\ell i_\ell = mN$. The weight space labeled by $\i = (i_1, \dots, i_{2N})$ will be denoted $V(\i)$. 

In this notation, $\alpha_k = (0, \dots, -1,1, \dots, 0)$ so that $V(\i) \overset{E_k}{\underset{F_k}{\rightleftarrows}} V(\i+\alpha_k)$. The bilinear form $\la \cdot, \cdot \ra$ on weights is given by the usual dot product $\la \i, \i' \ra = \i \cdot \i'$ of tuples. The highest weight is $(0,\dots,0,m,\dots,m)$ where there are $N$ $0$s and $m$s. The Weyl group of $\sl_{2N}$, with generators $s_1, \dots, s_{2N-1}$, permutes the weights
$$s_k \cdot (i_1, \dots, i_k, i_{k+1}, \dots, i_{2N}) = (i_1, \dots, i_{k+1}, i_k, \dots, i_{2N}).$$

Given a weight ${\i}$ denote by $\rho(\i)$ the sequence obtained by dropping all $\i_\ell \in \{0,m\}$. For example, if $\i = (1,0,3,m,5)$ then $\rho(\i) = (1,3,5)$. Moreover, we denote by $S(\i)$ the set of weights $\i'$ such that $\rho(\i) = \rho(\i')$. 

\begin{Lemma}\label{lem:4} 
Fix $\i$. One can canonically identify all the weight spaces $V(\j)$ with $\j \in S(\i)$.
\end{Lemma}
\begin{proof}
Any two weights in $S(\i)$ can be related by a sequence of moves which exchange $i_\ell, i_{\ell+1} \in \i$ as long as at least one of them belongs to $\{0,m\}$. Now, if $i_k$ or $i_{k+1}$ belongs to $\{0,m\}$ then it is easy to see $\t_k \1_\l$ is either $E_k^{(-\la \i,\alpha_k \ra)} \1_\i$ or $F_k^{(\la \i,\alpha_k \ra)} \1_\i$ depending on whether $\la \i, \alpha_k \ra \le 0 $ or $\la \i, \alpha_k \ra \ge 0$. This means that $\t_k^2 \1_{\i} = \1_{\i}$. 

Since the action is transitive on $S(\i)$ we can assume that $\i = (0,\dots,0,i_1,\dots,i_j,m,\dots,m)$ where there are $n_1$ $0$'s and $n_2$ $m$'s and $i_\ell \not\in \{0,m\}$. Now $\t_k \1_\i$ clearly acts by the identity on the $n_1$ $0$'s or on the $n_2$ $m$'s. Combining this with the fact that $\t_k^2 \1_{\i} = \1_{\i}$ if $i_k$ or $i_{k+1}$ is in $\{0,m\}$ (proven above) gives the result. 

\end{proof}

Letting $N \rightarrow \infty$ we denote the resulting vector space by $\Lambda_q^{m\infty}(\k^m \otimes \k^{2\infty})$. More precisely, 
\begin{align*}
\Lambda_q^{m \infty}(\k^m \otimes \k^{2 \infty})
&\cong \Lambda_q^{m \infty}(\bigoplus_{k \in \Z} \k^m) \\
&\cong \bigoplus_{\i} \left[ \dots \otimes \Lambda_q^{i_k}(\k^m) \otimes \Lambda_q^{i_{k+1}}(\k^m) \otimes \dots \right]
\end{align*}
where the direct sum is over all sequences $\i$ where $i_k = 0$ if $k \ll 0$ and $i_k = m$ if $k \gg 0$ and the sum of all $i_k \not\in \{0,m\}$ is divisible by $m$. Because of the former condition the infinite tensor product in each summand above is actually finite. As before, $E_k$ and $F_k$ correspond to maps 
$$(\dots, i_k, i_{k+1}, \dots) \overset{E_k}{\underset{F_k}{\rightleftarrows}} (\dots, i_k-1, i_{k+1}+1, \dots).$$
except now $k \in \Z$. The weight space labeled by $\i$ is still denoted $V(\i)$. 

Given $\i$ we again have $\rho(\i)$ which forgets all the terms in $\i$ equal to $0$ or $m$. Note that $\rho(\i)$ is still a finite sequence. As before, we denote by $S(\i)$ all sequences $\i'$ such that $\rho(\i) = \rho(\i')$. 

Now consider the embedding $U_q(\sl_{2N}) \rightarrow U_q(\sl_{\infty})$ given by $E_k \mapsto E_{k-N}$ and $F_k \mapsto F_{k-N}$ where $k=1, \dots, 2N-1$. If we restrict $\Lambda_q^{m\infty}(\k^m \otimes \k^{2\infty})$ to $U_q(\sl_{2N})$ we find that it contains the module $\Lambda_q^{mN}(\k^m \otimes \k^{2N})$ as a direct summand. Moreover, given any weight ${\i}$, the weight space $V({\i})$ sits as the weight space of such a direct summand for $N$ sufficiently large. The following is an immediate corollary of Lemma \ref{lem:4}

\begin{Corollary}\label{cor:3} 
One can canonically identify all the weight spaces $V(\j)$ of $\Lambda_q^{m \infty}(\k^m \otimes \k^{2 \infty})$ if $\j \in S(\i)$. 
\end{Corollary}

To define the tangle invariants it will be convenient to define the complexes 
$$\T'_k \1_\i := \begin{cases} \T_k \1_\i [-i_k] \la i_k \ra & \text{ if } \la \i, \alpha_k \ra = -i_k + i_{k+1} \ge 0 \\ \T_k \1_\i [-i_{k+1}]\la i_{k+1} \ra & \text{ if } \la \i, \alpha_k \ra = -i_k + i_{k+1} \le 0 \end{cases}$$ 
It is easy to check that these $\T'_k$ also satisfy the braid relations of Proposition \ref{prop:braidgrpaction}. The advantage is that we now have the following ({\it c.f.} Lemma \ref{lem:2}). 

\begin{Corollary}\label{COR:ET}
If $|i-j|=1$ then $\T'_i \T'_j \E_i \cong \E_j \T'_i \T'_j$ and $\T'_i \T'_j \F_i \cong \F_j \T'_i \T'_j$. 
\end{Corollary}

We will denote by $\t'_k \1_\i$ the decategorification of $\T'_k \1_\i$. 

\subsection{Tangle invariants: fundamental representations}\label{sec:tanglesfund}

Consider an oriented tangle $T$ whose strands are labeled by representations (or equivalently, dominant weights) of $\sl_m$. Let $ \ul = (\l^{(1)}, \dots, \l^{(n)}) $ be the labels on the strands at the top and $ \umu = (\mu^{(1)}, \dots, \mu^{(n')}) $ be the labels on the strands at the bottom. As usual, $V_\l$ denotes the irreducible $\sl_m$ representation with highest weight $\l$. Note that if a strand labeled by $V_\l$ is oriented upward then the boundary point is marked by $\l$ whereas if it is oriented downward then it is marked by $\l'$ where $V_{\l'} \cong V_\l^*$ is the dual. 

A $U_q(\sl_m)$ oriented tangle invariant associates to such a tangle a map of $U_q(\sl_m)$-modules 
\begin{equation*}
\psi(T) :  V_\ul := V_{\lambda^{(1)}} \otimes \dots \otimes V_{\lambda^{(n)}} \rightarrow V_\umu := V_{\mu^{(1)}} \otimes \dots \otimes V_{\mu^{(n')}}.
\end{equation*}
This is done by analyzing a tangle projection from bottom to top and assigning maps to each cap, cup and crossing (the generating tangles are shown in figure (\ref{fig:1})). If the maps $\psi(T)$ do not depend on the planar projection of the tangle then we get a map
\begin{equation*}
\psi : \biggl\{ (\ul, \umu) \text{ tangles } \biggr\} \rightarrow \Hom_{U_q(\sl_m)}(V_\ul, V_\umu)
\end{equation*}
If $T = K$ is an oriented link then $ \psi(K) $ is a map $ \k[q,q^{-1}] \rightarrow \k[q,q^{-1}]$ and $\psi(K)(1)$ becomes a polynomial invariant of $K$. 

Reshetikhin and Turaev defined such an oriented tangle invariant in \cite{RT}. We now explain how to recover their maps from the $U_q(\sl_\infty)$-module $\Lambda_q^{m \infty}(\k^m \otimes \k^{2 \infty})$ using skew Howe duality. For the moment we consider the case when all strands are labeled by fundamental weights. So $\ul = (\Lambda_{i_1}, \dots, \Lambda_{i_n})$ and $V_\ul$ is identified with $V(\i)$ where $\i = (\dots, 0, i_1, \dots, i_n, m, \dots)$. 

\begin{itemize}
\item To the four over crossings from figure (\ref{fig:1}) one associates maps $V(\i) \rightarrow V(s_k \cdot \i)$ given by $T'_k$, $(-q)^{-i_k} T'_k$, $(-q)^{-m+i_{k+1}} T'_k$ and $(-q)^{i_{k+1}-i_k} T'_k$ respectively. The corresponding four under crossings are associated the inverse maps. 
\item To the cap and cup from figure (\ref{fig:1}) are associated the maps 
\begin{align*}
& E_k^{(i_k)}: V(\dots,i_k,m-i_k,\dots) \rightarrow V(\dots,0,m,\dots) \ \ \text{ and } \\
& F_k^{(i_k)}: V(\dots,0,m,\dots) \rightarrow V(\dots,i_k,m-i_k,\dots)
\end{align*}
regardless of the orientation of the strand.
\end{itemize}

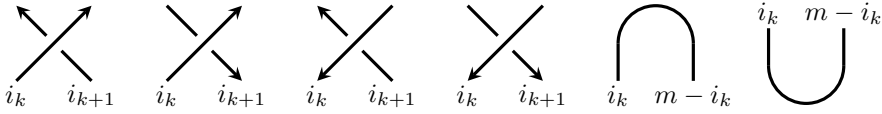
\begin{figure}[!h]
\centering
\begin{tikzpicture}[>=stealth]
\draw [->](0,0) -- (1,1) [very thick];
\draw [->](0.4,0.6) -- (0,1) [very thick];
\draw [-](1,0) -- (0.6,0.4) [very thick];
\draw [shift={+(0,-0.2)}](0,0) node {$i_k$};
\draw [shift={+(0,-0.2)}](1,0) node {$i_{k+1}$};

\draw [shift={+(2,0)}][->](0,0) -- (1,1) [very thick];
\draw [shift={+(2,0)}][-](0.4,0.6) -- (0,1) [very thick];
\draw [shift={+(2,0)}][<-](1,0) -- (0.6,0.4) [very thick];
\draw [shift={+(2,0)}][shift={+(0,-0.2)}](0,0) node {$i_k$};
\draw [shift={+(2,0)}][shift={+(0,-0.2)}](1,0) node {$i_{k+1}$};

\draw [shift={+(4,0)}][<-](0,0) -- (1,1) [very thick];
\draw [shift={+(4,0)}][->](0.4,0.6) -- (0,1) [very thick];
\draw [shift={+(4,0)}][-](1,0) -- (0.6,0.4) [very thick];
\draw [shift={+(4,0)}][shift={+(0,-0.2)}](0,0) node {$i_k$};
\draw [shift={+(4,0)}][shift={+(0,-0.2)}](1,0) node {$i_{k+1}$};

\draw [shift={+(6,0)}][<-](0,0) -- (1,1) [very thick];
\draw [shift={+(6,0)}][-](0.4,0.6) -- (0,1) [very thick];
\draw [shift={+(6,0)}][<-](1,0) -- (0.6,0.4) [very thick];
\draw [shift={+(6,0)}][shift={+(0,-0.2)}](0,0) node {$i_k$};
\draw [shift={+(6,0)}][shift={+(0,-0.2)}](1,0) node {$i_{k+1}$};

\draw [shift={+(8,.5)}](0,0) -- (0,-.5) [very thick];
\draw [shift={+(8,.5)}](1,0) -- (1,-.5)[-] [very thick];
\draw [shift={+(8,.5)}](0,0) arc (180:0:.5) [very thick];
\draw [shift={+(8,0)}][shift={+(0,-0.2)}](0,0) node {$i_k$};
\draw [shift={+(8,0)}][shift={+(0,-0.2)}](1,0) node {$m-i_k$};

\draw [shift={+(10,.2)}](0,0) -- (0,.5) [-][very thick];
\draw [shift={+(10,.2)}](1,0) -- (1,.5) [very thick];
\draw [shift={+(10,.2)}](0,0) arc (180:360:.5) [very thick];
\draw [shift={+(10,0)}][shift={+(0,.9)}](0,0) node {$i_k$};
\draw [shift={+(10,0)}][shift={+(0,.9)}](1,0) node {$m-i_k$};
\end{tikzpicture}
\caption{The cap and cup can have either orientation.}\label{fig:1}
\end{figure}

\begin{Proposition}\label{prop:4}
These maps define a $U_q(\sl_m)$ oriented tangle invariant. 
\end{Proposition}
\begin{proof}
In Lemma \ref{lem:commute} we saw that the actions of $U_q(\sl_m)$ and $U_q(\sl_n)$ on $\Lambda_q^N(\k^m \otimes \k^n)$ commute. This implies that the actions of $U_q(\sl_m)$ and $U_q(\sl_\infty)$ on $\Lambda_q^{m\infty}(\k^m \otimes \k^{2\infty})$ also commute. Since $V(\i)$ is a weight space of $U_q(\sl_\infty)$ the action of $U_q(\sl_m)$ preserves it. Morever, all the maps defined above (crossings, caps and cups) are written in terms of elements in $U_q(\sl_\infty)$ which means that they commute with the $U_q(\sl_m)$ action and hence induce maps of $U_q(\sl_m)$-modules. 

\begin{figure}[ht]
\begin{center}
\puteps[0.4]{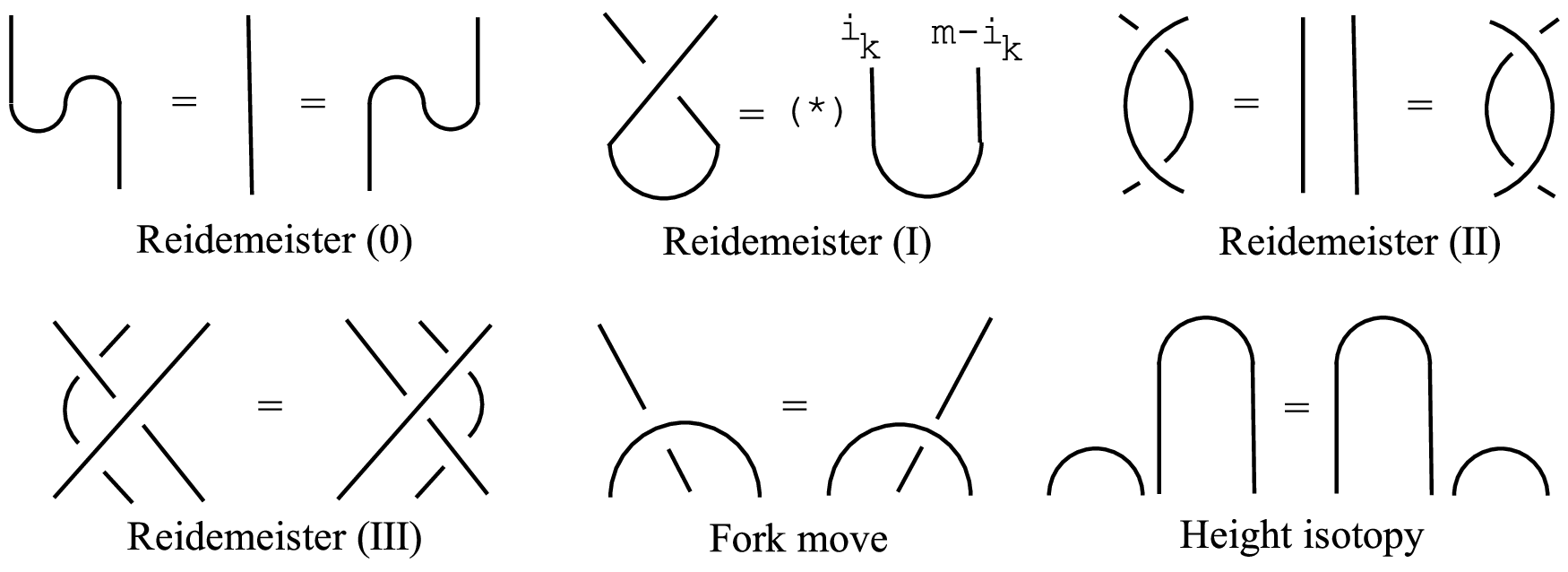}
\end{center}
\caption{Invariance relations where $(\star) = q^{i_k(m-i_k)}$.}\label{fig:2}
\end{figure}

To show that we have an invariant it suffices to check the relations in figure (\ref{fig:2}) where the strands are labeled by arbitrary fundamental weights and orientations. One also needs to check isotopy relations involving changing the height of crossings, caps and cups which are far apart (for example, the right most relation in the second line of figure (\ref{fig:2})). However, these isotopy relations are clear because the functors $E_i, F_i$ commute with $E_j, F_j$ if $|i-j|>1$. We now prove the remaining relations. 

{\bf (R0).} This relation amounts to showing that the following two compositions 
\begin{align*}
& V(\dots, 0,m,i_k,\dots) \xrightarrow{F_{k-2}^{(i_k)}} V(\dots, i_k, m-i_k, i_k, \dots) \xrightarrow{E_{k-1}^{(m-i_k)}} V(\dots, i_k, 0, m, \dots) \\
& V(\dots,i_k,0,m,\dots) \xrightarrow{F_{k+1}^{(m-i_k)}} V(\dots,i_k,m-i_k,i_k,\dots) \xrightarrow{E_{k}^{(i_k)}} V(\dots,0,m,i_k,\dots)
\end{align*}
are both equal to the identity map once you identify $V(\dots, 0,m,i_k,\dots)$ and $V(\dots, i_k, 0,m,\dots)$. We prove the first composition (the second follows in the same way). 

Using Lemma \ref{lem:4}, the identification $V(\dots,i_k,0,m,\dots) \xrightarrow{\sim} V(\dots,0,m,i_k,\dots)$ is given by $F_{k+1}^{(m-i_k)} E_k^{(i_k)}$. So we need to show that 
$$F_{k-1}^{(m-i_k)} E_{k-2}^{(i_k)} E_{k-1}^{(m-i_k)} F_{k-2}^{(i_k)} = id \in \End_{U_q(\sl_m)}(V(\dots,0,m,i_k,\dots)).$$ 
The left side equals $F_{k-1}^{(m-i_k)} E_{k-2}^{(i_k)}  F_{k-2}^{(i_k)} E_{k-1}^{(m-i_k)} = F_{k-1}^{(m-i_k)} E_{k-1}^{(m-i_k)} = id$ which completes the proof. 

{\bf (RI).} The (RI) relation consists of showing that 
$$V(\dots,0,m,\dots) \xrightarrow{F_k^{(i_k)}} V(\dots,i_k,m-i_k,\dots) \xrightarrow{\t'_k} V(\dots,m-i_k,i_k,\dots)$$
is equal to $(-q)^{i_k} q^{i_k(m-i_k)} F_k^{(m-i_k)}$. There are two cases. If $2i_k \le m$ then 
$$\t'_k F_k^{(i_k)} = q^{i_k(m-i_k)} E_k^{(i_k)} \t'_k = q^{i_k(m-i_k)} E_k^{(i_k)} F_k^{(m)} = q^{i_k(m-i_k)} F_k^{(m-i_k)}$$
where we used Corollary \ref{cor:0} to get the second equality (recall that in K-theory the shift $\la 1 \ra$ corresponds to $q^{-1}$). The case $2i_k \ge m$ is similar. 

{\bf (RII) and (RIII).} The (RII) relation follows from \cite{CKL3} where we show that $\t'_k$ is invertible. The (RIII) relation follows from \cite{CK3} (see the discussion in section \ref{sec:braidgrpaction} and Proposition \ref{prop:braidgrpaction}).  

{\bf Fork relations.} The fork relations involving cups follows formally from those involving caps together wit (R0). There are eight fork relations involving caps depending on the two possible ways to orient the strands and whether the crossings are over or under. We prove one of these cases, shown in figure (\ref{fig:A}), as the others are exactly the same. 

\begin{figure}[!h]
\centering
\begin{tikzpicture}[>=stealth]
\draw [-](0,0) -- (.7,.7) [very thick];
\draw [->](0.4,0.6) -- (0,1) [very thick];
\draw [-](1,0) -- (0.6,0.4) [very thick];
\draw (0.7,0.7) arc (125:65:.5) [very thick];
\draw [->](1.2,0.75) -- (2,0) [very thick];
\draw [shift={+(0,-0.2)}](0,0) node {$i_k$};
\draw [shift={+(0,-0.2)}](1,0) node {$i_{k+1}$};
\draw [shift={+(0,-0.2)}](2,0) node {$m-i_k$};

\draw [shift={+(0,0)}](3,.5) node {$=$};

\draw [shift={+(4,0)}][-](0,0) -- (.7,.7) [very thick];
\draw [shift={+(4,0)}][-](1,0) -- (1.4,0.4) [very thick];
\draw [shift={+(4,0)}][->](1.6,0.6) -- (2,1) [very thick];
\draw [shift={+(4,0)}](0.7,0.7) arc (125:65:.5) [very thick];
\draw [shift={+(4,0)}][->](1.2,0.75) -- (2,0) [very thick];
\draw [shift={+(4,0)}][shift={+(0,-0.2)}](0,0) node {$i_k$};
\draw [shift={+(4,0)}][shift={+(0,-0.2)}](1,0) node {$i_{k+1}$};
\draw [shift={+(4,0)}][shift={+(0,-0.2)}](2,0) node {$m-i_k$};
\end{tikzpicture}
\caption{One of the fork relations.}\label{fig:A}
\end{figure}
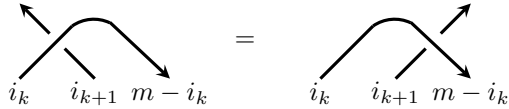

The left side is equal to $E_{k+1}^{(i_k)} \t'_k \1_{(\dots,i_k,i_{k+1},m-i_k,\dots)}$. The right hand side is equal to 
$$(-q)^{m-i_{k+1}} F_k^{(m-i_{k+1})} E_{k+1}^{(m-i_{k+1})} \1_{(\dots, 0,m,i_{k+1}, \dots)} E_k^{(i_k)} (\t'_{k+1})^{-1}.$$
The left most two terms are equal to $(-q)^{-m+i_{k+1}} \t'_k \t'_{k+1}$. The result follows since $\t'_k \t'_{k+1} E_k^{(i_k)} = E_{k+1}^{(i_k)} \t'_k \t'_{k+1}$ using (the decategorification of) Corollary \ref{COR:ET}.
\end{proof}

Above we used the $U_q(\sl_\infty)$ structure to construct isomorphisms $V_\ul = V(\i) \xrightarrow{\sim} V(s_k \cdot \i) = V_\umu$ and to define caps and cups. Alternatively, one can define this isomorphism as $Flip \circ R$ where $R$ is the $R$-matrix defined in \cite{RT} and $Flip$ exchanges factors $k$ and $k+1$ (one can also define caps and cups in this way). However, these two approaches are equivalent.

\begin{Proposition} The maps $T'_k \1_\i$, $E_k^{(i_k)} \1_\i$ and $\1_\i F_k^{(i_k)}$ used above to define crossings, cups and caps are equal to the Reshetikhin-Turaev maps \cite{RT} up to sign and multiplication by some power of $q$. 
\end{Proposition}
\begin{proof}
This follows from \cite{CKM} where we construct a functor from $U_q(\sl_n)$ (thought of as a category where the objects are weights) to the category of representations of $U_q(\sl_m)$. In particular, under this functor, the maps $T'_k$ are taken to the isomorphisms $Flip \circ R$. In \cite[Sect. 6]{CKM} we determine the precise factor needed to rescale $T'_k \1_\i$ in order to get the Reshetikhin-Turaev crossing on the nose. 
\end{proof}

\subsection{Tangle invariants: arbitrary representations}\label{sec:tanglesgen}

To obtain the Reshetikhin-Turaev invariants for tangles labeled by arbitrary weights we use the clasps $P \1_{\i}$ defined in section \ref{sec:clasps}. Recall that $P \1_\i$ is the composition 
\begin{equation}\label{eq:11}
\Lambda_q^{\i}(\k^m) \xrightarrow{\pi} V_{\i} \xrightarrow{\iota} \Lambda_q^{\i}(\k^m)
\end{equation}
where $V_\i = V_{\sum_k \Lambda_{i_k}}$, the first map is projection and the second map is inclusion. Such a map is unique up to scalar but if we insist that $P^2=P$ then this scalar is also uniquely determined.  

\begin{figure}[ht]
\begin{center}
\puteps[0.4]{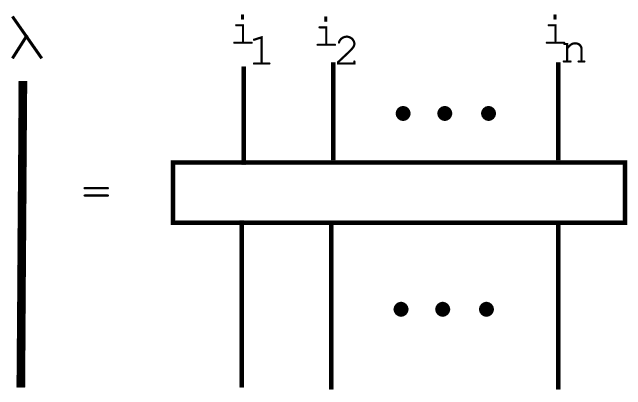}
\end{center}
\caption{A strand labeled $\l = \sum_{k=1}^n \Lambda_{i_k}$ defined using a clasp.}\label{fig:3}
\end{figure}

Instead of dealing with a strand labeled by $\l = \sum_{k=1}^n \Lambda_{i_k}$ we replace it with $n$ strands labeled by $i_1, \dots, i_n$ together with a clasp. Diagrammatically, the clasp is illustrated by a box as in figure (\ref{fig:3}). In this setup, we work with clasps and strands labeled only by fundamental weights. One can then associate a map to crossings, cups and caps involving strands labeled by arbitrary representations as follows. 

Since $P \1_\i$ is idempotent we recover $V_\i$ as the image of $P \1_\i$ on $\Lambda_q^{\i}(\k^m)$. Now, if $\sum_{k=1}^n \Lambda_{i_k}$ and $\mu = \sum_{k=1}^{n'} \Lambda_{j_k}$ one recovers the crossing map $V_\i \otimes V_\j \rightarrow V_\j \otimes V_\i$ from the crossing map $V(\i) \otimes V(\j) \rightarrow V(\j) \otimes V(\i)$. Note that this latter map involves multiple crossings of strands labeled by fundamental representations and hence is somewhat complicated. 

However, this construction only makes sense if this crossing map maps $V_\i \otimes V_\i \subset V(\i) \otimes V(\j)$ to $V_\j \otimes V_\i \subset V(\j) \otimes V(\i)$. This is indeed the case because of the relation in the first row of figure (\ref{fig:4}). The relations in figure (\ref{fig:4}) also ensure that the resulting maps satisfy the relations in figure (\ref{fig:2}) where $(\star) = q^{\sum_k i_k(m-i_k)}$ if the top left strand is labeled by $\sum_k \Lambda_{i_k}$. 

\begin{figure}[ht]
\begin{center}
\puteps[0.4]{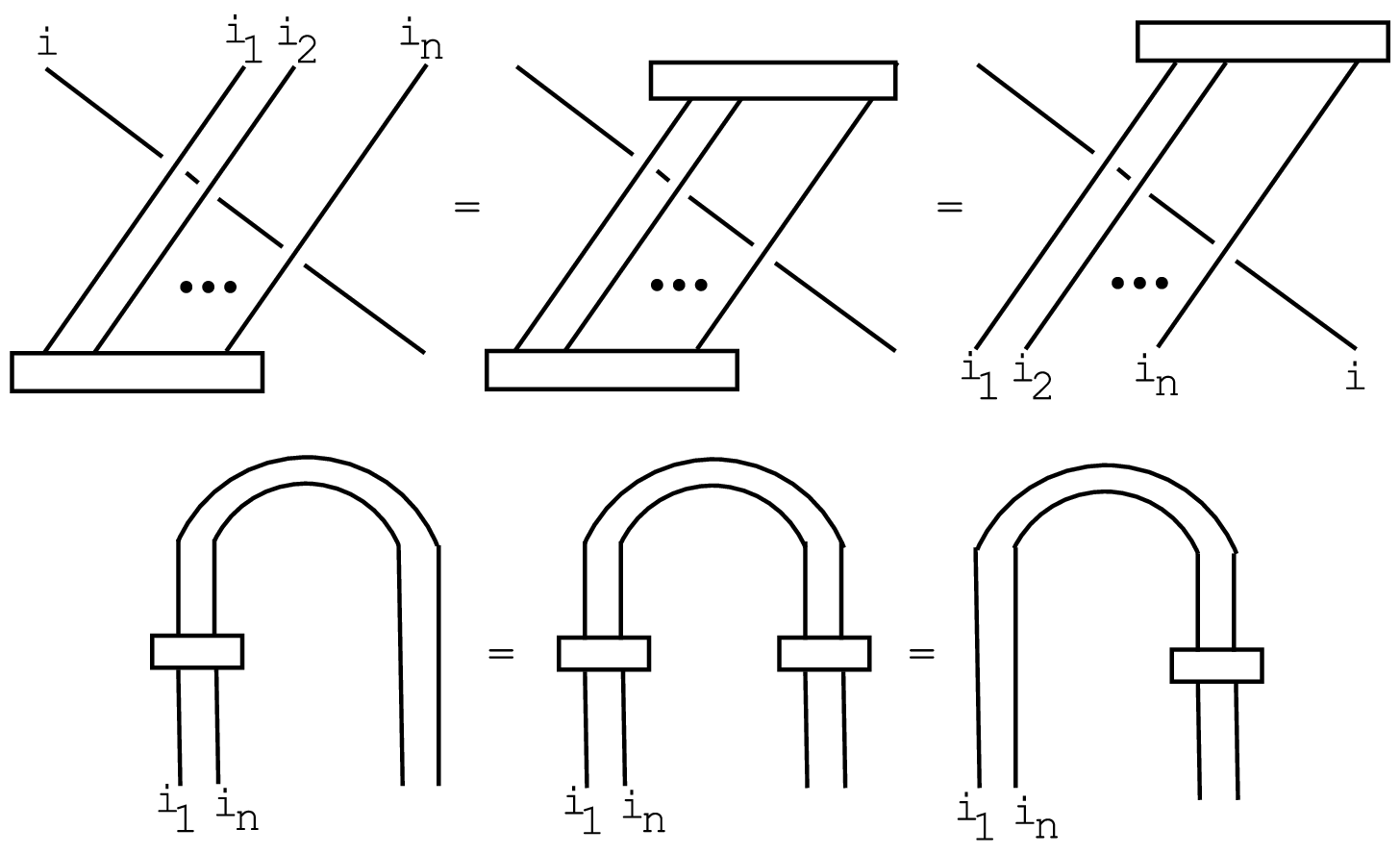}
\end{center}
\caption{Relations involving clasps.}\label{fig:4}
\end{figure}

\begin{Proposition}\label{prop:5}
The relations in figure (\ref{fig:4}) are valid. 
\end{Proposition}
\begin{proof}
This is a consequence of Proposition \ref{prop:5'} where we check these relations at the categorical level. This means that we regard $P \1_\i$ as the class $[\P^- \1_\i]$ where $\P^- = \lim_{\ell \rightarrow \infty} \T_\omega^{2\ell}$ rather than as the composition $\iota \circ \pi$ from (\ref{eq:11}). From this point of view the relations in figure (\ref{fig:4}) follow very easily. 
\end{proof}

\begin{Remark}\label{rem:justweights}
From the construction above it becomes clear that we only used two things about the representation $\Lambda_q^{m \infty}(\k^m \otimes \k^{2\infty})$ in order to obtain our knot invariant. Namely, we used information about certain weight spaces being zero and we used that the highest weight space is one-dimensional. 
\end{Remark}

\subsection{Categorical knot invariants}\label{sec:catknots}

We now categorify the picture from the previous section. In place of $\Lambda_q^{m \infty}(\k^m \otimes \k^{2\infty})$ we take a 2-category $\K$ equipped with an $(\sl_\infty,\th)$ action. Moreover, we require that the nonzero weight spaces of this 2-category match precisely with the nonzero weight spaces of $\Lambda_q^{m \infty}(\k^m \otimes \k^{2\infty})$. In other words, the nonzero weight spaces of $\K$ are indexed by sequences $\i = (\dots, 0,0, i_1, \dots, i_n, m,m, \dots)$. Notice that we do not require that the objects $\l$ of $\K$ be categories whose K-theory recovers $\Lambda_q^{m \infty}(\k^m \otimes \k^{2\infty})$ ({\it cf.} Remark \ref{rem:justweights}). 

Instead of $E_k^{(r)},F_k^{(r)}, \t_k, \t'_k$ we now have $\E_k^{(r)},\F_k^{(r)}, \T_k, \T'_k$.  

\begin{Corollary}\label{cor:3'} 
Fix $\i$. One can canonically identify all the weights $\j \in S(\i)$ in $\K$. 
\end{Corollary}
\begin{proof}
This is the categorical analogue of Corollary \ref{cor:3}. The proof is the same.
\end{proof}

We now immitate the definition from section \ref{sec:tanglesfund} to define categorical tangle invariants. 
\begin{itemize}
\item To the four over crossings from figure (\ref{fig:1}) one associates maps $\i \rightarrow s_k \cdot \i$ by $\T'_k$, $\T'_k [-i_k] \la i_k \ra$, $\T'_k [-m+i_{k+1}]\la m-i_{k+1} \ra$ and $\T'_k [i_{k+1}-i_k] \la -i_{k+1}+i_k \ra$ respectively. The corresponding four under crossings are associated the inverse maps.
\item To the cap and cup from figure (\ref{fig:1}) are associated the maps
\begin{align*}
& \E_k^{(i_k)}: (\dots,i_k,m-i_k,\dots) \rightarrow (\dots,0,m,\dots) \ \ \text{ and } \\
& \F_k^{(i_k)}: (\dots,0,m,\dots) \rightarrow (\dots,i_k,m-i_k,\dots)
\end{align*}
regardless of the orientation of the strand.
\end{itemize}

\begin{Proposition}\label{prop:4'}
These maps define a homological oriented tangle invariant. In other words, the relations in figure (\ref{fig:2}) hold as isomorphisms of functors, where $(\star) = \la -i_k(m-i_k) \ra$. 
\end{Proposition}
\begin{proof}
This is the categorical analogue of Proposition \ref{prop:4}. The proof is the same. Recall that $\la 1 \ra$ denotes a grading shift by one while $[1]$ is a cohomological shift by one. The (RI) relation states that the curl can be undone if we include a shift by $(\star)$, which is the categorical analogue of multiplying by $q^{i_k(m-i_k)}$. 
\end{proof}

Next we define the clasp as the complex of 1-morphisms $\P^- \1_\i \in \Kom_*^-(\K)$. 

\begin{Proposition}\label{prop:5'}
The relations in figure (\ref{fig:4}) are valid as isomorphisms of functors.
\end{Proposition}
\begin{proof}
We check the relation between the first two diagrams in the first row of figure (\ref{fig:4}) (the other cases are proved in the same way). The left and right hand sides are equal to 
$$(\T'_1 \dots \T'_n)(\P^- \I) \1_{(\i,i)} \text{ and } (\I \P^-)(\T'_1 \dots \T'_n)(\P^- \I)\1_{(\i,i)}$$
Now consider the map 
\begin{equation}\label{eq:12}
(\T'_1 \dots \T'_n)(\P^- \I) \1_{(\i,i)} \rightarrow (\I \P^-)(\T'_1 \dots \T'_n)(\P^- \I) \1_{(\i,i)}
\end{equation} 
induced by the map $\1_{(i,\i)} \rightarrow (\I \P^-) \1_{(i,\i)}$. The cone of this is (by definition) 
\begin{equation}\label{eq:13}
(\R)(\T'_1 \dots \T'_n)(\P^- \I) \1_{(\i,i)}.
\end{equation} 
Now, for any $n \ge 0$, $\P^- \1_{(\i,i)} \cong (\T_\omega^{2n}) (\P^-) \1_{(\i,i)}$ and, based on the RIII relation for the $\T_i$'s, we have $(\T'_1 \dots \T'_n)(\T_\omega^{2n} ) \cong ( \T_\omega^{2n}) (\T'_1 \dots \T'_n)$. Thus we get
$$(\R)(\T'_1 \dots \T'_n)(\P^- \I) \1_{(\i,i)} \cong (\R)(\T_\omega^{2n})(\T'_1 \dots \T'_n)(\P^- \I) \1_{(\i,i)}.$$
Now $(\R)(\T_\omega^{2n})$ is supported in homological degrees $\le -2n$ by Proposition \ref{prop:2}. Since $(\T'_1 \dots \T'_n)(\P^- \I)$ is supported in homological degrees bounded above by some $B$ this means that (\ref{eq:13}) is supported in degrees $\le -2n+B$. Since we can choose $n$ arbitrarily this means that (\ref{eq:13}) is contractible and hence the map in (\ref{eq:12}) is an isomorphism. 
\end{proof}

This concludes the proof of Theorem \ref{thm:main3}. 

\subsection{A homological link invariant}\label{sec:linkinv}
Using the process above, the invariant associated to an oriented link $K$ is a complex of 1-morphisms $\Psi_-(K) \in \End_{\Kom^-(\K)}((\0, \um))$. All the terms in this complex are a direct sum of $\1_{(\0, \um)}$ since any composition of 1-morphisms which starts and ends at $(\0, \um)$ breaks up as a direct sum of identity morphisms (with various degree shifts). 

To obtain a (doubly graded) homology from $\Psi_-(K)$ consider the $\Z$-graded algebra 
$$ \rA := \End_{\K}^*(\1_{({\0}, {\um})}) = \bigoplus_{k \in \Z} \End_\K(\1_{({\0}, {\um})}, \1_{({\0}, {\um})} \la k \ra).$$
and the associated functor 
$$\Hom^*_{\Kom^-(\K)}(\1_{({\0}, {\um})}, \bullet): \End_{\Kom^-(\K)}((\0,\um)) \rightarrow \Kom^-(\rA \dmod).$$
Since under this map $\1_{({\0}, {\um})} \mapsto \rA$, the image $\Psi_-'(K)$ of $\Psi_-(K)$ is a complex of projective $\rA$-modules. Tensoring with $\k$ gives us $\overline{\Psi}_-(K) := \k \otimes_{\rA} \Psi_-'(K)$ which is a complex of $\Z$-graded vector spaces. We then define $\sH_-^{i,j}(K)$ to be $\mbox{gr}^j H^i(\overline{\Psi}_-(K))$. 

\begin{Remark}
In this paper, the 2-categories $\K$ have the property that $\End^k(\1_{(\0,\um)}) = 0$ if $k \ne 0$. Hence $\rA \cong \k$ and the procedure above is redundant ({\it i.e.} in this case $\Psi_-(K)$ contains the same information as $\sH_-^{i,j}(K)$).  
\end{Remark}

\section{Affine Grassmannians and the 2-category $\K_{\Gr,m}$}\label{sec:affGrass}

We now define a 2-category $\K_{\Gr,m}$ and an $(\sl_\infty,\th)$ action on it. The objects in $\K_{\Gr,m}$ are derived categories of coherent sheaves on certain iterated Grassmannian bundles associated to the affine Grassmannian for ${\rm PGL}_m$. Taking K-theory we recover the $U_q(\sl_\infty)$-module $\Lambda_q^{m \infty}(\k^m \otimes \k^{2\infty})$. 

\subsection{Derived categories of coherent sheaves}\label{sec:FMkernels}
If $X$ is a variety then we write $D(X)$ for the bounded derived category of coherent sheaves on $X$. An object (kernel) $\sP \in D(X \times Y)$ whose support is proper over $Y$ induces a functor $\Phi_{\sP}: D(X) \rightarrow D(Y)$ via $(\cdot) \mapsto \pi_{2*}(\pi_1^* (\cdot) \otimes \sP)$ (where every operation is derived). If $\sQ \in D(Y \times Z)$ then $\Phi_{\sQ} \circ \Phi_{\sP} \cong \Phi_{\sQ * \sP}: D(X) \rightarrow D(Z)$ where $\sQ * \sP = \pi_{13*}(\pi_{12}^* \sP \otimes \pi_{23}^* \sQ)$ is the composition (a.k.a. convolution product) of kernels. 

If $X$ carries a $\k^\times$ then, abusing notation, we denote by $D(X)$ the derived category of $\k^\times$-equivariant coherent sheaves on $X$ which. We write $\{\cdot\}$ for a shift in the $\k^\times$ grading. We can similarly define everything above for the bounded above (resp. below) derived categories $D^-(X)$ (resp. $D^+(X)$) of coherent sheaves on $X$.

\subsection{Varieties}\label{sec:varieties}
We define
\begin{equation*}
Y(\i) := \{ \k[z]^m = L_0 \subset L_1 \subset \dots \subset L_{2N} \subset \k(z)^m : z L_j \subset L_{j-1}, \dim(L_j/L_{j-1}) = i_j \}
\end{equation*}
where the $L_i$ are complex vector subspaces. For $j=1, \dots, 2N$ there is a natural vector bundle on $Y(\i)$ whose fibre over a point $\{L_0 \subset \dots \subset L_{2N}\}$ is the vector space $L_j/L_0$. We denote this bundle by $\sL_j$.

There is an action of $\k^\times$ on $\k(z)$ given by $t \cdot z^k = t^{2k} z^k$. This induces an action of $\k^\times$ on $\k(z)^m$. Since for any $v \in \k(z)^m$ we have $t \cdot (zv) = t^2 z(t \cdot v)$ this induces a $\k^\times$ action on $Y(\i)$ via $t \cdot (L_0, \dots ,L_{2N}) = (t \cdot L_0, \dots , t \cdot L_{2N})$. Everything we define or claim will be $\k^\times$-equivariant with respect to this action. 

\begin{Remark}
If we forget $L_{2N}$ then we get a Grassmannian bundle 
$$Y(i_1, \dots, i_{2N-1}, i_{2N}) \rightarrow Y(i_1, \dots, i_{2N-1})$$
with fibres $\G(i_{2N},m)$ because $L_{2N}$ can be any subspace in $z^{-1}(L_{2N-1})/L_{2N-1} \cong \k^m$ of dimension $i_{2N}$. Thus $Y(\i)$ is just an iterated Grassmannian bundle. The relation of $Y(\i)$ to the affine Grassmannian is via the twisted (or convolution) product \cite{CK2}. More precisely, $Y(\i) = Y(i_1) \tilde{\times} Y(i_2) \tilde{\times} \dots \tilde{\times} Y(i_{2N})$.
\end{Remark}

\subsection{Kernels}\label{sec:kernels} 
For $r \ge 0$ we define correspondences $W_k^r(\i) \subset Y(\i) \times Y(\i + r \alpha_k)$ as the subvariety $\{(L_\bullet, L'_\bullet): L_\bullet \subset L_\bullet'\}$. Here, as in section \ref{sec:weightspaces}, $\alpha_k$ denotes $(0,\dots, -1,1,\dots, 0)$ where the $-1$ is the $k$th term. More explicitly, 
\begin{align*}
W_k^r(\i) =& \{\k[z]^m = L_0 \xrightarrow{i_1} \dots \xrightarrow{i_{k-1}} L_{k-1} \xrightarrow{i_k-r} L_k' \xrightarrow{r} L_k \xrightarrow{i_{k+1}} L_{k+1} \xrightarrow{i_{k+2}} \dots \xrightarrow{i_{2N}} L_{2N} \subset \k(z)^m: \\
& zL_j \subset L_{j-1} \text{ for all } j, \text{ and } zL'_k \subset L_{k-1} \}
\end{align*}
where the arrows are inclusions and the superscripts indicate the codimension of the inclusion. We define the kernels
\begin{align}
\label{eq:kernel1}
\sE_k^{(r)} \1_\i &:= \O_{W_k^r(\i)} \otimes \det(\sL_k'/\sL_{k-1})^r \otimes \det(\sL_k/\sL_k')^{-i_k+r} \{r(i_k-r)\} \in D(Y(\i) \times Y(\i + r\alpha_k)) \\
\label{eq:kernel2}
\1_\i \sF_k^{(r)} &:= \O_{W_k^r(\i)} \otimes \det(\sL_{k+1}/\sL_k')^{-r} \otimes \det(\sL_k'/\sL_k)^{i_{k+1}} \{r i_{k+1}\} \in D(Y(\i + r\alpha_k) \times Y(\i))
\end{align}
where the prime denotes pullback from the second factor.

\subsection{Deformations}\label{sec:deforms}
The variety $Y(\i)$ has a natural deformation over $\bA^{2N}$ given by
$$\{\k[z]^m = L_0 \subset L_1 \subset \dots \subset L_{2N} \subset \k(z)^m; {\ux} \in \bA^{2N}: (z-x_j) L_j \subset L_{j-1}, \dim(L_j/L_{j-1}) = i_j\}.$$
Notice that if $x_j=0$ for all $j$ we recover $Y(\i)$. This deformation is trivial over the main diagonal in $\bA^{2N}$. For this reason we restrict it to the locus where $x_{2N}=0$ to obtain a deformation $\tY(\i) \rightarrow \bA^{2N-1}$. We identify $\bA^{2N-1}$ with the complexified root lattice $Y_\k$. The $\k^\times$ action on $Y(\i)$ extends to a $\k^\times$ action on all of $\tY(\i)$ if we act on the base $\bA^{2N-1}$ via $\ux \mapsto t^2 \ux$.

As explained in \cite[Sect. 14.2]{C2} such a deformation gives us a linear map 
\begin{equation}
\label{eq:new4}
\theta: Y_\k \rightarrow \Hom(\Delta_* \O_{Y(\i)}, \Delta_* \O_{Y(\i)} [2]\{-2\})
\end{equation}
where $\Delta: Y(\i) \rightarrow Y(\i) \times Y(\i)$ is the diagonal map. 

\subsection{The 2-category $\K_{\Gr,m}$}\label{sec:catsl8}
The 2-category $\K_{\Gr,m}$ consists of:
\begin{itemize}
\item objects: $D(Y(\0, \i, \um))$ where $\i = (i_1, \dots, i_{2N})$ for some $N$ and $\sum_k i_k$ is divisible by $m$,
\item 1-morphisms: all kernels, 
\item 2-morphisms: all morphisms of kernels. 
\end{itemize}
This 2-category is endowed with the $\Z$-grading $\la 1 \ra = [1]\{-1\}$. Since the variety $Y(\0,\i,\um)$ can be naturally identified with $Y(\i)$ we find that (\ref{eq:new4}) gives us a map $\theta: Y_\k \rightarrow \End^2(\1_{\0,\i,\um})$. This is because $\End^k(\1_\i)$ is identified with $\Hom(\Delta_* \O_{Y(\i)}, \Delta_* \O_{Y(\i)} [k]\{-k\})$. 

\begin{Theorem}\label{thm:geomcat}
The data above gives a $(\sl_\infty,\th)$ action on $\K_{\Gr,m}$. This action categorifies the $U_q(\sl_\infty)$-module $\Lambda_q^{m \infty}(\k^m \otimes \k^{2\infty})$.
\end{Theorem}
\begin{proof}
Conditions (i),(ii),(iii) were proven in \cite{CKL1} as part of Theorem 3.3. Note that in that paper we used different line bundles to define the $\sE$'s and $\sF$'s. However, those line bundles differ from the ones here by conjugating with $\prod_k \det(\sL_k/\sL_0)^{-i_k}$ on $Y(\i)$ so all the relations still hold. 

Condition (iv) was checked in the case of cotangent bundles to partial flag varieties in \cite[Sect. 3]{CK3} but the exact same proof applies here. Conditions (vi) and (vii) are obvious. Finally, condition (v) is a consequence of the fact that the sheaves $\sE_i$ and $\sF_i$ deform along $\alpha_i^\perp \subset Y_\k$. This was again shown in the case of cotangent bundles in \cite{CK3} but the same proof applies. 

What remains is to identify the $U_q(\sl_\infty)$-module it categorifies. Let us restrict ourselves to varieties $Y(\i)$ where $\i = (i_1, \dots, i_{2N})$. The variety $Y(\i)$ is an iterated Grassmannian bundle. Since the Grothendieck group of the Grassmannian $\bG(i,m)$ has rank $\binom{m}{i}$ it follows that 
$$\dim_\k K(Y(\i)) = \prod_{k=1}^{2N} \binom{m}{i_k}.$$
Since a finite dimensional $U_q(\sl_{2N})$-module is uniquely determined by the dimensions of its weight spaces it follows that the restriction of $(\sl_\infty,\th)$ to this subcategory categorifies the $U_q(\sl_{2N})$-module $\Lambda_q^{mN}(\k^m \otimes \k^{2N})$. Letting $N \rightarrow \infty$ we obtain the $U_q(\sl_\infty)$-module $\Lambda_q^{m \infty}(\k^m \otimes \k^{2\infty})$. 
\end{proof}

\section{Nakajima quiver varieties and the 2-category $\K_{{\rm Q},m}$}\label{sec:quivers}

Since $\Lambda_q^{mN}(\k^m \otimes \k^{2N})$ has commuting actions of $U_q(\sl_m)$ and $U_q(\sl_{2N})$ each $\sl_m$ weight space is preserved by the action of $U_q(\sl_{2N})$. In particular, we can restrict the $U_q(\sl_{2N})$ action to the zero weight space of $\sl_m$, which is isomorphic to $(\Lambda_q^{N}(\k^{2N}))^{\otimes m}$. 

Letting $N \rightarrow \infty$ we can likewise restrict the action of $U_q(\sl_\infty)$ from $\Lambda_q^{m \infty}(\k^m \otimes \k^{2 \infty})$ to $\Lambda_q^{\infty}(\k^{2\infty})^{\otimes m}$. In this section we explain how this smaller $U_q(\sl_\infty)$-module can be categorified using Nakajima quiver varieties. This leads to the same link invariants as the ones constructed above. 

\subsection{The 2-category $\K_{{\rm Q},m}$} 

Let $\Gamma$ be the Dynkin diagram of $\sl_{2N}$ and denote by $I$ the sets of vertices of $\Gamma$. Given $\v, \w \in \N^I \cong \N^{2N-1}$ one can define the Nakajima quiver variety $\M(\v,\w)$ as a symplectic quotient (see \cite[Sect. 3]{CKL4}). Now, if we fix $\w$ and allow $\v$ to vary then $\M(\v,\w)$ corresponds to the weight space $\l := \sum_{i \in I} w_i \Lambda_i - \sum_{i \in I} v_i \alpha_i$. 

We now define the 2-category $\K_{{\rm Q},\w}$ to consist of:
\begin{itemize}
\item objects: weight $\l$ indexes $D(\M(\v,\w))$ where $\l,\v,\w$ are related as above,
\item 1-morphisms: all kernels,
\item 2-morphisms: all morphisms of kernels.
\end{itemize}
Certain correspondences between these varieties define kernels
$$\sE_i \in D(\M(\v,\w)) \times \M(\v-{\e}_i,\w) \text{ and }
\sF_i \in D(\M(\v-{\e}_i,\w) \times \M(\v,\w))$$
where $\e_i = (0,\dots,0,1,0,\dots,0)$ has a $1$ in position $i$. Moreover, varying the moment map condition also gives deformations $\widetilde{\M}(\v,\w) \rightarrow Y_\k$ which give us  maps $\th: Y_\k \rightarrow \End^2(\1_\l)$ as in section \ref{sec:deforms}. 

This data gives us an $(\sl_{2N},\th)$ action on $\K_{{\rm Q},\w}$. This follows from \cite{CKL4} where we showed that it gives us a geometric categorical $\sl_{2N}$ action together with \cite[Prop. 14.4]{C2} which explains how such a geometric action induces an $(\sl_{2N},\th)$ action. Moreover, this action categorifies the $U_q(\sl_{2N})$-module $\bigotimes_i V_{\Lambda_i}^{\otimes w_i} \cong \bigotimes_i \Lambda_q^{i}(\k^{2N})^{\otimes w_i}$. 

Now take $\w = (0, \dots, 0,m,0, \dots, 0) \in \N^{2N-1}$. Then $\oplus_{\v} K(\M(\v,\w)) \cong \Lambda_q^N(\k^{2N})^{\otimes m}$ as $U_q(\sl_{2N})$-modules. In this case the weight spaces $K(\M(\v,\w))$ are in bijection with $2N$-tuples $(i_1, \dots, i_{2N})$ with $0 \le i_1, \dots, i_{2N} \le m$ and $\sum_\ell i_\ell = mN$. This bijection is given by 
\begin{equation}\label{eq:phi}
\phi: \v \mapsto (0^N,m^N) + (v_1,-v_1+v_2,-v_2+v_3,\dots,-v_{2N-2}+v_{2N-1},-v_{2N-1}).
\end{equation}
Letting $N \rightarrow \infty$ we obtain a 2-category $\K_{{\rm Q},m}$. The discussion above implies the following. 

\begin{Proposition}[\cite{CKL4}]\label{prop:6} 
The data above gives a $(\sl_\infty,\th)$ action on $\K_{{\rm Q},m}$. This action categorifies the $U_q(\sl_\infty)$-module $\Lambda_q^{\infty}(\k^{2\infty})^{\otimes m}$. 
\end{Proposition}

Given this $(\sl_\infty,\th)$ action on $\K_{{\rm Q},m}$ one can define functors for cups, caps and crossings just as with $\K_{\Gr,m}$ which, when applied to an oriented link, gives us a homological link invariant. 

\subsection{Geometrization}

In the discussion above we explained that there exists a projection map
\begin{equation}\label{eq:projection}
\Lambda_q^{mN}(\k^m \otimes \k^{2N}) \longrightarrow \Lambda_q^{\infty}(\k^{2\infty})^{\otimes m}
\end{equation}
of $U_q(\sl_\infty)$-modules. Moreover, the left side is categorified using the varieties $Y(\i)$ while the right using quiver varieties. One can realize this projection geometrically as follows. 

Recall that in section \ref{sec:affGrass} we defined the twisted product varieties 
$$Y(\i) = \{ \k[z]^m = L_0 \subset L_1 \subset \dots \subset L_{2N} \subset z^{-2mN} \k[z]^m : z L_j \subset L_{j-1}, \dim(L_j/L_{j-1}) = i_j \}$$
where $\sum_k i_k = Nm$ (for notational convenience we have replaced $\k(z)^m$ with $z^{-2mN} \k[z]^m$ in the definition of $Y(\i)$ above). This allows us to define the projection 
$$Pr: z^{-2mN} \k[z]^m \cong z^{-2mN} \k[z] \otimes \k^m \longrightarrow z^{-N} \k[z] \otimes \k^m \cong z^{-N} \k[z]^m$$ 
as follows 
$$Pr(z^k \otimes v) = \begin{cases} z^k \otimes v &\text{ if } k \ge -N \\ 0 &\text{ if } k < -N. \end{cases}$$
where $v \in \k^m$. Then we can define $U(\i) \subset Y(\i)$ as the locus of points in $L_\bullet \in Y(\i)$ such that $Pr(L_{2N}) = z^{-N} \k[z]^m$ or, equivalently, the locus of points where $\dim(Pr(L_{2N})/\k[z]^m) = Nm$. Since $\dim(L_{2N}/\k[z]^m) = \sum_k i_k = Nm$ it is not hard to see that $U(\i) \subset Y(\i)$ is an open subscheme.

Using \cite{MV} one can show that $U(\i)$ is actually isomorphic to $\M(\v,\w)$ where $\w = (0^{N-1},m,0^{N-1})$ and $\phi(\v) = \i$ via the bijection $\phi$ from (\ref{eq:phi}). Moreover, if we denote by $u$ the natural inclusion of $\M(\v,\w)$ into $Y(\i)$ then the restriction map $u^*: D(Y(\i)) \rightarrow D(\M(\v,\w))$ gives a functor 
$$u^*: \K_{\Gr,m} \rightarrow \K_{{\rm Q},m}$$
which intertwines the $(\sl_\infty,\th)$ actions either side and categorifies the projection map (\ref{eq:projection}). 

\section{Some example computations: $\Sym^2 V$ and the adjoint representation}\label{sec:example}

To illustrate our constructions of clasps we will compute the cohomology of the unknot labeled by $V_{2 \Lambda_1} = \Sym^2 V$ (where $V$ is the standard $m$-dimensional representation of $\sl_m$) and by the adjoint representation. In the case of $\Sym^2 V$ the unknot is illustrated in figure (\ref{fig:5}) where the P denotes the clasp $\P^-$ categorifying the composition $V \otimes V \xrightarrow{\pi} V_{2 \Lambda_1} \xrightarrow{\iota} V \otimes V$.

\begin{figure}[ht]
\begin{center}
\puteps[0.4]{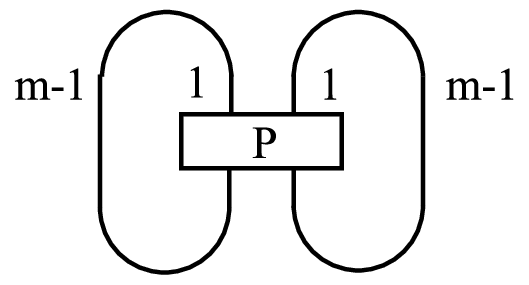}
\end{center}
\caption{}\label{fig:5}
\end{figure}

In what follows we will use all the properties of a 2-representation in the sense of \cite{KL3}. One motivation for performing these calculations is to illustrate how the computation of our link homologies can be performed entirely within the realm of higher representation theory. 

\subsection{Some notation}

Very briefly, recall the diagrammatic notation from \cite{KL3}. We will only deal with the $\sl_2$ case where $\E=\E_1$ and $\F = \F_1$. An upward (resp. downward) pointing strand $\sUup$ (resp. $\sUdown$) denotes $\E$ (resp. $\F$). Caps and cups denote adjunctions. A crossing denotes the affine nilHecke 2-morphism $T: \E\E \rightarrow \E\E \la -2 \ra$ while a solid dot $\sUupdot$ indicates the 2-morphism $X: \E \rightarrow \E \la 2 \ra$. By convention, 1-morphisms are composed horizontally going to the left while 2-morphisms are composed vertically going upwards. 

We refer the reader to \cite{KL3} for complete list of relations satisfied by such diagrams. One should not confuse such diagrams with link diagrams like the one in figure (\ref{fig:5}). 

\subsection{The clasp} 

First we need an explicit description of the clasp $\P^-$. In this case it involves only two strands, so $\g = \sl_2$ and we have
$$(-2) \overset{\E}{\underset{\F}{\rightleftarrows}} (0) \overset{\E}{\underset{\F}{\rightleftarrows}} (2)$$
where weight space $(0)$ corresponds to $V \otimes V$ while $(-2)$ and $(2)$ to $\Lambda^2_q(V) \otimes \Lambda^0_q(V)$ and $\Lambda^0_q(V) \otimes \Lambda^2_q(V)$ respectively. The projector $\P^- \1_0$ is given by $\lim_{\ell \rightarrow \infty} \T^{2\ell} \1_0$ which we now describe explicitly. 

First, we have 
\begin{align}\label{eg1}
\T \1_0 \cong [ \E \F \1_0 \la -1 \ra \xrightarrow{\Ucapr} \1_0 ]
\end{align}
and squaring gives 
\begin{align*}
\T^2 \1_0 \cong \left[ \E \F \E \F \1_0 \la -2 \ra \xrightarrow{\left(\sUup \sUdown \Ucapr, \Ucapr \sUup \sUdown \right)} \E \F \1_0 \la -1 \ra \oplus \E \F \1_0 \la -1 \ra \xrightarrow{\left(\Ucapr, - \Ucapr \right)} \1_0 \right]
\end{align*}
Now $\F \E \1_{-2} \cong \1_{-2} \la -1 \ra \oplus \1_{-2} \la 1 \ra$ where the isomorphisms are given by 
\begin{equation}\label{eg0}
\Ucupr \ \ \text{ and } \ \  \Udotcupr \ \ \ \ \text{ with inverses } \ \ \ \  \Udotcapl - \Ucapl \hspace{-.2cm} \smccbub{2}{} \ \ \text{ and } \ \ \Ucapl.
\end{equation}
Using these we find that $\T^2 \1_0$ is isomorphic to the complex 
\begin{align*}
\xymatrix{
\E \F \1_0 \la -3 \ra \ar[rr]^{\sUupdot \sUdown} \ar@/_1pc/[drr]|-{\sUup \sUdowndot} & & \E \F \1_0 \la -1 \ra \ar[r]^{\Ucapr} & \1_0 \\
\E \F \1_0 \la -1 \ra \ar@/^1pc/[rru]|-{\id} \ar[rr]_{\id}  & & \E \F \1_0 \la -1 \ra \ar[ur]_{-\Ucapr} & 
}
\end{align*}
Using the cancellation on the two terms in the bottom row we end up with 
\begin{align}\label{eg2}
\T^2 \1_0 \cong [ \E \F \1_0 \la -3 \ra \xrightarrow{\sUupdot \sUdown - \sUup \sUdowndot} \E \F \1_0 \la -1 \ra \xrightarrow{\Ucapr} \1_0 ].
\end{align}

Now composing (\ref{eg1}) with (\ref{eg2}) gives us that $\T^3 \1_0$ is isomorphic to 
\begin{align*}
\xymatrix{
& & & \E \F \1_0 \la -3 \ra \ar[rr]^{\sUupdot \sUdown - \sUup \sUdowndot} & & \E \F \1_0 \la -1 \ra \ar[r]^{\Ucapr} & \1_0 \\
\E \F \E \F \1_0 \la -4 \ra \ar@/^/[urrr]^{-\Ucapr \sUup \sUdown} \ar[rrr]_{\sUup \sUdown \sUupdot \sUdown - \sUup \sUdown \sUup \sUdowndot} & & &  \E \F \E \F \1_0 \la -2 \ra \ar[rr]_{\sUup \sUdown \Ucapr} \ar@/_/[urr]^{\Ucapr \sUup \sUdown} & & \E \F \1_0 \la -1 \ra \ar[ur]_{-\Ucapr} & }
\end{align*}
Using (\ref{eg0}) again we can expand to obtain
\begin{align*}
\xymatrix{
& & \E \F \1_0 \la -3 \ra \ar[rr]^{-\sUupdot \sUdown + \sUup \sUdowndot} & & \E \F \1_0 \la -1 \ra \ar[r]^{\Ucapr} & \1_0 \\
\E \F \1_0 \la -5 \ra \ar@/^1pc/[urr]^{\sUupdot \sUdown} \ar[rr]|-<<<<{\alpha} \ar@/_/[drr]|->>>{\beta} & & \E \F \1_0 \la -1 \ra \ar[urr]|-{\id} \ar[rr]|-<<<<{\id} & & \E \F \1_0 \la -1 \ra \ar[ur]_{-\Ucapr} & \\
\E \F \1_0 \la -3 \ra \ar[uurr]|->>>>{\id} \ar@/_/[urr]|->>>>{- \sUup \sUdowndot} \ar[rr]_{\id} & & \E \F \1_0 \la -3 \ra \ar@/_/[uurr]|-<<<{\sUupdot \sUdown} \ar@/_1pc/[urr]_{\sUup \sUdowndot} & & &}
\end{align*}
where $\alpha = - \sUup \smcbub{-1} \sUdown$ and $\beta = - \sUup \sUdowndot + \sUup \smccbub{2} \sUdown$. Note that to end up with this complex we used several relations such as the one relating clockwise and counterclockwise bubbles \cite[Eq. 3.7]{KL3} and the bubble slide relations \cite[Prop. 3.3]{KL3}. In the end, using the cancellation lemma we end up with 
\begin{align}\label{eg3}
\T^3 \1_0 \cong [\E \F \1_0 \la -5 \ra \xrightarrow{\sUupdot \sUdown + \sUup \sUdowndot - \sUup \smccbub{2} \sUdown} \E \F \1_0 \la -3 \ra \xrightarrow{\sUupdot \sUdown - \sUup \sUdowndot} \E \F \1_0 \la -1 \ra \xrightarrow{\Ucapr} \1_0 ].
\end{align}
The fact that the composition of the first two maps above is zero is a consquence of 
$$\sUupdot {\mbox{\tiny 2}} \sUdown + \sUupdot \smcbub{+1} \sUdown + \sUup \smcbub{+2} \sUdown = 0 = \sUup \sUdowndot {\mbox{\tiny 2}} + \sUup \smcbub{+1} \sUdowndot + \sUup \smcbub{+2} \sUdown$$
which follows from relation \cite[Eq. 3.5]{KL3} together with $\smcbub{+1} = - \smccbub{2}$ from \cite[Eq. 3.7]{KL3}.

Iterating the argument above we find that
\begin{equation}\label{eq:Tinfty}
\T^{2n} \1_0 \cong \left[ \E \F \1_0 \la -2n-1 \ra \rightarrow \E \F \1_0 \la -2n+1 \ra \rightarrow \dots \rightarrow \E \F \1_0 \la -1 \ra \rightarrow \1_0 \right]
\end{equation}
where the maps alternate between $\sUupdot \sUdown - \sUup \sUdowndot$ and $\sUupdot \sUdown + \sUup \sUdowndot - \sUup \smccbub{2} \sUdown$. Thus $\P^- \1_0 = \lim_{\ell \rightarrow \infty} \T^{2 \ell} \1_0$ is the obvious limit of this complex. 

\subsection{The unknot labeled by $\Sym^2 V$}\label{sec:computation1}

In terms of our categorical 2-representation, the homology of the knot in figure (\ref{fig:5}) is given by the composition
\begin{equation}\label{eq:calculate}
\E_3 \F_1 \P^-_2 \E_1 \F_3 \in \Hom_{\Kom^-_*(\K)}((\0,m,0,0,m,\um), (\0,m,0,0,m,\um))
\end{equation}
where $\P^-_2 = \lim_{\ell \rightarrow \infty} \T_2^{2 \ell}$ is the infinite complex 
\begin{equation}\label{eq:tempcpx}
\left[ \dots \rightarrow \E_2 \F_2 \1_\i \la -2n-1 \ra \rightarrow \E_2 \F_2 \1_\i \la -2n+1 \ra \rightarrow \dots \rightarrow \E_2 \F_2 \1_\i \la -3 \ra \rightarrow \E_2 \F_2 \1_\i \la -1 \ra \rightarrow \1_\i \right]
\end{equation}
where $\i = ({\0}, m-1,1,1,m-1, {\um})$. Now, 
\begin{align*}
\E_3 \F_1 (\E_2 \F_2) \E_1 \F_3 \1_{\bk} 
&\cong \E_3 \E_2 \F_1 \E_1 \1_{({\0}, m,1,0,m-1, {\um})} \F_2 \F_3 \1_{\bk} \\
&\cong \oplus_{[m-1]} \E_3 \E_2 \F_2 \1_{({\0},m,0,1,m-1, {\um})} \F_3 \1_{\bk} \\
&\cong \oplus_{[m-1]} \E_3 \F_3 \1_{\bk} \\
&\cong \oplus_{[m]} \oplus_{[m-1]} \1_{\bk}
\end{align*}
where $\bk = ({\0},m,0,0,m, {\um})$. Subsequently, the composition in (\ref{eq:calculate}) simplifies to give a complex 
\begin{equation}\label{eq:cpxhomology}
\left[ \dots \rightarrow \bigoplus_{[m][m-1]} \1_\bk \la -2n-1 \ra \xrightarrow{d_n} \bigoplus_{[m][m-1]} \1_\bk \la -2n+1 \ra \xrightarrow{d_2} \dots \xrightarrow{d_1} \bigoplus_{[m][m-1]} \1_\bk \la -1 \ra \xrightarrow{d_0} \bigoplus_{[m][m]} \1_\bk \right]
\end{equation}
where the value of the right most term is a consequence of the fact that 
$$\E_3 \F_1 \E_1 \F_3 \1_\bk \cong \E_3 \F_3 \E_1 \F_1 \1_\bk \cong \bigoplus_{[m][m]} \1_\bk.$$
It remains to identify the differentials in (\ref{eq:cpxhomology}). 

\begin{Lemma}\label{lem:diffs} 
The differential $d_0$ in (\ref{eq:cpxhomology}) is injective. After that, $d_n=0$ if $n$ is odd while if $n>0$ is even then $d_n$ has the highest possible rank, namely $(m-1)^2$. 
\end{Lemma}
\begin{proof}
First we consider the differential $d_0$. As we vary $0 \le a_1 \le m-1, 0 \le b_1 \le m-2$ we pick out every summand $\1_\bk$ inside $\E_3 \F_1 \E_2 \F_2 \E_1 \F_3 \1_\bk$ as the composition depicted in the lower half (the part below the lower dashed line) of the left hand diagram in (\ref{fig:6}). 

\begin{figure}[ht]
\begin{center}
\puteps[0.4]{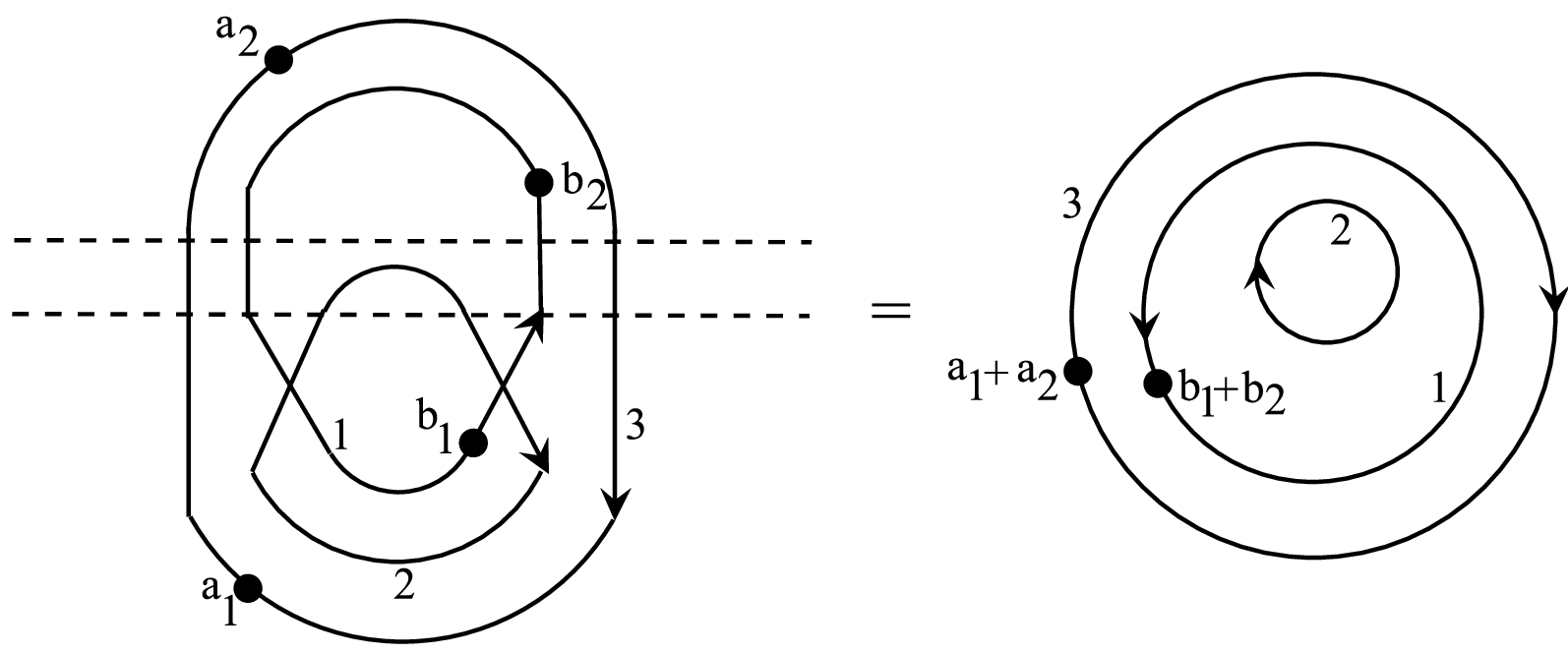}
\end{center}
\caption{}\label{fig:6}
\end{figure}

Next, the differential $d_0$ in (\ref{eq:cpxhomology}) is induced by the adjunction map. Finally, we pick out every summand $\1_\bk$ in $\E_3 \F_1 \E_1 \F_3 \1_\bk$ as the composition depicted above the upper dashed line in (\ref{fig:6}). Subsequently, the differential $d_0$ is given by an $(m-1)(m-2) \times (m-1)^2$ matrix $M_0$, with rows labeled by $(a_1,b_1)$ and columns by $(a_2,b_2)$ and where the corresponding matrix entry is the composition on the left of figure (\ref{fig:6}). It remains to simplify this composition. 

First, we move the circle labeled 2 inside the circle labeled 1 and, using the fact that the dots move through up-down crossings, obtain the composition encoded by the right hand diagram of figure (\ref{fig:6}). Next, we use the bubble slide relations \cite[Prop. 3.3,3.4]{KL3} to slide the innermost bubble all the way to the outside (this takes two steps). Keeping in mind that $\End(\1_\bk)$ is supported in degree zero (so any positive degree maps are zero) we end up with
\begin{equation}\label{eq:c1}
\smccbub{b_1+b_2} \smcbub{a_1+a_2+2} - \smccbub{b_1+b_2+1} \smcbub{a_1+a_2+1}.
\end{equation}
Now, $\smccbub{b} \smcbub{a} = \delta_{b,m-1} \delta_{a,m-1}$ so (\ref{eq:c1}) is equal to $\delta_{b_1+b_2,m-1} \delta_{a_1+a_2,m-3} - \delta_{b_1+b_2,m-2} \delta_{a_1+a_2,m-2}$. This means that $M_0$ is a block matrix where the blocks are indexed by $\ell = a_1+b_1$. For example, if $\ell=m-2$ then we obtain the block
\begin{equation}\label{eq:b1}
\begin{matrix}
(a_1,b_1) \setminus (a_2,b_2) & (m-2,0) & (m-3,1) & \dots & (1,m-3) & (0,m-2) \\
(0,m-2) & -1 & 1 & \dots & 0 & 0\\
(1,m-3) & 0 & -1 & \dots & 0 & 0 \\
\dots & \dots & \dots & \dots & \dots & \dots \\
(m-3,1) & 0 & \dots & \dots & -1 & 1 \\
(m-2,0) & 0 & \dots & \dots & 0 & -1 
\end{matrix}
\end{equation}
Subsequently, we conclude that $d_0$ is injective, with cokernel $\bigoplus_{[m]} \1_{\bk} \la m-1 \ra$. 

Next, we need to understand the remaining maps in the complex (\ref{eq:cpxhomology}). Recall that the differentials in (\ref{eq:tempcpx}) alternate between $\sUupdot \sUdown - \sUup \sUdowndot$ and $\sUupdot \sUdown + \sUup \sUdowndot - \sUup \smccbub{2} \sUdown$. Arguing as above, $\sUupdot \sUdown$ and $\sUup \sUdowndot$ induce the two compositions on the left and middle diagrams in figure (\ref{fig:7}) (where this time $0 \le a_1,a_2 \le m-1$ and $0 \le b_1,b_2 \le m-2$). Both of these are equal to the composition on the right in figure (\ref{fig:7}) (which is the same as the composition in (\ref{fig:6})). This explains why $\sUupdot \sUdown - \sUup \sUdowndot$ induces the zero map ({\it i.e.} $d_n=0$ if $n$ is odd). 

\begin{figure}[ht]
\begin{center}
\puteps[0.4]{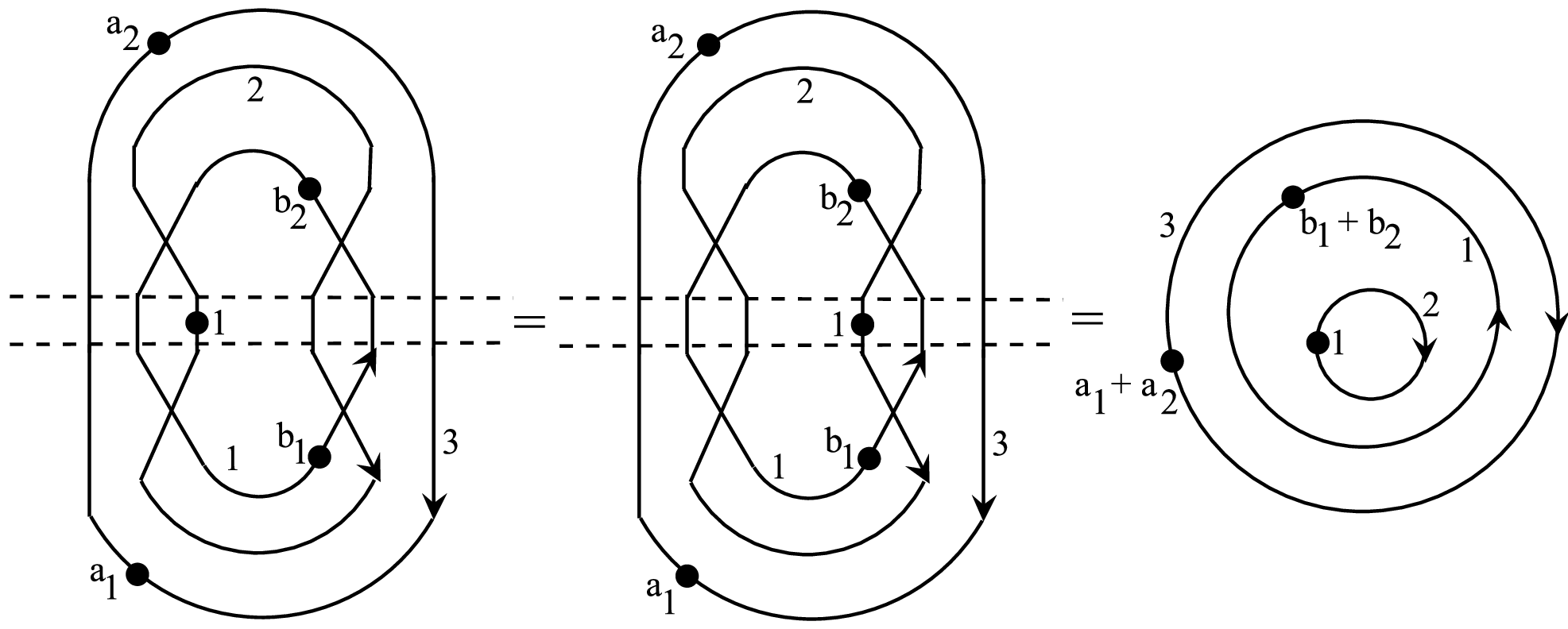}
\end{center}
\caption{}\label{fig:7}
\end{figure}

Finally, we need to compute the contribution of $\sUup \smccbub{2} \sUdown$. Proceeding as above we get the composition from figure (\ref{fig:8}). 

\begin{figure}[ht]
\begin{center}
\puteps[0.4]{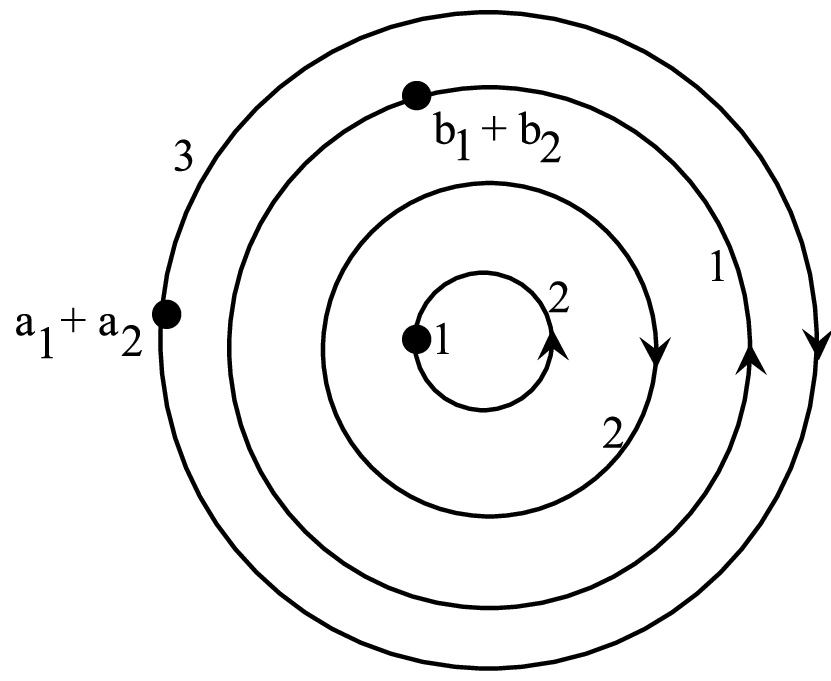}
\end{center}
\caption{}\label{fig:8}
\end{figure}

Using the bubble slide relations, the two inner bubbles (both labeled by $2$) in (\ref{fig:8}) are equal to 
$$\ccbub{} \cbub{} + 2 \smcbub{1} = \smcbub{1} - \smccbub{1}$$
where to obtain the equality above we used that $\smccbub{1} + \ccbub{} \cbub{} + \smcbub{1} = 0$ (here all strands are labeled by $2$). The calculation involving $\smcbub{1}$ was done above and gave (\ref{eq:c1}). A similar computation involving $\smccbub{1}$ gives
\begin{equation}\label{eq:c2}
\smccbub{b_1+b_2+2} \smcbub{a_1+a_2} - \smccbub{b_1+b_2+1} \smcbub{a_1+a_2+1}.
\end{equation}
Combining (\ref{eq:c1}) and (\ref{eq:c2}) we find that $d_n$, where $n > 0$ is even, is given by 
$$2 \times (\ref{eq:c1}) - (\ref{eq:c2}) = \smccbub{b_1+b_2} \smcbub{a_1+a_2+2} - 2 \smccbub{b_1+b_2+1} \smcbub{a_1+a_2+1} + \smccbub{b_1+b_2+2} \smcbub{a_1+a_2}.$$
Using that $\smccbub{b} \smcbub{a} = \delta_{b,m-1} \delta_{a,m-1}$, this again gives a matrix $M_n$ whose blocks are indexed by $\ell=a_1+b_1$. For example, if $\ell=m-2$ then we obtain the block
\begin{equation}\label{eq:b2}
\begin{matrix}
(a_1,b_1) \setminus (a_2,b_2) & (m-2,0) & (m-3,1) & \dots & (1,m-3) & (0,m-2) \\
(0,m-2) & -2 & 1 & \dots & 0 & 0\\
(1,m-3) & 1 & -2 & \dots & 0 & 0 \\
\dots & \dots & \dots & \dots & \dots & \dots \\
(m-3,1) & 0 & \dots & \dots & -2 & 1 \\
(m-2,0) & 0 & \dots & \dots & 1 & -2
\end{matrix}
\end{equation}
It is now straightforward to check that, when $n>0$ is even, $d_n$ induces a map of highest possible rank.
\end{proof}

\begin{Corollary}\label{cor:homology}
Denote by $\chi_{q,t}(V_{2\Lambda_1})$ the graded Poincar\'e polynomial of the $\sl_m$ homology of the unknot labeled by $V_{2\Lambda_1}$. Then 
\begin{equation}\label{eq:Poincare}
\chi_{q,t}(V_{2\Lambda_1}) = 1 + [m-1] \frac{q^{-m}+ t^3 q^{m+4}}{1-t^2q^4}.
\end{equation}
\end{Corollary}
\begin{Remark} 
Recall that, by convention, $q$ corresponds to the grading shift $\la -1 \ra$ while $t$ corresponds to $[1]$ (a downward shift by one in cohomology). 
\end{Remark}
\begin{proof}
We replace $\1_\bk$ by $\k$ in (\ref{eq:cpxhomology}) to obtain a complex of graded vector spaces. To simplify notation denote $A := \bigoplus_{[m-1]} \k$ so that $\bigoplus_{[m][m-1]} \k = \bigoplus_{[m]} A$. Since $d_n$ has maximal rank when $n>0$ is even we find that its kernel and cokernel are $A \la -m+1 \ra$ and $A \la m+1 \ra$ respectively. So, after canceling out terms we are left with the following complex 
\begin{equation}\label{eq:invt}
\dots \xrightarrow{} A \la -m-8 \ra \rightarrow A \la m-8 \ra \xrightarrow{} A \la -m-4 \ra \rightarrow A \la m-4 \ra \xrightarrow{} 0 \rightarrow \bigoplus_{[m]} \k \la m-1 \ra.
\end{equation}
where all the maps are zero. Thus the graded Poincar\'e polynomial is 
$$q^{-m+1}[m] + [m-1]\left( q^{-m+4}t^2+q^{m+4}t^3+q^{-m+8}t^4+q^{m+8}t^5 + \dots \right)$$
which simplifies to give (\ref{eq:Poincare}). 
\end{proof}

It is easy to check that $\chi_{q,-1}(V_{2\Lambda_1}) = \frac{[m][m+1]}{[2]}$ which, as expected, is the quantum dimension of $V_{2\Lambda_1} = \Sym^2 V$. 

\subsubsection{The cases $m=2,3$} When $m=2$, Corollary \ref{cor:homology} gives $\chi_{q,t}(V_{2\Lambda_1}) = 1 + \frac{q^{-2} + t^3 q^6}{1-t^2q^4}$. This invariant was also computed in \cite{CoK} (see section 4.3.1). Their homology is supported in positive degree so their Poincar\'e polynomial is actually $1 + \frac{q^{-2} + t^{-3} q^6}{1-t^{-2}q^4}$, obtained by replacing $t$ with $t^{-1}$ (one should also take $q \mapsto q^{-1}$ but this is not necessary since they use the convention that $\la 1 \ra \leftrightarrow q$ which is opposite from ours).  

When $m=3$, Corollary \ref{cor:homology} gives $\chi_{q,t}(V_{2\Lambda_1}) = 1 + [2] \frac{q^{-3} + t^3 q^7}{1-t^2q^4}$. This was also computed in \cite[Example 5.4]{Ros} (the calculation in the first arXiv version of that paper contained a small typo).

\subsubsection{Working over $\Z$}
The calculation above was performed carefully enough that it actually works over $\Z$ rather than $\k$. In this case we still have $d_n=0$ if $n$ is odd, while $d_n$ is of largest possible rank if $n$ is even, but now we get some torsion. If $d_0$ then the blocks all look like the map in (\ref{eq:b1}) so there is no torsion. On the other hand, if $n>0$ is even then the matrix from (\ref{eq:b2}), which is just minus the Cartan matrix of type $A_{m-1}$, is equivalent to 
$$
\left( \begin{matrix} 
-m & 0 & 0 & \dots & 0  \\
0 & 1 & 0 & \dots & 0 \\
\dots & \dots & \dots & \dots & \dots \\
0 & 0 & \dots & 1 & 0 \\
0 & 0 & \dots & 0 & 1
\end{matrix} \right).$$
This has cokernel $\Z/m\Z$. It is easy to check that the cokernels of the other blocks do {\em not} have any torsion. Thus, we get $\Z/m\Z$ torsion which is isomorphic to 
$$\bigoplus_{d \ge 1} \Z/m\Z \la -4d \ra [2d]$$
({\it i.e.} it occurs in cohomological degrees $-2,-4,-6,\dots$). When $m=2$ this torsion was also computed in \cite{CoK} at the end of section 4.3.1 (in their notation, it corresponds to using the parameter $\alpha=0$). Note that the $\Z$-rank of the homology is still given by (\ref{eq:Poincare}). 

\subsubsection{Torus links} 
Cutting off the calculations above also computes the invariants of torus links. For example, consider the link in figure (\ref{fig:5}) where P is replaced by $n$ crossings ({\it i.e.} $\T^{n}$). This gives the $(2,n)$ torus link. Then the associated invariant is calculated by the complex in (\ref{eq:invt}) where we chop off everything after the $(n-1)$th term $A \la \cdot \ra$. So, for example, if $n$ is odd then the Poincar\'e polynomial of the $(2,n)$ torus link becomes 
$$1 + [m-1] \frac{q^{-m}+ t^3 q^{m+4}}{1-t^2q^4} - [m-1] \frac{(tq^2)^{n+1}(q^mt+q^{-m})}{1-t^2q^4}.$$
When $m=2$ this homology was originally computed by Khovanov in \cite[Prop. 26]{K1}. 

\subsection{The unknot labeled by the adjoint representation}\label{sec:computation2}
We now compute the homology of the unknot labeled by $V_{\Lambda_1 + \Lambda_{m-1}}$. In this case we use the same knot as in figure (\ref{fig:5}) but relabel the middle two strands $m-1,1$ instead of $1,1$. Then the composition we need to compute is 
$$\E_3 \E_1 \P_2^- \F_1 \F_3 \in \Hom_{\Kom^-_*(\K)}((\bk),(\bk))$$
where $\bk = (\0,0,m,0,m,\um)$ and $\P_2^- = \T_2^\infty$ is now the infinite complex
$$\left[ \dots \rightarrow \E_2 \F_2 \1_\i \la - m(2n-1)- 1 \ra \rightarrow \E_2 \F_2 \1_\i \la - m(2n-1) + 1 \ra \rightarrow \dots \rightarrow \E_2 \F_2 \1_\i \la - m + 1 \ra \rightarrow \1_\i \right].$$
Here $\i = ({\0},1,m-1,1,m-1,{\um})$ and, based on some calculations similar to the ones in the last section, the differentials alternate between 
\begin{equation}\label{eg6}
\sUupdot \sUdown - \sUup \sUdowndot \ \ \ \text{ and } \ \ \ \ \sum_{a+b+c=m-1} \sUupdot \mbox{\tiny{a}} \smcbub{+(b)} \sUdowndot \mbox{\tiny{c}}
\end{equation}
where $+(b)$ indicates that the bubble has the appropriate number of dots so that its degree is $b$. 

Now, we will show that $\E_3 \E_1 \E_2 \F_2 \F_1 \F_3 \1_\bk \cong \oplus_{[m]} \1_\bk$ by computing $\E_3 \E_1 \F_2 \E_2 \F_1 \F_3 \1_\bk$ in two ways. On the one hand, we have
\begin{align*}
\E_3 \E_1 \F_2 \E_2 \F_1 \F_3 \1_\bk 
&\cong \E_3 \E_1 \E_2 \F_2 \F_1 \F_3 \1_\bk \bigoplus_{[m-2]} \E_3 \E_1 \F_1 \F_3 \1_\bk 
\cong \E_3 \E_1 \F_2 \E_2 \F_1 \F_3 \1_\bk \bigoplus_{[m-2][m][m]} \1_\bk
\end{align*}
while, on the other hand, we have 
\begin{align*}
\E_3 \E_1 \F_2 \E_2 \F_1 \F_3 \1_\bk  \cong \F_2 \E_3 \E_1 \F_1 \F_3 \E_2 \1_\bk \cong \bigoplus_{[m-1][m-1]} \F_2 \E_2 \1_\bk \cong \bigoplus_{[m-1]^2[m]} \1_\bk.
\end{align*}
Since $[m-1]^2[m]-[m-2][m]^2 = [m]$ we get that $\E_3 \E_1 \E_2 \F_2 \F_1 \F_3 \1_\bk \cong \oplus_{[m]} \1_\bk$. Thus we end up with a complex 
$$\left[ \dots \xrightarrow{d_{2n}} \bigoplus_{[m]} \1_\bk \la - m(2n-1)- 1 \ra \xrightarrow{d_{2n-1}} \bigoplus_{[m]} \1_\bk \la - m(2n-1)+ 1 \ra \rightarrow \dots \xrightarrow{d_1} \bigoplus_{[m]} \1_\bk \la - m + 1 \ra \xrightarrow{d_0} \bigoplus_{[m][m]} \1_\bk \right].$$
where, once again, we need to figure out the differentials. 

\begin{Lemma}\label{lem:iso}
As we vary $0 \le a \le m-1$, the map depicted below the dashed line on the left in figure (\ref{fig:9}) induces an isomorphism 
$$\bigoplus_{[m]} \1_\bk \xrightarrow{\sim} \E_3 \E_1 \E_2 \F_2 \F_1 \F_3 \1_\bk.$$ 
\end{Lemma}
\begin{proof}
We will show that as you vary $0 \le b \le m-1$ the map above the dashed line in figure (\ref{fig:9}) gives $(-1)^{m-1}$ times the inverse map. To do this we evaluate the composition on the left in (\ref{fig:9}). 

\begin{figure}[ht]
\begin{center}
\puteps[0.4]{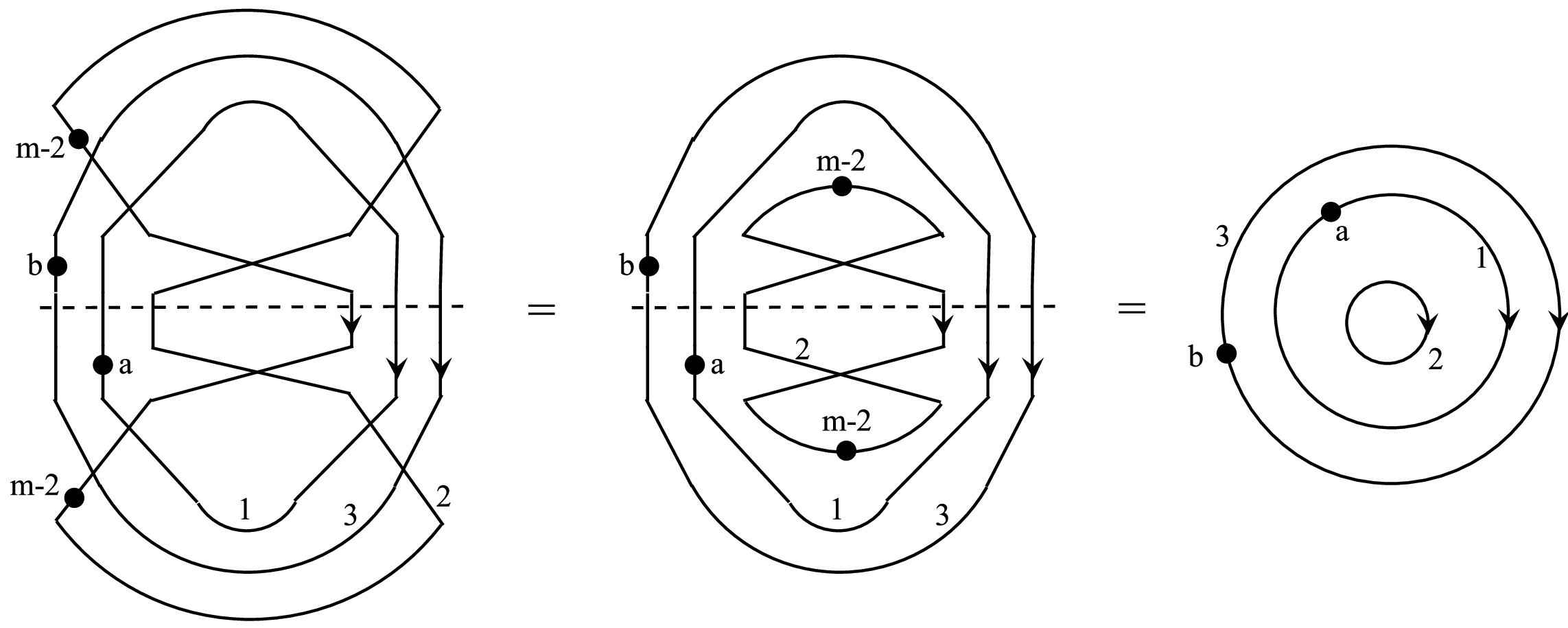}
\end{center}
\caption{}\label{fig:9}
\end{figure}

Since $\E_i$ and $\F_j$ commute when $i \ne j$, the composition in the left of (\ref{fig:9}) can be simplified to give the middle diagram. Using for instance \cite[5.16]{La}, the bottom (resp. top) curl in the middle diagram is equal to a cup (resp. cap) and we obtain the right hand diagram. Now, we can slide the inner bubble towards the outside, using \cite[Prop. 3.3,3.4]{KL3} as we did before, to obtain
$$\sum_{k=0}^{m-1-j} \sum_{j=0}^{m-1} (-1)^{k+j} \smcbub{+(m-1-j-k)} \hspace{-.8cm} \mbox{\tiny{2}} \hspace{.3cm} \smcbub{b+k} \hspace{-.2cm} \mbox{\tiny{3}} \smcbub{a+j} \hspace{-.2cm} \mbox{\tiny{1}} 
= \sum_{k=0}^{m-1-j} \sum_{j=0}^{m-1} (-1)^{k+j} \delta_{j+k,m-1} \delta_{b+k,m-1} \delta_{a+j,m-1}$$
which simplifies to give $(-1)^{m-1} \delta_{a+b,m-1}$. Thus the $m \times m$ matrix whose $(a,m-1-b)$ entry is given by the composition on the left in figure (\ref{fig:9}) equals $(-1)^{m-1}$ times the identity matrix (as claimed). 
\end{proof}

Now, if we sandwich $\sUup \sUup (\sUupdot \sUdown - \sUup \sUdowndot) \sUdown \sUdown$ where the dashed line is in figure (\ref{fig:9}) then we get zero. This means that $d_n=0$ if $n$ is odd. If $n>0$ is even then, for degree reasons, $d_n$ must induce zero between all but one summand $\1_\bk$. The differential on this summand is the sum over all $0 \le a,b,c \le m-1$ with $a+b+c=m-1$ of the composition on the left in figure (\ref{fig:10}). It remains to simplify this composition. 

\begin{figure}[ht]
\begin{center}
\puteps[0.4]{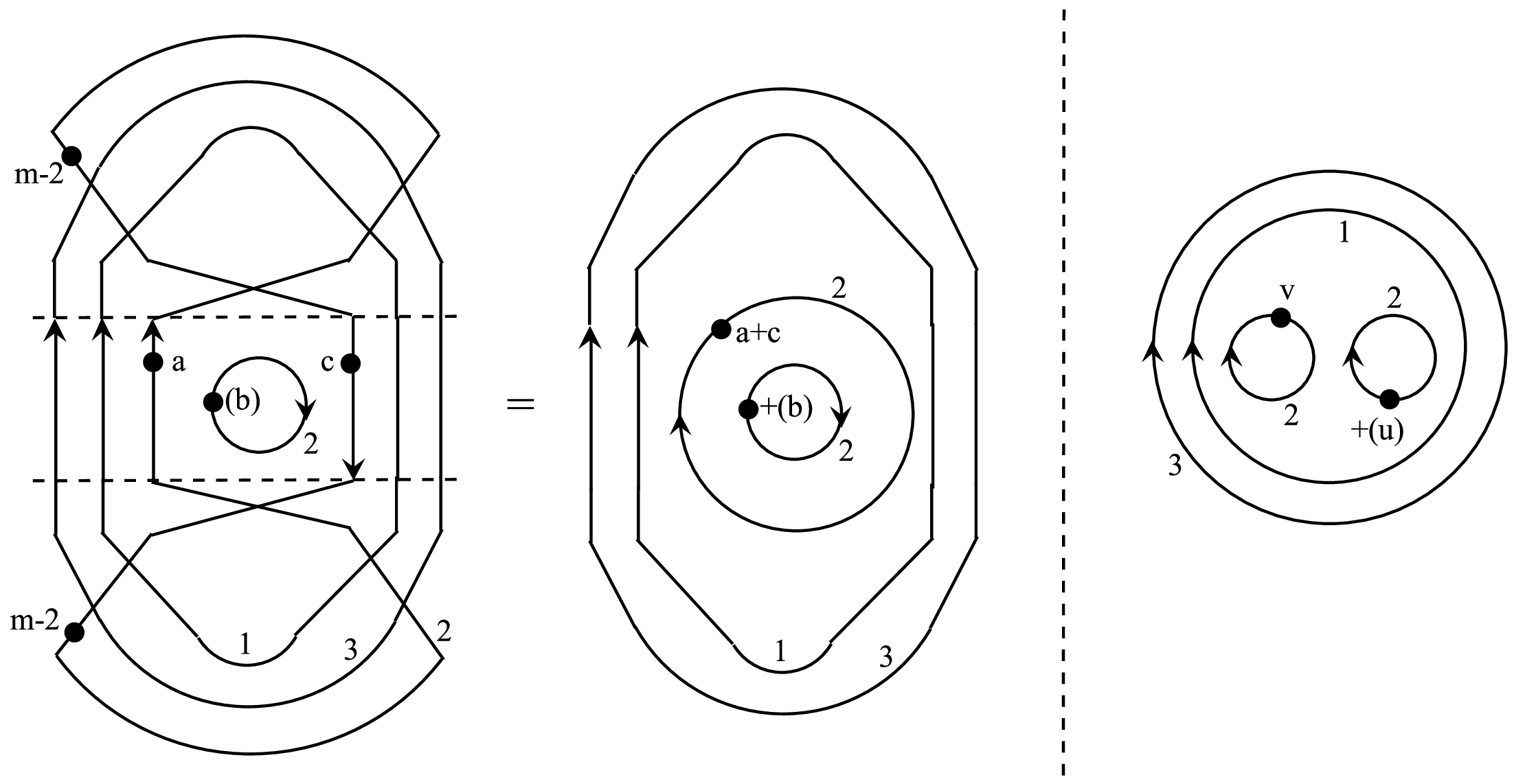}
\end{center}
\caption{}\label{fig:10}
\end{figure}

Now, the two middle concentric circles in figure (\ref{fig:10}) are equal to 
\begin{equation}\label{eq:15}
\smcbub{+(b)} \smcbub{a+c} - 2 \cdot \smcbub{+(b-1)} \smcbub{a+c+1} + \smcbub{+(b-2)} \smcbub{a+c+2}.
\end{equation}
At the same time, if $v+u=m-1$, then the composition on the right hand side of figure (\ref{fig:10}) can be simplified using the bubble silde relations to give $u+1$. Thus, if $b > 0$ then the sum in (\ref{eq:15}) contributes zero. If $b=0$ then (\ref{eq:15}) contributes $1$, and since we are summing over all $a+c=m-1$, we obtain $m$. 

In other words, the differential on the lone summand is multiplication by $m$. Finally, a similar argument shows that $d_0$ is injective. So, after cancelling out and replacing $\1_\bk$ with $\k$, we are left with 
$$\left[ \dots \xrightarrow{0} \bigoplus_{[m-1]} \k \la -5m \ra \xrightarrow{0} \bigoplus_{[m-1]} \k \la -3m \ra \xrightarrow{0} \bigoplus_{[m-1]} \k \la -3m \ra \xrightarrow{0} \bigoplus_{[m-1]} \k \la -m \ra \xrightarrow{0} 0 \rightarrow \bigoplus_{[m][m-1]} \k \la 1 \ra \right].$$
Thus, we arrive at the Poincar\'e polynomial
\begin{equation}\label{eq:Poincare2}
\chi_{q,t}(V_{\Lambda_1+\Lambda_{m-1}}) = [m-1] \left( [m-1]+ \frac{q^{-m} + t^3 q^{3m}}{1-t^2q^{2m}} \right).
\end{equation}
When $m=3$ this was also computed in \cite[Example 5.5]{Ros}). 

\begin{Remark}
The computation above is also valid over $\Z$. In other words, we find that over $\Z$ the rank of the homology is given by (\ref{eq:Poincare2}) while the torsion is equal to $\bigoplus_{d \ge 1} \Z/m\Z \la -2md \ra [2d]$. 
\end{Remark}

\section{Final remarks}\label{sec:remarks}

\subsection{The clasp $\P^+$ and the associated link homology $\sH_+^{i,j}(K)$}\label{sec:P+} 

In section \ref{sec:linkinv} we defined the link invariant $\sH^{i,j}_-(K)$ using the projectors $\P^- \in \Kom^-_*(\K)$ defined as $\lim_{\ell \rightarrow \infty} \T_\omega^{2 \ell}$. However, one can equally well define the projector
$$\P^+ := \lim_{\ell \rightarrow \infty} \T_\omega^{- 2 \ell} \in \Kom^+_*(\K)$$
where $\Kom_*^+$ is the analogue of $\Kom_*^-$ consisting of complexes that are bounded below (note that while $\P^-$ is defined as a direct limit, $\P^+$ is defined as an inverse limit). Using $\P^+$ we obtain another link invariant $\sH^{i,j}_+(K)$. These two invariants are related as follows. 

\begin{Proposition}\label{prop:Dmap}
Let $K$ be an oriented link and $K^!$ its mirror. Then $\sH^{i,j}_+(K) \cong \sH^{-i,-j}_-(K^!)$. 
\end{Proposition}
\begin{Remark}
If $K$ does not contain any clasps ({\it i.e.} its strands are labeled only by fundamental representations) then $\sH^{i,j}_+(K) \cong \sH^{i,j}_-(K)$ and there is only one homology, which we denote simply $\sH^{i,j}(K)$. Then Proposition \ref{prop:Dmap} implies that $\sH^{i,j}(K) \cong \sH^{-i,-j}(K^!)$. When $m=2$ ({\it i.e.} in the case of Khovanov homology) this fact was originally observed in \cite[Cor. 11]{K1}. 
\end{Remark}

To prove this result we work with $\K_{\Gr,m}$. This choice of 2-category has some extra structure worth discussing. Namely, each variety $Y(\i)$ is equipped with the Grothendieck-Verdier dualizing functor
$$\sA \mapsto \sA^\vee \otimes \omega_{Y(\i)} [\dim Y(\i)] \text{ for any } \sA \in D(Y(\i)).$$
This functor is contravariant and satisfies $\bD(\sA \la 1 \ra) = \bD(\sA) \la -1 \ra$ for any object $\sA \in \D(\l)$. Moreover, $\bD$ exchanges $\E_i$ and $\F_i$ in the sense that $\bD \E_i \1_\l \cong \F_i \bD \la \l_i+1 \ra$ which can be checked by a direct calculation of kernels.

One can extend $\bD$ to the whole 2-category by acting on objects as $\l \mapsto -\l$ and on 1-morphisms by taking their right adjoint. Notice that $\l \mapsto -\l$ corresponds to $\i = (\dots, i_{k-1},i_k,\dots) \mapsto (\dots, m-i_{k-1},m-i_k,\dots)$. Also, the fact that $\bD$ is contravariant on 1-morphisms means that $\bD$ exchanges $\Kom^-_*(\K_{\Gr,m})$ and $\Kom^+_*(\K_{\Gr,m})$. 
\begin{Remark}
In this paper we considered the $U_q(\sl_\infty)$-module $\Lambda_q^{m \infty}(\k^m \otimes \k^{2\infty})$ as a highest weight representation with highest weight $(\dots,0,0,m,m,\dots)$. However, $\bD$ exchanges highest and lowest weight modules so now we have a lowest weight representation with lowest weight $(\dots,m,m,0,0,\dots)$. This particular involution on categorified quantum groups, which exchanges highest and lowest weight 2-categories, has been considered before (see \cite[Remark 4.9]{Rou2}). 
\end{Remark}

We now check how $\bD$ acts on the associated link invariants. First, it flips cups and caps (up to a shift). This is because a cup and a cap were defined by maps 
$$\F_i^{(k)}: (\dots,0,m,\dots) \rightarrow (\dots,k,m-k,\dots) \text{ and } 
\E_i^{(k)}: (\dots,k,m-k,\dots) \rightarrow (\dots,0,m,\dots)$$
so that 
\begin{equation}\label{eq:shifts}
\bD \circ \E_i^{(k)} \cong \F_i^{(k)} \circ \bD \la m-k \ra \ \ \text{ and } \ \ \bD \circ \F_i^{(k)} \cong \E_i^{(k)} \circ \bD \la -m+k \ra.
\end{equation}
Second, it exchanges positive and negative crossings because $(\T_i)_R \cong \T_i^{-1}$. Subsequently, it also exchanges $\P^-$ and $\P^+$. 

Now, for a link $K$, the homology $\sH^{i,j}_-(K)$ is computed from $\Psi_-(K)$. Now consider $\bD \circ \Psi_-(K)$. Moving $\bD$ to the right has the effect of flipping $K$ about a horizontal line, interchanging positive and negative crossings and exchanging all projectors $\P^-$ with $\P^+$ (at least up to some overall grading shift). 

Now, flipping $K$ gives an equivalent link while exchanging over and under crossings amounts to replacing $K$ by its mirror $K^!$. Finally, in the end there is no overall grading shift since, in a link, caps and cups always come in pairs and so the shifts in (\ref{eq:shifts}) actually cancel out. Thus $\bD \circ \Psi_-(K) \cong \Psi_+(K^!) \circ \bD$ which completes the proof of Proposition \ref{prop:Dmap}.

\subsection{Triangulated structure}\label{sec:triangulated}

The 2-category $\K_{\Gr,m}$ has an extra triangulated structure. Namely, the categories $D(Y(\i))$ and the categories of kernels between them are triangulated. We have ignored this structure until now. 

Recall that the grading $\la 1 \ra$ we have used is equal to $[1]\{-1\}$ where $[\cdot]$ is the cohomological grading and $\{\cdot\}$ is the grading corresponding to the $\k^\times$ action on the varieties $Y(\i)$. We now consider the 2-category $\K^+_{\Gr,m}$ where we allow objects and kernels which are bounded below but not necessarily above (one can also consider the 2-category $\K^-_{\Gr,m}$ where objects and kernels are bounded above but not below). The 2-category $\K^+_{\Gr,m}$ should not be confused with $\Kom^+(\K_{\Gr,m})$. 

A complex $ (\A_\bullet, f_\bullet) = \A_n \xrightarrow{f_n} \A_{n-1} \rightarrow \cdots \xrightarrow{f_1} \A_0 $ of 1-morphisms is naturally an object in $\Kom(\K_{\Gr,m})$, but we can also try to take its convolution $C(\A_\bullet)$ (see \cite{GM} section IV, exercise 1). If $n=1$ then $C(\A_\bullet)$ is just the cone of $f_1: \A_1 \rightarrow \A_0$. If $n>1$ it is an iterated cone. 

In general a complex may not have a convolution or, if it exists, it might not be unique. However, there exist certain conditions (see for instance \cite[Prop. 8.3]{CK1}) which imply the existence and uniqueness of a convolution. In \cite{CKL1} we showed that the complex (\ref{eq:cpx1}) defining $\T_i \1_\l$ satisfies these conditions and hence has a unique convolution $C(\T_i) \1_\l \in \K_{\Gr,m}$. Notice that $C(\T_i)$ is now a 1-morphism in $\K_{\Gr,m}$ (not a complex of 1-morphisms). Thus $\T_\omega^2$ has a convolution $C(\T_\omega^2) \in \K_{\Gr,m}$ and the map $\1_\l \rightarrow \T_\omega^2 \1_\l$ induces a map $\1_\l \rightarrow C(\T_\omega^2 \1_\l)$. Subsequently we can then consider the limit $\lim_{n \rightarrow \infty} C(\T_\omega^2 \1_\l)^{n}$ and ask if it is well defined.

\begin{Lemma}\label{lem:convolution}
Suppose $\l$ is a weight such that $\T_i \1_\l = \left[\E_i \F_i \1_\l \la -1 \ra \rightarrow \1_\l \right]$. Then $\lim_{n \rightarrow \infty} C(\T_i)^{\pm 2n} \1_\l$ exists as a 1-morphism in $\K^\pm_{\Gr,m}$. 
\end{Lemma}
\begin{Remark}
Notice that although $\P^-$ (resp. $\P^+$) is a complex unbounded {\em below} (resp. {\em above}), its convolution turns out to be a kernel unbounded {\em above} (resp. {\em below}). 
\end{Remark}
\begin{proof}
The composition $\T_i^{2n} \1_\l$ is homotopic to a complex of the form 
\begin{equation}\label{eq:lemmaconv}
\left[ \E_i \F_i \1_\l \la -2n-1 \ra \rightarrow \E_i \F_i \1_\l \la -2n+1 \ra \rightarrow \dots \rightarrow \E_i \F_i \1_\l \la -1 \ra \rightarrow \1_\l \right]
\end{equation}
as described in equation (\ref{eq:Tinfty}). Thus $C(\T_i^{2n} \1_\l)$ is equal to some convolution of this complex. It is not hard to check that in $\K_{\Gr,m}$ the kernel $\E_i \F_i \1_\l$ is actually a sheaf. Since $\la 1 \ra = [1] \{-1\}$ this means that in $C(\T_i^{2n}) \1_\l$, the term $\E_i \F_i \1_\l \la -2k+1 \ra$ in homological degree $-k$ contributes a sheaf in homological degree $-(-2k+1)-k=k-1$. Taking the limit we find that $\lim_{n \rightarrow \infty} C(\T_i)^{2n} \1_\l$ lies in $\K^+_{\Gr,m}$. 
\end{proof}

\begin{Conjecture}\label{conj:convolution}
The limit $\lim_{n \rightarrow \infty} C(\T_i)^{\pm 2n} \1_\l$ exists as a 1-morphism in $\K^\pm_{\Gr,m}$. 
\end{Conjecture}

\subsection{Geometric construction of clasps}\label{sec:projgeom}
It is natural to wonder if the clasps $\P^-$ have a geometric description inside $\K_{\Gr,m}$. It seems the answer is ``yes'' when $m=2$ and ``perhaps not'' for $m > 2$. 

Suppose $m=2$ and $\i = ({\0},1,1,{\u2})$. Forgetting $L_1$ gives us a projection map $p: Y(\i) \rightarrow \overline{Y(\i)}$ where
$$\overline{Y(\i)} = \{ \k[z]^2 = L_0 \subset L_2 \subset \k(z)^2: zL_2 \subset L_2, z^2L_2 \subset L_0 \text{ and } \dim(L_2/L_0) = 2 \}.$$
Generically $p$ is one-to-one since $L_1$ can be recovered as $L_1 = zL_2$. Note that $Y(\i)$ in this case is a compactification of $T^\star \bP^1$. The restriction of $p$ to $T^\star \bP^1$ is just the affinization morphism. 

Consider the composition $p^* p_*: D^-(Y(\i)) \rightarrow D^-(Y(\i))$. Note that we have to work with unbounded below complexes since $\overline{Y(\i)}$ is singular.  

\begin{Proposition}\label{prop:pclasp}\cite[Prop. 6.8]{C1}. 
If $m=2$ and $\i = ({\0},1,1,{\u2})$ then $p^* p_*: D^-(Y(\i)) \rightarrow D^-(Y(\i))$ is induced by the kernel $C(\P^+) \in D^-(Y(\i) \times Y(\i))$.
\end{Proposition}

If you are interested in the geometry then this result gives a representation theoretic way to understand the somewhat complicated functor $p^* p_*$ in terms of the simpler functors $\E$ and $\F$ used to define $\P^+$. For instance, it allows you to compute the cohomology of the kernel which induces $p^* p_*$. 

More generally, if $m=2$ and $\i = ({\0}, 1^k, {\u2})$ then 
$$Y(\i) = \{\k[z]^2 = L_0 \subset L_1 \subset \dots \subset L_k \subset \k(z)^2: zL_{j+1} \subset L_j, \dim(L_{j+1}/L_j) = 1 \}$$
and again we have a projection $p: Y(\i) \rightarrow \overline{Y(\i)}$ which forgets $L_1, \dots, L_{k-1}$. 
\begin{Conjecture}\label{conj:clasp}
If $m=2$ and $\i = ({\0},1^k,{\u2})$ then $p^* p_*: D^-(Y(\i)) \rightarrow D^-(Y(\i))$ is induced by the kernel $C(\P^+) \in D^-(Y(\i) \times Y(\i))$.
\end{Conjecture}

Unfortunately, if $m > 2$ the functors $p^* p_*$ and $C(\P^+)$ (assuming the latter convolution exists) do not generally agree even at the level of K-theory. Thus, it remains an open question to give a representation theoretic interpretation of $p^* p_*$. 

\subsection{Clasps and category $\O$}

In \cite{FSS} the module $V^{\otimes n}$, where $V$ is the standard representation of $\sl_2$, is categoried by $\oplus_{k=0}^n A_{k,n}-{\rm gmod }$ for some graded algebras $A_{k,n}$ (this follows from \cite{FKS} where these categories correspond to a certain category $\O$). The Jones-Wenzl projectors are then categorified by explicit $(A_{k,n},A_{k,n})$-bimodules.

The simplest nontrivial example of their construction is when $k=1,n=2$. In this case $A_{1,2}$ is a 5-dimensional algebra which can be described as the path algebra of the quiver 
$$\xymatrix{ v \bullet \ar@/^/[r]^x & \bullet \ar@/^/[l]^y w } \ \ \text{ with relations } \ \ yx = 0. $$
If we denote by $e_w$ the constant path that starts and ends at $w$ then $e_w A_{1,2} e_w \cong \k[t]/t^2$ where $t=xy$. The projector is then defined by the $(A_{1,2},A_{1,2})$-bimodule $M_{1,2} := A_{1,2} e_w \otimes_{\k[t]/t^2} e_w A_{1,2}$ (the tensor is derived so this ends up being a complex of bimodules). 

On the other hand, \cite{CoK} also defines a complex of bimodules which categorifies this particular Jones-Wenzl projector. These two projectors are related by Koszul duality \cite{SS}. One can draw the following parallels between this work and the results in this paper. 
$$\begin{tabular}{|ccc|}
\hline
\noalign{\smallskip}
$A_{1,2}$ - gmod &
\xymatrix{ \ar@{<~>}[rr] & &} & \txt{coherent sheaves on $T^\star \bP^1$ \\ (or on its compactification $Y(1,1)$)} \\
\hline
\noalign{\smallskip}
projector induced by bimodule $M_{1,2}$ & \xymatrix{ \ar@{<~>}[rr] & & } & \txt{ functor  $p^*p_*$ where $p: T^\star \bP^1 \rightarrow \overline{T^\star \bP^1}$ \\ is the affinization map} \\
\hline
\noalign{\smallskip}
\txt{ relation proved in \cite{SS} between \\ $M_{1,2}$ and \cite{CoK} projector} & \xymatrix{ \ar@{<~>}[rr] & & } & \text{ Proposition \ref{prop:pclasp}} \\ 
\hline
\noalign{\smallskip}
\txt{ conjectural relation between \cite{FSS} \\ projectors and \cite{CoK} projectors } & \xymatrix{ \ar@{<~>}[rr] & & } & \text{ Conjecture \ref{conj:clasp}} \\
\hline
\end{tabular}$$
It would be interesting to generalize the category $\O$ approach of \cite{FSS,SS} from $\sl_2$ to $\sl_n$ invariants and to extend both sides of the table above.

\end{document}